\documentclass[11pt]{article}
\usepackage{amssymb,amsfonts,amsmath,amsthm}%
\usepackage{esint}%
\usepackage{graphicx}%
\usepackage{color}%
\usepackage{hyperref}%
\usepackage[font=small,format=plain,labelfont=bf,up]{caption}

\usepackage{hyperref}%
\definecolor{gray}{rgb}{0.1,0.1,.1}
\hypersetup{
    colorlinks=true,       
    linkcolor=gray,          
    citecolor=gray      
}
\newcommand{\figdraft}{false}%
\newcommand{\figfile}[1]{#1}%
\theoremstyle{plain}%
\newtheorem{theorem}{Theorem}[]%
\newtheorem{lemma}[theorem]{Lemma}%
\newtheorem{proposition}[theorem]{Proposition}%
\newtheorem*{result}{Main result}%
\newtheorem{definition}[theorem]{Definition}%
\newtheorem{assumption}[theorem]{Assumption}%
\newtheorem{remark}[theorem]{Remark}%
\newcommand{\tsnu}{t_*}

\def\longrightharpoonup{
\relbar\joinrel\joinrel\relbar\joinrel\joinrel\relbar\joinrel\joinrel\relbar\joinrel\joinrel\relbar\joinrel\joinrel\relbar\joinrel\joinrel\rightharpoonup}
\newcommand{\xrightharpoonup}[1]{\stackrel{#1}{\longrightharpoonup}}

\newcommand{\sgn}{\mathrm{sgn}}

\newcommand{\fspace}[1]{{\mathsf{#1}}}
\newcommand{\fspaceL}{\fspace{L}}

\newcommand{\fspaceC}{\fspace{C}}

\newcommand{\av}{\mathrm{av}}

\newcommand{\ol}[1]{{\overline{#1}}}

\newcommand{\Rset}{{\mathbb{R}}}

\newcommand{\Nset}{{\mathbb{N}}}

\newcommand{\ocinterval}[2]{\left(#1,\,#2\right]}%
\newcommand{\cointerval}[2]{\left[#1,\,#2\right)}%
\newcommand{\oointerval}[2]{\left(#1,\,#2\right)}%
\newcommand{\ccinterval}[2]{\left[#1,\,#2\right]}%

\newcommand{\bccinterval}[2]{\big[#1,\,#2\big]}%

\newcommand{\DO}[1]{{O\at{#1}}}
\newcommand{\Do}[1]{{o\at{#1}}}
\newcommand{\nDO}[1]{{O\nat{#1}}}

\newcommand{\threshold}{{\rm thres}}

\newcommand{\med}{{\rm med}}

\newcommand{\ini}{{\rm ini}}

\newcommand{\const}{{\rm const}}

\newcommand{\tdots}{{...}}%

\newlength{\mhpicDwidth}
\newlength{\mhpicDvsep}
\newlength{\mhpicDhsep}

\newlength{\mhpicPwidth}
\newlength{\mhpicPvsep}
\newlength{\mhpicPhsep}

\newlength{\mhpicWhsep}

\newcommand{\pair}[2]{{\left({#1},\,{#2}\right)}}

\newcommand{\bpair}[2]{{\big({#1},\,{#2}\big)}}

\newcommand{\at}[1]{{\left({#1}\right)}}
\newcommand{\nat}[1]{(#1)}
\newcommand{\bat}[1]{{\big(#1\big)}}
\newcommand{\Bat}[1]{{\Big(#1\Big)}}

\newcommand{\ato}[1]{{\left[{#1}\right]}}

\newcommand{\triple}[3]{{\left({#1},\,{#2},\,{#3}\right)}}
\newcommand{\btriple}[3]{{\big({#1},\,{#2},\,{#3}\big)}}

\newcommand{\ul}[1]{\underline{#1}}

\newcommand{\D}{\displaystyle}

\newcommand{\bigpar}{\par\quad\newline\noindent}

\newcommand{\norm}[1]{\|{#1}\|}
\newcommand{\abs}[1]{\left|{#1}\right|}
\newcommand{\babs}[1]{\big|{#1}\big|}
\newcommand{\Babs}[1]{\Big|{#1}\Big|}
\newcommand{\nabs}[1]{|{#1}|}

\newcommand{\dint}[1]{\,\mathrm{d}#1}

\newcommand{\Om}{{\Omega}}
\newcommand{\al}{{\alpha}}
\newcommand{\be}{{\beta}}
\newcommand{\ga}{{\gamma}}
\newcommand{\eps}{{\varepsilon}}

\newcommand{\si}{{\sigma}}

\newcommand{\calD}{\mathcal{D}}
\newcommand{\calE}{\mathcal{E}}

\newcommand{\calI}{\mathcal{I}}

\newcommand{\calL}{\mathcal{L}}

\newcommand{\calR}{\mathcal{R}}
\newcommand{\calS}{\mathcal{S}}

\newcommand{\calV}{\mathcal{V}}

\usepackage[%
a4paper,%
nohead,%
twoside,
ignorefoot,
scale=0.85,
bindingoffset=1.25cm
]{geometry}%
\usepackage{float}%
\begin{document}%
%
%
\title{Rate-independent dynamics and Kramers-type phase transitions \\
in nonlocal Fokker-Planck equations with dynamical control}%
\date{\today}%
\author{ %
Michael Herrmann\footnote{{\scriptsize{Institute for Computational and Applied Mathematics, University of M\"unster, {\tt michael.herrmann@uni-muenster.de}}}}%
\and
Barbara Niethammer\footnote{{\scriptsize{Institute for Applied Mathematics, University of Bonn, {\tt niethammer@iam.uni-bonn.de}}}}%
\and
Juan J.L. Vel\'azquez\footnote{{\scriptsize{Institute for Applied Mathematics, University of Bonn, {\tt velazquez@iam.uni-bonn.de}}}}%
}%
\maketitle%
%
%
%
\vspace{-5mm}%
\begin{abstract}%
The hysteretic behavior of many-particle systems with non-convex free energy 
can be modeled by nonlocal Fokker-Planck equations that involve two small parameters 
and are driven by a time-dependent constraint. In this paper we consider the fast reaction 
regime related to Kramers-type phase transitions and show that the dynamics in the 
small-parameter limit can be described by a rate-independent evolution equation with hysteresis. For the proof 
we first derive mass-dissipation estimates by means of  Muckenhoupt constants, 
formulate conditional stability estimates, and characterize 
the mass flux between the different phases in terms of moment estimates that encode
large deviation results. Afterwards we combine all these partial results and
establish the dynamical stability of localized peaks as well as sufficiently strong compactness 
results for the basic macroscopic quantities.
\end{abstract}%
%
%
\begin{minipage}[t]{0.15\textwidth}%
\emph{\small Keywords:} %
\end{minipage}%
\begin{minipage}[t]{0.8\textwidth}%
\small nonlocal Fokker-Planck equations, multi-scale dynamics of PDE and gradient flows, \\
mass-dissipation estimates, Kramers' formula in time-dependent potentials,\\
rate-independent models for hysteresis and phase transitions

\end{minipage}%
\medskip
\newline\noindent
\begin{minipage}[t]{0.15\textwidth}%
\emph{\small MSC (2000):} %
\end{minipage}%
\begin{minipage}[t]{0.8\textwidth}%
\small
35B40, 
35Q84, 
82C26, 
82C31 
\end{minipage}%
%
%
%
%
%
\renewcommand{\contentsname}{\small{Contents}}
\setcounter{tocdepth}{3} %
\setcounter{secnumdepth}{3}
{\scriptsize{\tableofcontents}}%
%

\section{Introduction}\label{sect:intro}
%

It is an ubiquitous and intriguing question in the mathematical analysis 
under which conditions the dynamics of a given
high-dimensional system with small parameters
can be described by low-dimensional, reduced evolution equations. 
In this paper we answer this question, at least partially, for a particular example, namely
the Fokker-Planck equation
\begin{align}
\label{Eqn:FP1}\tag{FP$_1$}
\tau\partial_t\varrho\pair{t}{x}=\partial_x\Bat{\nu^2\partial_x\varrho\pair{t}{x} + \bat{H^\prime\at{x}-\si\at{t}}\varrho\pair{t}{x}},
\end{align}
where $\tau$ and $\nu$ are the small parameters and
 $x\in\Rset$ is a one-dimensional state variable. Moreover, $H$ is supposed to be a double-well potential and $\si$ is 
a dynamical multiplier chosen such that the solution complies with
\begin{align}
\label{Eqn:FP2P}\tag{FP$_2$}
\int_\Rset x\varrho\pair{t}{x}\dint{x}=\ell\at{t},
\end{align}
where $\ell$ is a prescribed control function. This dynamical constraint
is, for admissible initial data, equivalent to the mean-field formula
\begin{align}
\label{Eqn:FP2}\tag{FP$_2^{\,\prime}$}
\si\at{t}=\int_\Rset H^\prime\at{x}\varrho\pair{t}{x}\dint{x}+\tau\dot\ell\at{t},
\end{align}
which turns \eqref{Eqn:FP1} into a nonlocal, nonlinear, and non-autonomous PDE.
\par
Nonlocal Fokker-Planck equations like \eqref{Eqn:FP1}+\eqref{Eqn:FP2P} have been introduced in \cite{DGH11a} in order to model the hysteretic behavior of many-particle storage systems such as modern Lithium-ion batteries (for the physical background, we also refer to
\cite{DJGHMG10}). In this context, $x\in\Rset$ describes 
the thermodynamic state of a single particle (nano-particle made of iron-phosphate in the battery case),
$H$ is the free energy of each particle, and $\nu$ accounts for entropic effects. Moreover, 
$\varrho$ is the probability density of a many-particle ensemble and
the dynamical control $\ell$ reflects that the whole system is driven by some external process
(charging or discharging of the battery).
\par
Since $H$ is non-convex, the dynamics of \eqref{Eqn:FP1}+\eqref{Eqn:FP2P} can be rather involved as they are
related 
to three different time scales, namely the small relaxation time $\tau$, the time 
scale of the control $\ell$ (which is supposed to be of order $1$),
and the Kramers time scale. The latter is given by
\begin{align}
\label{Eqn:KramersScale}
\tau\exp\at{\frac{\min\{h_-\at{\si},h_+\at\si\,\}}{\nu^2}}
\end{align}
and corresponds, as discussed below, to probabilistic transitions between the different wells of the
time-dependent effective potential with energy barriers $h_-\at\si$, $h_+\at\si$. 
\par
 In this paper we restrict our considerations to the fast reaction regime, in which the particular scaling relation between $\tau$ and $\nu$ guarantees that the time scale \eqref{Eqn:KramersScale} is of order $1$ for certain values of $\si$, and study the small-parameter limit  $\tau,\nu\to0$. Our main 
result is that  the microscopic PDE \eqref{Eqn:FP1}+\eqref{Eqn:FP2P} 
can be replaced by a low-dimensional dynamical system which
governs the evolution of the dynamical multiplier $\si$ and the phase fraction 
\begin{align*}
\mu\at{t}:=\int_{\text{right stable region}}\varrho\pair{t}{x}\dint{x}-\int_{\text{left stable region}}\varrho\pair{t}{x}\dint{x},
\end{align*}
where the stable regions (or `phases') are the connected components of $\{x\in\Rset\;:\;H^{\prime\prime}\at{x}>0\}$.
The micro-to-macro transition studied here has much in common with those in \cite{PT05,Mie11b,MT12}, which likewise
derive macroscopic models for the dynamics of phase transitions from 
microscopic gradient flows with non-convex energy and external driving. Our microscopic system, however, is different as it involves the diffusive term $\nu^2\partial_x^2\varrho$, which causes specific effects and necessitates the use
of different methods. 
\paragraph{Many-particle interpretation}%
It is well-known, see for instance \cite{Ris89}, that the linear Fokker-Planck equation \eqref{Eqn:FP1} with $\si\equiv0$
is equivalent to the Langevin equation $\tau\dint{}x=-H^\prime\at{x}\dint{t}+\sqrt{2}\nu\dint{W}$, where
$W$ denotes a standard Brownian motion in $\Rset$. In other words, 
$\varrho$ describes for $\si\equiv0$ the statistics of a large ensemble of identical particles, 
where each single particle is driven by the gradient flow of $H$ but also affected by stochastic fluctuations. If both $\tau$ and $\nu$ are small, the deterministic force is very strong and dominates the stochastic term. Most of the particles are therefore located near either one of the two local minima of $H$ and
random transitions between these wells are very unlikely. For the linear
Fokker-Planck equation this means that $\varrho$ quickly approaches 
a meta-stable state composed of two narrow peaks; the time scale 
of this fast relaxation process is $\tau$, the width of the peaks scales with $\nu$, and the 
mass distribution between the meta-stable peaks depends on the initial data for $\varrho$.
In the long run, however, the stochastic fluctuations imply that the mass distribution between the peaks
converges to its unique equilibrium value. Kramers investigated this problem 
in the context of chemical reactions in \cite{Kra40} and derived his seminal formula
for the effective mass flux between the two phases, see formula \eqref{Eqn:KramersFormula} below. For more details on Kramers' formula and the connection to the theory of large deviations we refer, 
for instance, to \cite{HTB90, Ber13}.  
\par
For the nonlocal Fokker-Planck equation \eqref{Eqn:FP1}, the dynamical constraint \eqref{Eqn:FP2P} 
augments the deterministic force by the nonlocal coupling term $\si$. The energy landscape for the 
single-particle evolution is therefore no longer given by $H$ but by the effective potential 
\begin{align}
\label{Eqn:EffectivePotential}
H_\si\at{x}:=H\at{x}-\si{x},
\end{align}
which is, depending on the value of $\si$, either a non-convex single-well or a double-well potential. 
The crucial point is that the number, the positions, and the energies of local minima
of $H_\si$ are time-dependent as $\si$ is non-constant. Moreover, we have no a priori information 
about the evolution of $H_\si$ because $\si$ is not given explicitly but only implicitly via \eqref{Eqn:FP2}. Any
asymptotic statement about the small-parameter dynamics of nonlocal Fokker-Planck equations therefore
requires a careful analysis of the  time scale on which $\si$ is changing. In the parameter 
regime studied below we are able to establish a substitute to $\abs{\dot{\si}}\leq\nDO{\dot\ell}=\DO{1}$, which then
implies, roughly speaking, that $\si$ evolves regularly and hence that there is always a clear but state-dependent relation between
the different time scales in the problem.
\paragraph{Parameter regimes}%
The different dynamical regimes for $0<\nu,\tau\ll1$
have been investigated by the authors in \cite{HNV12} using formal asymptotic analysis. According to these results,
there exist two main regimes, called \emph{slow reactions} and \emph{fast reactions}, as well as several sub-regimes
related to limiting or borderline cases, see Table \ref{Tbl:ScalingRegimes}. In each small-parameter regime, numerical simulations as well as heuristic arguments indicate that the probability density $\varrho\pair{t}{\cdot}$ can -- at least at most of the times $t$ -- 
be described by either one or two narrow and localized peaks. We therefore expect that
the small-parameter dynamics can always be characterized in terms of the positions and masses of 
these peaks. The key dynamical features, however, depend very much on the scaling relation 
between $\tau$ and $\nu$ and the temporal behavior of the macroscopic quantities $\si$ and $\mu$
is rather different for $0<\tau\ll\nu^{2/3}\ll1$ and $0<\nu^{2/3}\ll\tau\ll1$; see Figure \ref{Fig:NumSim} for an illustration and \cite{HNV12} for further numerical results.
\begin{table}[ht!]%
\centering{%
\begin{tabular}{llll}%
\emph{scaling law}&\emph{regime} &\emph{}%
\\\hline\\%
$\displaystyle\tau=\frac{a_\#}{\log{1}/{\nu}}$&slow reactions & single-peak limit for $a_\#\geq a_\threshold$
\medskip\\%
$\tau=\nu^{p_\#}$&borderline regime & either $a_\#\to0$ (for $p_\#<\tfrac23$) or $h_\#\to0$ (for $p_\#>\tfrac23$)
\medskip\\%
$\displaystyle\tau=\exp\at{-\frac{h_\#}{\nu^2}}$&fast reactions&  quasi-stationary limit for 
$h_\#\geq h_\threshold$
\end{tabular}%
}%
\caption{The different small-parameters regimes as described in \cite{HNV12}, where 
$a_\#$, $p_\#$, and $h_\#$ denote positive scaling parameters which are, at least to leading order, independent of $\nu$; the borderline regime can be viewed as the limiting case in which either $a_\#$ or $h_\#$ tends to $0$ as $\nu\to0$. 
}%
\label{Tbl:ScalingRegimes}%
\end{table}%
\par
In the fast reaction regime, each localized peak is always confined to one of the stable regions and only
a very small fraction of the mass is inside the spinodal interval $\{x\,:\,H^{\prime\prime}\at{x}<0\}$. Moreover, 
for most values of $\si$ the mass distribution between both peaks does not change and the dynamical constraint 
just alters the peak positions by adapting them to the instantaneous values of $\ell$. For certain values of $\si$, however, 
the peaks do exchange mass since the particles that pass through the spinodal region by stochastic fluctuations 
produce a non-negligible mass flux between the different phases. This 'leaking' or 'tunneling' process is governed by 
Kramers formula and described in \S\ref{sect:HeuristicDynamics} in more detail.
\par
The peak dynamics in the slow reaction regime is completely different for two reasons. First,
Kramers-type phase transitions are not feasible anymore since the corresponding time 
scale \eqref{Eqn:KramersScale} is, for all relevant values of $\si$, much larger than $1$.
The second important observation is that the dynamical constraint \eqref{Eqn:FP2P} can drive a localized peak 
with positive mass into the spinodal region. The emerging unstable peak remains narrow for a while but 
splits suddenly into two stable peaks due to a subtle interplay between the parabolic and the kinetic terms 
of \eqref{Eqn:FP1}; see \cite{HNV12} for more details.
\par
The macroscopic behavior in the borderline regime can be regarded as the limiting case of both the fast and the slow reaction regime; see also the respective comments in the caption of Figure \ref{Fig:NumSim}. More precisely, for algebraic scaling relations $\tau\sim\nu^{p_\#}$
with $0<p_\#<2/3$, the localized peaks can still penetrate the spinodal interval but they split into stable peaks after a short time
depending on $\nu$. For $p_\#>2/3$, however, the Kramers mechanism still produces macroscopic phase transitions but only when the minimal energy barrier of $H_\si$ is algebraically small in $\nu$. Finally, 
phase transitions happen in the quasi-stationary limit whenever the two wells of $H_\si$ have the same energy, that means when $\si$ attains a certain value which depends on $H$ but not on $h_\#$.
\par
\begin{figure}[ht!]%
\centering{%
\includegraphics[height=.20\textwidth, draft=\figdraft]{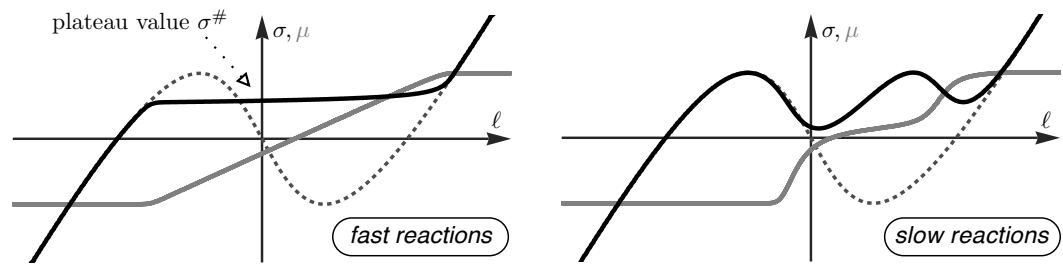}%
}%
\caption{Two typical numerical solutions for strictly increasing $\ell$ and an 
even double-well potential $H$, see \cite{HNV12} for more details. The black and gray 
solid lines represent the curves  $t\mapsto\bpair{\ell\at{t}}{\si\at{t}}$ and 
$t\mapsto\bpair{\ell\at{t}}{\mu\at{t}}$, respectively; the dashed line is the graph of the 
bistable function $H^\prime$. \emph{Left.} In the fast reaction regime, there exists a time interval 
in which $\si$ is basically constant with value $\si^\#$, where $\si^\#$ is determined by the scaling parameter $h_\#$. During this `plateau dynamics', a 
Kramers-type phase transition transfers mass form the left stable region to the right stable 
one. \emph{Right.} In the slow reaction regime, the 
function $\si$ is oscillatory since the phase transition is now driven by different phenomena, namely the 
formation and eventual splitting of unstable peaks (which requires a different definition of the phase 
fraction $\mu$). \emph{Remark on limiting cases.} The borderline regime $\tau=\nu^{p_\#}$ corresponds
to the maximal plateau value $\si^\#=\si^*$, where $\si^*$ is the local maximum of $H^\prime$. For 
$p_\#<\tfrac23$, however, this maximal plateau is superimposed by many oscillations with very small amplitudes. In the 
quasi-stationary limit $h_\#\geq h_\threshold$ we find the minimal plateau value $\si^\#=0$.
}%
\label{Fig:NumSim}
\end{figure}%
\paragraph{Limit dynamics}
In what follows we solely consider the fast reaction regime, that means we suppose that $\tau$ and $\mu$ are coupled as in the bottom row of  Table \ref{Tbl:ScalingRegimes} by some positive scaling parameter $h_\#$, which is basically independent of $\nu$ but not too large (see Assumption \ref{Ass:Tau} below). 
The key dynamical observation is that Kramers-type phase transitions
can manifest on the macroscopic scale only if the dynamical multiplier $\si$ attains one of two 
critical values $\si_\#$ and $\si^\#$, which depend on $H$ and $h_\#$, because otherwise the corresponding microscopic mass flux between the two phases is either too small or too large. The limit dynamics for $\nu\to0$ is therefore completely characterized by the flow rule
\begin{align*}
\dot\mu\at{t}\leq0\quad\text{for}\quad\si\at{t}=\si_\#,
\qquad
\dot\mu\at{t}\geq0\quad\text{for}\quad\si\at{t}=\si^\#,
\qquad
\dot\mu\at{t}=0\quad\text{otherwise},
\end{align*}
and pointwise relations $\ell\at{t}=\calL\bpair{\si\at{t}}{\mu\at{t}}$ 
that encode the dynamical constraint \eqref{Eqn:FP2P}. These findings can be summarized as follows.
\begin{result}
Under natural assumptions on $H$, the control $\ell$,  and the initial data, the triple
$\triple{\ell}{\si}{\mu}$ satisfies in the limit $\nu\to0$ of the fast reaction regime a closed rate-independent 
evolution equation with hysteresis. Moreover, the limit solution is unique provided that
the initial data are well-prepared.
\end{result}
We expect
that our methods and results 
can be generalized to both the borderline regime and the quasi-stationary limit for the price of more notational and analytical effort; see the discussion at the end of \S\ref{sect:limitmodel}. The small-parameter dynamics of the
slow reaction regime, however, has to be studied in a completely different framework; it is currently unclear, at least to the authors, how the formal results from \cite{HNV12} can be justified rigorously. 
\bigpar
The rest of the paper is organized as follows. In \S\ref{sect:prelim} we give a more detailed introduction to the problem.
More  precisely, in \S\ref{sect:potential} and \S\ref{sect:solutions} we specify our assumptions and review the existence theory for \eqref{Eqn:FP1}+\eqref{Eqn:FP2} with arbitrary $\nu,\tau>0$ as it is developed in Appendix \ref{app:ExistenceAndUniqueness}. In \S\ref{sect:HeuristicDynamics} we then
explain heuristically the key dynamical features in the fast reaction regime and proceed with a precise formulation of the limit model in \S\ref{sect:limitmodel}.  The main analytical work is done in \S\ref{sect:AuxResults} and \S\ref{sect:limit} but we postpose the discussion of the underlying ideas to \S\ref{sect:overview}. 
%
%
%
\section{Preliminaries}\label{sect:prelim}
%
%
In this section we introduce our assumption on $H$, $\ell$, and the initial data,  
and summarize some important properties of solutions to the non-local Fokker-Plank equation.
Moreover, we discuss the dynamics in the fast reaction regime on a heuristic level
and formulate the rate-independent limit model. 
%

\subsection{Assumptions on the potential, the parameter, and the data}\label{sect:potential}
Throughout this paper we assume that $H$ is a double-well potential with the following properties, see 
Figure \ref{Fig:IntroPotential} for an illustration.

\begin{figure}[ht!]%
\centering{%
\includegraphics[height=.275\textwidth, draft=\figdraft]{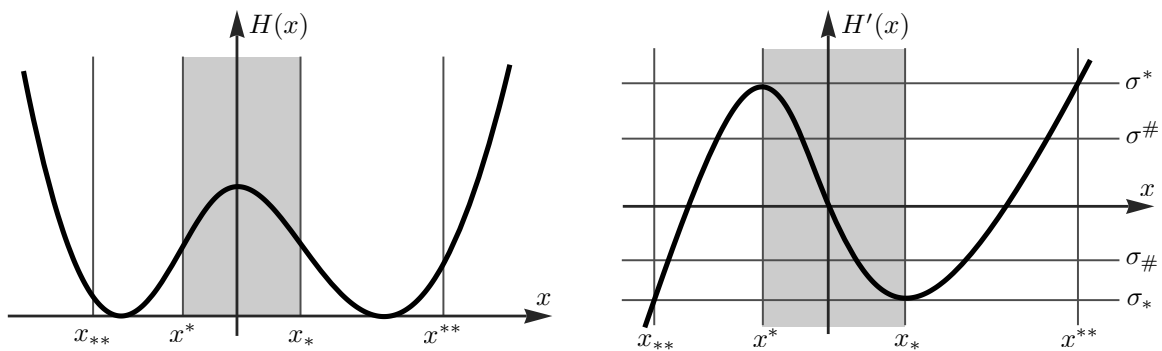}%
}%
\caption{Cartoon of a double-well potential $H$ that satisfies Assumption \ref{Ass:Potential} with $\si_\#$ and $\si^\#$ as in
Assumption \ref{Ass:Tau}. The shaded regions illustrate the spinodal (or unstable) interval $\oointerval{x^*}{x_*}$.%
}%
\label{Fig:IntroPotential}
\end{figure}%

\begin{assumption}[properties of $H$]
\label{Ass:Potential}\quad
\begin{enumerate}
\item  $H$ is three times continuously differentiable, attains a local maximum at $x=0$ and
its global minimum at precisely two points. 
\item $H^{\prime\prime}$ has only two zeros $x^*$, $x_*$ with
 $x^*<0<x_*$ such that \begin{enumerate}
\item $H^{\prime\prime}\at{x}<0$ for any $x$ with $x^*<x<x_*$, and
\item $H^{\prime\prime}\at{x}>0$ for all $x<x^*$ and all $x>x_*$.
\end{enumerate}
We set $\si^*:=H^\prime\at{x^*}$, $\si_*:=H^\prime\at{x_*}$ 
and the properties of $H^{\prime\prime}$ imply $\si_*<0<\si^*$.
\item $H^\prime$ is asymptotically linear in the sense of
$\lim_{x\to\pm\infty}H^{\prime\prime\prime}\at{x}=0$ and 
$\lim_{x\to\pm\infty}H^{\prime\prime}\at{x}=c_\pm$ for some constants $c_\pm$. In particular, there exist unique numbers $x^{**}$ and $x_{**}$
such that 
\begin{align*}
H^\prime\at{x_{**}}=H^\prime\at{x_*},\qquad 
H^\prime\at{x^{**}}=H^\prime\at{x^*}
\end{align*}
as well as $x_{**}<x^*$, $x^{**}>x_*$.
\end{enumerate}
\end{assumption}
The assumption that the two wells of $H$ are global minima is not crucial and can always
be guaranteed by means of elementary transformations. In fact, \eqref{Eqn:FP1} and \eqref{Eqn:FP2} are, for any given $c\in\Rset$, invariant under $H\rightsquigarrow H+c{x}$, $\si\rightsquigarrow \si+c$. Moreover, 
by an appropriate shift in $x$ we can always ensure that the local maximum is attained at $x=0$. The assumption that $H$ grows quadratically at infinity is of course more restrictive and excludes, for instance, polynomial double-well potentials.
We made this assumption in order to simplify some technical arguments, especially in the context of moment estimates, as it allows us control the tail contributions of $\varrho$ quite easily. However, we expect that our convergence result is also true for more general double-well potentials $H$ provided that
those grow super-quadratically or that the initial data
decay sufficiently fast.
\bigpar
As a direct consequence of Assumption \ref{Ass:Potential} we can define three functions $X_-$, $X_0$, and $X_+$ 
such that $H^\prime\circ X_j=\mathrm{id}$.
\begin{remark}[functions $X_-$, $X_0$, and $X_+$]
\label{Rem:PropertiesOfXi}
The inverse of $H^\prime$ has three strictly monotone and smooth branches
\begin{align*}
X_-:\ocinterval{-\infty}{\si^*}\to\ocinterval{-\infty}{x^*},\quad
X_0:\ccinterval{\si_*}{\si^*}\to\ccinterval{x^*}{x_*},\quad
X_+:\cointerval{\si_*}{+\infty}\to\cointerval{x_*}{+\infty}.
\end{align*}
In particular, we have
\begin{enumerate}
\item  $X_+\at{\si}-X_-\at{\si}\geq{c}$ for all $\si\in\ccinterval{\si_*}{\si^*}$,
\item $c\leq X_\pm^\prime\at{\si}\leq{C_\eps}$ for all $\si\in\ccinterval{\si_*+\eps}{\si^*-\eps}$,
\item $\abs{\si_2-\si_1}\leq {C} \babs{X_\pm\at{\si_2}-X_\pm\at{\si_1}}$ for
all $\si_1,\si_2$ in the domain of $X_\pm$,
\end{enumerate}
for any $\eps$ with $0<\eps<\tfrac{1}{2}\at{\si^*-\si_*}$ and some constants $c$, $C$ and $C_\eps$.
\end{remark}
We also observe that the effective potential $H_\si$ from \eqref{Eqn:EffectivePotential} 
is a genuine double-well potential for $\si\in\oointerval{\si_*}{\si^*}$, where the two energy barriers are given by
\begin{align}
\label{Eqn:DefEnergyBarriers}
h_\pm\at{\si}:=H_\si\at{X_0\at\si}-H_\si\at{X_\pm\at\si}
\end{align}
and feature prominently in Kramers' formula for the mass flux between the two phases; see Figure \ref{Fig:EffectivePotential} and discussion in \S\ref{sect:HeuristicDynamics}. For $\si<\si_*$ and $\si>\si^*$, however, 
$H_\si$ has only a single well located at $X_-\at{\si}$ and $X_+\at\si$, respectively.
\begin{remark}[properties of $h_\pm$]
The functions $h_-$ and $h_+$ are well-defined and smooth on the interval
$\ccinterval{\si_*}{\si^*}$ with $h_-\at{0}=h_+\at{0}>0$. Moreover, $h_-$ is strictly decreasing with $h_-\at{\si^*}=0$ and
$h_+$ is strictly increasing with $h_+\at{\si_*}=0$.
\end{remark}

\begin{figure}[ht!]%
\centering{%
\includegraphics[height=.275\textwidth, draft=\figdraft]{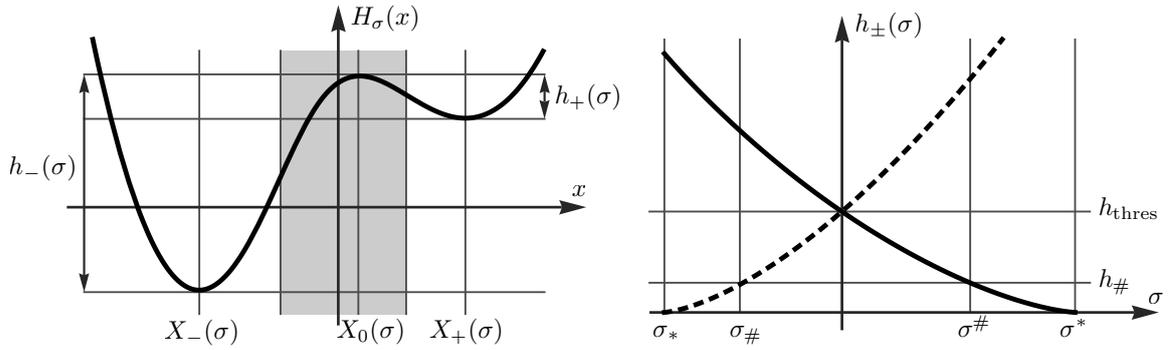}%
}%
\caption{\emph{Left panel}. 
Cartoon of the effective potential $H_\si$ for  $\si_*<\si<0$. For $0<\si<\si^*$, the barrier $h_+$ exceeds $h_-$,
while $\si<\si_*$ or $\si>\si^*$ implies that $H_\si$ is a nonconvex single-well potential.   
\emph{Right panel.} The  energy barriers  $h_-$ (solid line) and $h_+$ (dashed line) as functions of $\si$. The values $\si_\#$ and $\si^\#$ are defined by $h_-\at{\si^\#}=h_+\at{\si_\#}=h_\#$ with $h_\#=-\lim_{\nu\to0}\nu^2\ln\tau$.
}%
\label{Fig:EffectivePotential}
\end{figure}%

We finally describe the coupling between $\tau$ and $\nu$ and introduce the values $\si_\#$ and $\si^\#$.
\begin{assumption}[coupling between $\tau$ and $\nu$]
\label{Ass:Tau}
The parameter $\tau$ is positive, depends on $\nu$, and
satisfies 
\begin{align*}
\nu^2\ln\tau\quad\xrightarrow{\;\;\nu\to0\;\;}\quad - h_\#
\end{align*}
for some $h_\#$ with $0<h_\#<h_\threshold:=h_\pm\at{0}$. In particular, there exist $\si_\#$ and $\si^\#$ such that
\begin{align*}
\si_*<\si_\#<0<\si^\#<\si^*,\qquad h_\#=h_-\nat{\si^\#}=h_+\at{\si_\#},
\end{align*}
and hence $h_\#<h_{\threshold}<\min\{h_-\at{\si_\#},h_+\nat{\si^\#}\}$.
\end{assumption}

\subsection{Existence and properties of solutions }\label{sect:solutions}

It is well established, see \cite{JKO97,JKO98}, that the linear Fokker-Planck equation without dynamical constraint -- that is
\eqref{Eqn:FP1} with $\si\at{t}\equiv0$ --  is the Wasserstein gradient flow of the energy
\begin{align}
\label{Eqn:DefEnergy}
\calE\at{t}&:=
\nu^2\int_\Rset\varrho\pair{t}{x}\ln\varrho\pair{t}{x}\dint{x}+
\int_\Rset H\at{x}\varrho\pair{t}{x}\dint{x}.
\end{align}
Similarly, the non-driven variant of the nonlocal Fokker-Planck equations -- 
that is \eqref{Eqn:FP1}+\eqref{Eqn:FP2} with $\dot\ell\at{t}\equiv0$ -- 
can be regarded as the Wasserstein gradient flow of $\calE$ restricted to the constraint manifold 
$\int_\Rset x\varrho\dint{x}=\ell$. In the general case $\dot{\ell}\neq0$, however, the energy $\calE$ is no longer strictly
decreasing as \eqref{Eqn:FP1} and \eqref{Eqn:FP2} imply the energy law
\begin{align}
\notag%
\tau\dot{\calE}\at{t}
&=\int_\Rset
\tau\partial_t\varrho\pair{t}{x}\Bat{\nu^2\ln \varrho\pair{t}{x} +\nu^2 +H\at{x}}\dint{x}
\\\notag
&=-\int_\Rset
\Bat{\nu^2 \partial_x \varrho\pair{t}{x} + \bat{H^\prime\at{x}-\si\at{t}}\varrho\pair{t}{x}}\at{\nu^2\frac{\partial_x \varrho\pair{t}{x}}{\varrho\pair{t}{x}}+H^\prime\at{x}}
\dint{x}
\\\notag%
&=-\calD\at{t}-\si\at{t}\int_\Rset 
\bat{H^\prime\at{x}-\si\at{t}}\varrho\pair{t}{x}\dint{x}
\\\label{Eqn:EnergyBalance}%
&=-\calD\at{t}+\tau\si\at{t}\dot\ell\at{t},
\end{align}
where the dissipation is given by
\begin{align}
\label{Eqn:DefDissipation}
\calD\at{t}&:=\int_\Rset
\frac{\Bat{\nu^2\partial_x\varrho\pair{t}{x}+\bat{H^\prime\at{x}-\si\at{t}}\varrho\pair{t}{x}}^2}{\varrho\pair{t}{x}}\dint{x}.
\end{align}
In particular, the inequality $\dint{E}\leq  \sigma\dint\ell$ holds along each trajectory and can be viewed as the Second Law of Thermodynamics, evaluated for the free energy of the many-particle ensemble in the presence of the dynamical 
control \eqref{Eqn:FP2P}; see \cite{DGH11a} for the physical interpretation of \eqref{Eqn:DefEnergy} and \eqref{Eqn:EnergyBalance}.  
Moreover, it has been shown in \cite{DHMRW14} that the nonlocal
equations \eqref{Eqn:FP1}+\eqref{Eqn:FP2P}
can in fact be interpreted as a \emph{constraint gradient system} with proper Lagrangian multiplier $\si$.
\par
The energy-dissipation relation \eqref{Eqn:EnergyBalance} is essential for passing to the limit $\nu\to0$ since it reveals that the dissipation $\calD$ is very small with respect to the $\fspaceL^1$-norm and hence,
loosely speaking, also small at most of the times. For linear Fokker-Planck equations without constraint,
the underlying gradient structure can be used to establish $\Gamma$-convergence 
as $\tau\to0$. The resulting evolution equation is a one-dimensional reaction ODE for the phase fraction $\mu$ and 
equivalent to Kramers' celebrated formula, see \cite{PSV10,AMPSV12,HN11}. However, 
it is not clear to us whether this variational approach can be adapted to the present case with dynamical constraint;
the methods developed here employ the estimate for $\calD$ but make no further use of the gradient flow interpretation
of \eqref{Eqn:FP1}+\eqref{Eqn:FP2P}.
\bigpar
Since the system \eqref{Eqn:FP1}+\eqref{Eqn:FP2} is a nonlinear and nonlocal PDE,
it is not clear a priori that the initial value problem is well-posed in an appropriate function space.
In the case of a bounded spatial domain and Neumann boundary conditions, the existence and uniqueness of solutions
has been established in \cite{Hut12,DHMRW14} using an $\fspaceL^q$-setting for $\varrho$ with $q>1$. 
Since here we are interested in solutions that are defined on the whole real axis, we sketch an alternative existence
and uniqueness proof in Appendix \ref{app:ExistenceAndUniqueness}. The key idea there
is to obtain solutions as unique fixed points of a rather natural iteration scheme on the state space of all
probability measures with bounded variance. Moreover, adapting standard techniques for parabolic PDE 
we derive in Appendix \ref{app:ExistenceAndUniqueness} several bounds to reveal how these solutions depend on $\nu$. We finally mention that well-posedness with $\varrho\pair{t}{\cdot}\in\fspaceL^1\at\Rset$ has recently be shown in \cite{Ebe13} using a gradient flow approach. 
\par
Our assumptions and key findings concerning the existence and regularity of solutions 
to the nonlocal Fokker-Planck equation \eqref{Eqn:FP1}+\eqref{Eqn:FP2} can be summarized as follows.

\begin{assumption}[dynamical control $\ell$]
\label{Ass:Constraint}
The final time $T$ with $0<T<\infty$ is independent of $\nu$. The control $\ell$
is also independent  of $\nu$ and twice continuously differentiable on $\ccinterval{0}{T}$.
In particular, we have
\begin{align*}
\sup\limits_{t\in\ccinterval{0}{T}}\at{\babs{\ell\at{t}}+\babs{\dot{\ell}\at{t}}+\babs{\ddot{\ell}\at{t}}}\leq{C}
\end{align*}
for some constant $C$ independent of $\nu$.
\end{assumption}
\begin{assumption}[initial data]
\label{Ass:RegularityOfInitialData}
The initial data are nonnegative and satisfy
\begin{align*}
\int_\Rset\varrho\pair{0}{x} \dint{x}=1,\qquad
\int_\Rset x \varrho\pair{0}{x} \dint{x}=\ell\at{0},\qquad
\int_\Rset x^2  \varrho\pair{0}{x} \dint{x}\leq{C}
\end{align*} 
for some constant $C$ independent of $\nu$.
\end{assumption} 
 
\begin{lemma}[existence and properties of solution]
\label{Lem:PropertiesOfSolutions}
For any $\nu$ with $0<\nu\leq1$ and given initial data there exists a unique solution $\varrho$
to the initial value problem \eqref{Eqn:FP1}+\eqref{Eqn:FP2} which is nonnegative and smooth for $t>0$,
and satisfies
\begin{align*}
\int_\Rset\varrho\pair{t}{x} \dint{x}=1,\qquad
\int_\Rset x \varrho\pair{t}{x} \dint{x}=\ell\at{t}
\end{align*}
for all $t\in\ccinterval{0}{T}$. Moreover, each solution satisfies
\begin{align*}
\sup_{t\in\ccinterval{0}{T}}\at{\babs{\si\at{t}}+\int_\Rset x^2 \varrho\pair{t}{x} \dint{x}}+
\sup_{t\in\ccinterval{\tsnu}{T}}\nu^2\,\norm{\varrho\pair{t}{\cdot}}_\infty+\tau^{-1}\int_{\tsnu}^{T}
\calD\at{t}\dint{t}\leq{C}
\end{align*}
with $\tsnu:=\nu^2\tau$ for some constant $C$ which depends only on $H$, $\ell$ and $
\int_\Rset x^2  \varrho\pair{0}{x} \dint{x}$. 
\end{lemma}
\begin{proof}
All claims follow from 
Proposition \ref{Prop:ExistenceAndUniqueness} and Proposition \ref{Prop:UniformBounds}
in Appendix \ref{app:ExistenceAndUniqueness}.
\end{proof}
The assertions of Lemma \ref{Lem:PropertiesOfSolutions} 
reflect the existence of an initial transient regime. At first we have to wait for
the time $\tsnu$ before we can guarantee that
$\norm{\varrho\pair{t}{\cdot}}_\infty\leq C/\nu^2$ and $\int_{\tsnu}^T\calD\at{t}\dint{t}\leq C\tau$. The first estimate
is needed within \S\ref{sect:AuxResults} in order to show that no mass can penetrate the spinodal region from outside,
and that there is no mass flux through the spinodal region for subcritical $\si\in\oointerval{\si_\#}{\si^\#}$.
Furthermore, it is in general not before a
time of order $\tau^{1-\beta}$ that the dissipation $\calD\at{t}$ 
is eventually smaller than $\tau^\beta$ (the exponent $0<\beta<1$ will be identified below).
In \S\ref{sect:limit} we prove that the solutions to the nonlocal Fokker-Planck equation
behave nicely after the second time, even though we are not able to exclude that 
$\calD\at{t}$ becomes large (again) at some later time. 
\par
The initial transient regime corresponds to very fast relaxation processes during which the 
system dissipates a large amount of energy leading to rapid changes of especially the multiplier $\si$ and the phase fraction $\mu$. For generic initial data, we therefore expect to find
several limit solutions as $\nu\to0$ depending on the microscopic details of the initial data.
The only possibility to avoid such non-uniqueness is to start with well-prepared initial data.
\begin{definition}[well-prepared initial data]
\label{Def:WellPreparedData}
The initial data from Assumption \ref{Ass:RegularityOfInitialData} are \emph{well-prepared}, if 
they additionally satisfy
\begin{align*}
\nu^2\norm{\varrho\pair{0}{\cdot}}_\infty+\tau^{-1}\calD\at{0}\leq {C}
\end{align*}
for some constant $C$ independent of $\nu$, and if we have
\begin{align*}
\si\at{0}\quad\xrightarrow{\;\;\nu\to0\;\;}\quad\si_\ini
\end{align*}
for some $\si_\ini\in\Rset$.
\end{definition}
\begin{remark}
\label{Rem:WellPreparedInitialData} 
For well prepared initial data we can choose $\tsnu=0$ in Lemma \ref{Lem:PropertiesOfSolutions}. Moreover, we have
\begin{align*}
\varrho\pair{0}{x}\quad\xrightarrow{\;\;\nu\to0\;\;}\quad
\varrho_\ini:=\frac{1-\mu_\ini}{2}\delta_{X_-\at{\si_\ini}}\at{x}+\frac{1+\mu_\ini}{2}\delta_{X_+\at{\si_\ini}}\at{x}
\end{align*}
in the sense of weak$\star$ convergence of measures, where 
$\delta_X$ denotes the Dirac distribution at $X\in\Rset$ and
$\mu_\ini:=\int_{-\infty}^{x^*} \varrho_\ini\at{x}\dint{x}-\int_{x_*}^{+\infty} \varrho_\ini\at{x}\dint{x}$.
\end{remark}
\begin{proof}
The assertions follow from Lemma  \ref{Lem:ImpovedInitialData} and
the mass dissipation estimates formulated in
Lemma \ref{Lem:MD.MassOutsideTwoPeaks} and
Lemma \ref{Lem:MD.MassOutsideOnePeak}.
\end{proof}

\subsection{Heuristic description of the fast reaction regime}
\label{sect:HeuristicDynamics}

In order to highlight the key ideas for our convergence proof, we now
give an informal overview on the effective dynamics for $\nu\ll1$. For numerical simulations
as well as formal asymptotic analysis we refer again to \cite{HNV12}. 
\par
As explained above, the underlying gradient structure ensures
that the system approaches -- after a short initial transient regime with large dissipation --
at time $0<t_0\ll1$ a state with small dissipation. Assuming that $\calD\at{t}$ remains small 
for all times $t\geq t_0$, we can characterize the probability measures $\varrho\pair{t}{\cdot}$ for $\nu\ll1$ as follows. 
%
%
\paragraph*{Formation of peaks} 
%
%
\begin{figure}[ht!]%
\centering{%
\includegraphics[height=.3\textwidth, draft=\figdraft]{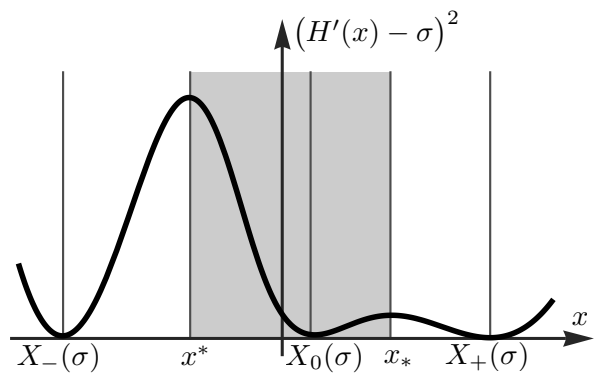}%
}%
\caption{Moment weight for the definition of $\xi$ with $\si_*<\si<0$. For $\si\in\oointerval{\si_*}{\si^*}$ and
$\xi\ll1$, almost all of the total mass is concentrated in three narrow peaks located at $X_-\at\si$, $X_0\at\si$, and $X_+\at\si$, but only the peaks at $X_\pm\at{\si}$ are dynamically stable. 
For $\si<\si_*$ and $\si>\si_*$, the mass is concentrated for $\xi\ll1$  in a single stable peak at $X_-\at\si$ and $X_+\at\si$, respectively.
}%
\label{Fig:MomentXi}
\end{figure}%
The small dissipation assumption implies (see also Lemma \ref{Lem:PS.AuxFormula} below)
that the moment
\begin{align}
\label{Eqn:Def.MomentXi}
\xi\at{t}&:=\int_{\Rset} 
\bat{H^\prime\at{x}-\si\at{t}}^2\varrho\pair{t}{x}\dint{x}
\end{align}
is also small, and we conclude that all the mass of the system must be concentrated in narrow peaks 
located at the solutions to $H^\prime\at{x}=\si\at{t}$, which are illustrated in Figure \ref{Fig:MomentXi} and provided by the functions $X_i$ from 
Remark \ref{Rem:PropertiesOfXi}. We can therefore
(at least in a weak$\star$-sense)
approximate
\begin{align}
\label{Eqn:Approx2}
\varrho\pair{t}{x}\approx \delta_{X_-\at{\si\at{t}}}\quad \text{for}\quad\si\at{t}<\si_*,\qquad
\varrho\pair{t}{x}\approx \delta_{X_+\at{\si\at{t}}}\quad \text{for}\quad\si\at{t}>\si^*
\end{align}
as well as
\begin{align}
\label{Eqn:Approx1}
\varrho\pair{t}{x}\approx
\sum_{i\in\{-,0,+\}}
m_i\at{t}\delta_{X_i\at{\si\at{t}}}\at{x}
\qquad \text{for}\quad \si\at{t}\in\oointerval{\si_*}{\si^*},
\end{align}
where the partial masses are defined by
\begin{align}
\label{Eqn:Def.PartialMasses}
m_-\at{t}:=
\int_{-\infty}^{x^*} \varrho\pair{t}{x}\dint{x},\qquad
m_0\at{t}:=
\int_{x^*}^{x_*} \varrho\pair{t}{x}\dint{x},\qquad
m_+\at{t}:=
\int_{x_*}^{+\infty} \varrho\pair{t}{x}\dint{x}.
\end{align}
Notice that $m_-\at{t}+m_0\at{t}+m_+\at{t}=1$ holds by construction
and that the moment $\xi$ can be regarded as the formal limit of the dissipation as $\nu\to0$.
\par
Thanks to \eqref{Eqn:Approx2}, we have $m_0\at{t}\approx m_+\at{t}\approx0$ for $\si\at{t}<\si_*$ and the dynamical
constraint \eqref{Eqn:FP2P} implies $X_-\bat{\si\at{t}}\approx \ell\at{t}$, which determines the evolution of $\si$. Similarly,
with $\si\at{t}>\si^*$ we find $m_-\at{t}\approx m_0\at{t}\approx0$ and $X_+\bat{\si\at{t}}\approx \ell\at{t}$.
These results reflect that $H_\si$ is a single-well potential for both $\si<\si_*$ and $\si>\si^*$ attaining the global minimum at
$X_-\at{\si}$ and $X_+\at{\si}$, respectively.
\par
In the case of $\si\at{t}\in\oointerval{\si_*}{\si^*}$, the corresponding effective potential 
has two local minima and a local maximum corresponding to the three possible peak positions in \eqref{Eqn:Approx1}.
The peaks located at $X_\pm\bat{\si\at{t}}$ are dynamically stable because
adjacent characteristics of
the transport operator in \eqref{Eqn:FP1} approach each other exponentially fast for $H^{\prime\prime}\at{x}>0$. 
Moreover, asymptotic analysis of the entropic effects reveals that each
stable peak is basically a rescaled Gaussian with width of order $\nu/\sqrt{H^{\prime\prime}}\at{X_\pm\at\si}$.
A peak at the center position $X_0\at{\si}$, however, is dynamically unstable because the spinodal characteristics separate exponentially fast with local rate proportional to $\tau^{-1}$, and because the width of each peak is at least of order $\nu$.
Each possible initial peak at $X_0\at{0}$ therefore disappears rapidly,
and by enlarging $t_0$ if necessary we can assume that $m_0\at{t}\approx 0$ for all $t\geq t_0$.
(This is different to the slow reaction regime, in which unstable peaks can survive for a long time
due to $0<\nu\ll\tau\ll1$).
\par
In summary, for any time $t\geq t_0$ with $\si\at{t}\in\oointerval{\si_*}{\si^*}$ 
we expect that the mass is concentrated in the two stable peaks at $X_\pm\at{\si\at{t}}$. In the limit $\nu\to0$,
we therefore have $m_0\at{t}=0$ and hence
\begin{align*}
m_-\at{t}+m_+\at{t}=1,\qquad
\ell\at{t}=m_-\at{t}X_-\bat{\si\at{t}}+m_+\at{t}X_+\bat{\si\at{t}},
\end{align*}
where the last identity stems from the dynamical constraint \eqref{Eqn:FP2P}. Notice that
these formulas hold also for $\si\at{t}<\si_*$ and $\si\at{t}>\si^*$ with
$m_+\at{t}=0$ and $m_-\at{t}=0$, respectively.
%
%
\paragraph*{Dynamics of peaks} 
%
%

%
%
It remains to understand 
the dynamics of the multiplier $\si\at{t}$ and the partial masses
$m_\pm\at{t}$ in the case of
$\si\at{t}\in\oointerval{\si_*}{\si^*}$. To this end we also suppose, at least for times $t\geq t_0$, that $\si$ is changing regularly, that means on the time scale $1$. The underlying idea is that 
the instantaneous probability density $\varrho\pair{t}{\cdot}$ corresponds to a meta-stable state of the linear Fokker-Planck equation with frozen $\si$, and hence that the mass flux between the different phases can in fact be related to Kramers formula.
On a heuristic level it is clear that this regularity assumption on $\si$ is 
connected with our first assumption on the smallness of the dissipation since both predict the formation of two narrow peaks,
 but it is not clear whether there exists a corresponding rigorous estimate.
\par
As already mentioned in \S\ref{sect:intro}, the key observation  for $\si\at{t}\in\oointerval{\si_*}{\si^*}$ is that the two spatially separated peaks can 
exchange mass by a Kramers-type phase transition, that means because the stochastic fluctuations allow the particles to
cross the energetic barriers of the effective potential \eqref{Eqn:EffectivePotential}. Using the energy barriers $h_\pm\at{\si}$ from \eqref{Eqn:DefEnergyBarriers}, the resulting net mass flux can be quantified by 
\begin{align}
\label{Eqn:KramersFormula}
-\dot{m}_-\at{t}\approx + \dot{m}_+\at{t}
\approx m_-\at{t}F_-\at{t} - m_+\at{t}F_+\at{t},\qquad
\tau F_\pm \approx  C_\pm\at{\si}\exp\at{-\frac{h_\pm\at{\si}}{\nu^2}},
\end{align}
which is a variant of Kramers' celebrated asymptotic formula, see \cite{HNV12} for the details. Notice also 
that the precise values of the constants $C_\pm\at{\si}$, which can be computed explicitly,
are not important in our context because the dominant effects are, as discussed below, the dynamical constraint and time dependence of the energy barriers $h_\pm$. 
\bigpar
We next discuss the implications of \eqref{Eqn:KramersFormula} for the different ranges of $\si$. In particular, we argue that
Kramers-type phase transitions can change the partial masses $m_\pm$ in the limit $\nu\to0$ if and only if
$\si$ attains one of the two critical values $\si_\#$, $\si^\#$.
\par
\emph{\ul{Subcritical case:}} For $\si\at{t}\in\oointerval{\si_\#}{\si^\#}$, both energetic barriers of 
$H_\si$ exceed the critical value, that means we have
$h_-\at{\si\at{t}}>h_\#$ and $h_+\at{\si\at{t}}>h_\#$. Combining this
with $\tau^{-1}\approx \exp\at{h_\#/\nu^2}$, see Assumption \ref{Ass:Tau}, we then obtain 
\begin{align*}
 F_\pm\at{t}\approx\tau^{-1}\exp\at{-\frac{h_\pm\bat{\si\at{t}}}{\nu^2}}=\exp\at{\frac{h_\#-h_\pm\bat{\si\at{t}}}{\nu^2}}
\ll1
\end{align*}
and conclude that there is virtually no mass exchange between both phases. The macroscopic dynamics therefore reduces to
\begin{align*}
\dot{m}_\pm\at{t}=0,\qquad \dot\ell\at{t}=\dot\si\at{t}\Bat{
m_-\at{t}X_-^\prime\bat{\si\at{t}}+
m_+\at{t}X_+^\prime\bat{\si\at{t}}
}
\end{align*}
and describes that the localized peaks are just transported by the dynamical constraint \eqref{Eqn:FP2}, see the right panel
in Figure \ref{Fig:Transport}.
\par
\begin{figure}[ht!]%
\centering{%
\includegraphics[height=.275\textwidth, draft=\figdraft]{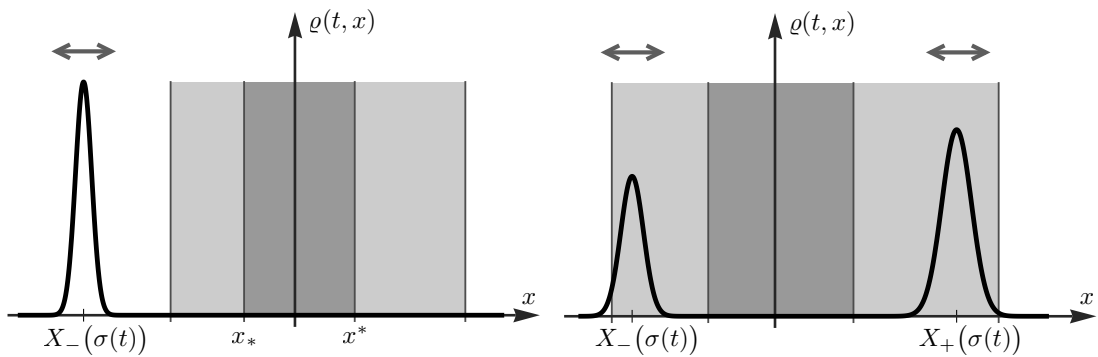}%
}%
\caption{\emph{Left panel:} 
For supercritical $\sigma<\si_\#$, all the mass is contained in a single stable peak, which is located at $X_-\at{\si}$ and 
transported by the dynamical constraint \eqref{Eqn:FP2P}. (A similar statement holds for supercritical $\si>\si_\#$.)
\emph{Right panel:} For subcritical $\si\in\oointerval{\si_\#}{\si^\#}$, the mass is in general
concentrated in two stable peaks, which are located at 
$X_\pm\at{\si}$ and move according to the dynamics of $\ell$. \emph{Both panels:} The width of each peak is proportional to ${\nu}/\sqrt{H^{\prime\prime}\at{X}}$, where $X$ denotes the position, and the arrows indicate that the peaks can move either to the left (for $\dot\ell<0$) or to the right 
(for $\dot\ell>0$). The shaded regions in light and dark gray represent the intervals 
$\ccinterval{X_-\at{\si_\#}}{X_+\at{\si^\#}}$ and $\ccinterval{x^*}{x_*}$, respectively.
}%
\label{Fig:Transport}
\end{figure}%
\begin{figure}[ht!]%
\centering{%
\includegraphics[height=.275\textwidth, draft=\figdraft]{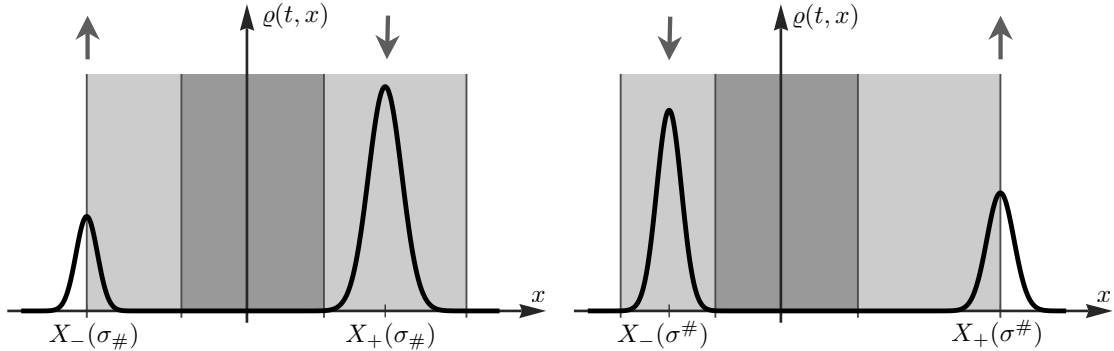}%
}%
\caption{For critical $\si$, the coexisting stable peaks exchange mass by a Kramers-type phase transition, where
$\sigma=\sigma_\#$ (left panel) and
$\sigma=\sigma^\#$ (right panel) correspond to 
negative and positive mass flux, respectively.
}%
\label{Fig:Kramers}
\end{figure}%
\emph{\ul{Critical cases:}} For $\si\at{t}=\si_\#$, we obtain $h_+\at{\si\at{t}}=h_\#<h_-\at{\si\at{t}}$ and hence
\begin{align*}
F_-\at{t} \ll 1 \approx  F_+\at{t}.
\end{align*}
In other words, the particles can move from the well at $X_+\bat{\si\at{t}}$ towards the well
at $X_-\bat{\si\at{t}}$ but not the other way around. Kramers-type phase transitions are therefore feasible
and correspond to time intervals of positive length in which the limit $\nu\to0$ provides 
\begin{align*}
\si\at{t}=\si_\#,\qquad +\dot{m}_+\at{t}=-\dot{m}_-\at{t}\leq0,\qquad
\dot\ell\at{t}=
\dot{m}_-\at{t}X_-\bat{\si_\#}+
\dot{m}_+\at{t}X_+\bat{\si_\#}.
\end{align*}
Notice that the macroscopic dynamics of $m_\pm$ is
completely determined by the evolution of $\ell$ and hence independent of
the prefactor $C_+\at{\si_\#}$ in Kramers formula \eqref{Eqn:KramersFormula}. This seems to be surprising at a first glance but can be understood as follows. For $\nu>0$, small
fluctuations around $\si_\#$ -- this means $\si\at{t}=\si_\#+\nu^2\tilde{\si}\at{t}$ with $\tilde\si\at{t}=\DO{1}$ -- can change $\dot{m}_\pm\at{t}$ considerably
and are hence capable of adjusting the mass flux to the dynamical constraint \eqref{Eqn:FP2P}.
\par
The discussion of the second critical case $\si=\si^\#$ is entirely similar. The macroscopic dynamics in this case is  given by
\begin{align*}
\si\at{t}=\si^\#,\qquad +\dot{m}_+\at{t}=-\dot{m}_-\at{t}\geq0,\qquad
\dot\ell\at{t}=
\dot{m}_-\at{t}X_-\bat{\si^\#}+
\dot{m}_+\at{t}X_+\bat{\si^\#}
\end{align*}
and reflects an effective mass flux from the well at $X_-\bat{\si^\#}$ towards the
well at $X_+\bat{\si^\#}$. Both critical cases are illustrated in Figure \ref{Fig:Kramers}. 
\par
\emph{\ul{Supercritical cases:}}
For $\si\at{t}\in\oointerval{\si_*}{\si_\#}$, we verify
\begin{align*}
h_+\at{\si\at{t}}<h_\#<h_-\at{\si\at{t}},\qquad
F_-\at{t}\ll 1 \ll F_+\at{t}
\end{align*}
and conclude that particles escape very rapidly from the well at $X_+\at{\si\at{t}}$ 
but are trapped inside the other well at $X_-\at{\si\at{t}}$. The only consistent choice for the
macroscopic dynamics in this case is 
\begin{align*}
m_-\at{t}=1,\qquad
m_+\at{t}=0,\qquad \dot\ell\at{t}=X_-^\prime\bat{\si\at{t}}\dot\si\at{t},
\end{align*}
which describes the transport of a single stable peak, see the left panel from Figure \ref{Fig:Transport}. To be more precise,
for states with $\si\at{t}\in\oointerval{\si_*}{\si_\#}$ and $m_+\at{t}>0$, the mass-dissipation estimates
derived below imply that $\calD\at{t}$ is large, and hence we expect 
that such states cannot be reached dynamically. (If such states are imposed as initial data,
a very rapid mass transfer during the initial transient regime produces $m_-\at{t_0}\approx0$.)
Similarly, for $\si\at{t}\in\oointerval{\si^\#}{\si^*}$ the macroscopic
evolution reads
\begin{align*}
m_-\at{t}=0,\qquad
m_+\at{t}=1,\qquad \dot\ell\at{t}=X_+^\prime\bat{\si\at{t}}\dot\si\at{t}
\end{align*}
and can be justified by analogous arguments. Notice
that the limit dynamics in the
supercritical cases is the same as in the single-well cases $\si\at{t}<\si_*$ and $\si\at{t}>\si^*$.
%
%
\subsection{Rate-independent model for the limit dynamics}\label{sect:limitmodel}
%
%
%
The above formulas for the limit dynamics can be translated into closed 
evolution equations for $\ell$, $\si$, and the phase fraction $\mu:=m_+-m_-$. These equations
are illustrated in Figure \ref{Fig:LimitModel} and turn out to be rate-independent 
because the macroscopic solution corresponding to $\tilde{\ell}\at{t}=\ell\at{ct}$ with $c>0$ 
is given by $\tilde{\si}\at{t}=\si\at{ct}$ and $\tilde{\mu}\at{t}=\mu\at{ct}$. For more details on the general theory of rate-independent systems and the different solution concepts we refer to \cite{Mie11c}. Moreover,
the limit dynamics exhibit hysteresis in the sense that the value of the output $\si$ at time $t$ depends not only
on the instantaneous value of the input $\ell$ but also on the history of the evolution (or, equivalently, on the state of the internal variable $\mu$). 
\par
In order to give a precise formulation of our limit model, we now define an 
appropriate notion of solutions in the space of Lipschitz-continuous functions. To this end 
we observe that the parameter constraints
\begin{align}
\label{Eqn:LimDyn.Constraints}
\mu\in\ccinterval{-1}{1},\qquad \si\in\Rset,\qquad 
\ell\in\left\{\begin{array}{lcl}
\{X_-\at\si\}&\text{for}&\si<\si_\#,\\
\ccinterval{X_-\at\si}{X_+\at\si}&\text{for}&\si\in\ccinterval{\si_\#}{\si^\#},\\
\{X_+\at\si\}&\text{for}&\si>\si^\#
\end{array}\right.
\end{align}
define the macroscopic state space
\begin{align}
\label{Eqn:LimDyn.Def1}
\Om:=\Big\{\triple{\ell}{\si}{\mu}\in\Rset^3 \quad \text{satisfying}\quad\eqref{Eqn:LimDyn.Constraints}\Big\}
\end{align}
and that the macroscopic analogue to the dynamical constraint \eqref{Eqn:FP2P} can be written as $\ell=\calL\pair{\si}{\mu}$ with
\begin{align}
\label{Eqn:LimDyn.Def2}
\calL\pair{\si}{\mu}:=\frac{1-\mu}{2}X_-\at{\sigma}+\frac{1+\mu}{2}X_+\at{\si}.
\end{align}
We also recall
that any Lipschitz function admits a classical derivative in almost all points (Rademacher's Theorem, see for instance
\cite[Proposition 23.2]{DiB02}
).
\begin{figure}[ht!]%
\centering{%
\includegraphics[height=.3\textwidth, draft=\figdraft]{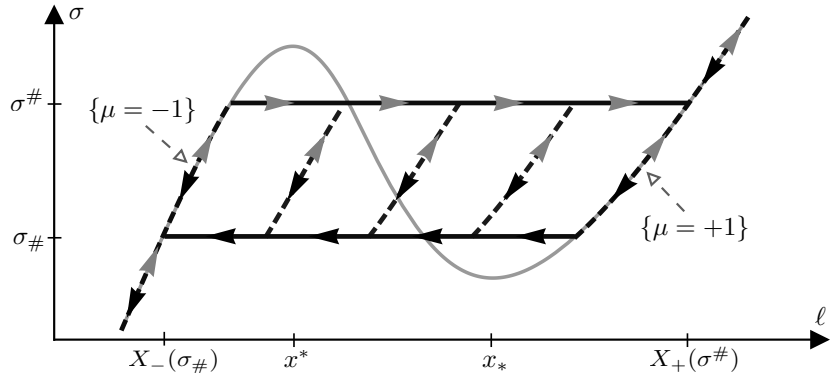}%
}%
\caption{Cartoon of the \emph{macroscopic limit dynamics} in the $\pair{\ell}{\si}$-plane.
The gray solid curve is the graph of $H^\prime$, the dashed black lines represent the level curves of $\mu$, and the 
solid black lines correspond to the critical values $\si_\#$ and $\si^\#$, for which mass  
transfer according to a Kramers-type phase transition is feasible. 
The black and gray arrows indicate the evolution for 
decreasing and increasing $\ell$, respectively.
\emph{Microscopic dynamics for small $\nu$:} The evolution of $\varrho$
along the level sets of $\mu$ is illustrated in Figure \ref{Fig:Transport}, whereas the panels in
Figure \ref{Fig:Kramers} correspond to $\si\at{t}=\si_\#$  and $\si\at{t}=\si^\#$.
}%
\label{Fig:LimitModel}
\end{figure}%
\begin{definition}[solutions to the limit model]
\label{Def:LimitModel}
A pair $\pair{\si}{\mu}\in\fspaceC^{0,\,1}\at{\ccinterval{0}{T};\Rset^2}$ is called \emph{a solution to the limit problem} for given $\ell\in\fspaceC^{0,\,1}\at{\ccinterval{0}{T}}$,
if the pointwise relations 
\begin{align}
\label{Eqn:LimitModel.AR}
\btriple{\ell\at{t}}{\si\at{t}}{\mu\at{t}}\in\Om\quad\text{with}\quad
\ell\at{t}=\calL\bpair{\si\at{t}}{\mu\at{t}}
\end{align}
are satisfied for all $t\in\ccinterval{0}{T}$, and if the dynamical relations
\begin{align}
\label{Eqn:LimitModel.DR}
\dot\mu\at{t}=0\quad\text{if}\quad \si\at{t}\notin\{\si_\#,\si^\#\},\qquad
\dot\mu\at{t}\leq 0\quad\text{if}\quad \si\at{t}=\si_\#,\qquad
\dot\mu\at{t}\geq 0\quad\text{if}\quad \si\at{t}=\si^\#
\end{align}
hold for almost all $t\in\ccinterval{0}{T}$. 
\end{definition}
In Appendix \ref{app:LimitModel}, Proposition \ref{App:Prop:WellPosednessLimitModel} we prove that for each $\ell$ as in Assumption \ref{Ass:Constraint} and any admissible choice of the initial data
$\pair{\si\at{0}}{\mu\at{0}}$ there exists a unique solution 
to the limit model, which is moreover piecewise continuously differentiable.
We also mention that the limit model is equivalent to 
 a constrained variational inequality. More precisely, introducing the convex functionals
\begin{align*}
\calR\at{\dot\mu}:=\dot\mu\left\{\begin{array}{lll}
\si_\#&\text{if}&\dot\mu\leq0,\\
\si^\#&\text{if}&\dot\mu\geq0,
\end{array}\right.\qquad
\calI\at{\mu}:=\left\{\begin{array}{lll}
0&\text{if}&-1\leq\mu\leq+1,\\
+\infty&\text{else},
\end{array}\right.
\end{align*}
the dynamical relations \eqref{Eqn:LimitModel.DR}  can be 
formulated as
\begin{align}
\notag
\si\at{t}\in\partial_{\dot\mu}\calR\bat{\dot{\mu}\at{t}}+\partial_\mu\calI\bat{\mu\at{t}}.
\end{align}
Here, $\partial$ means the set-valued derivative in the sense of subgradients, and the dynamical constraint
enters via the pointwise relations \eqref{Eqn:LimitModel.AR}.
\bigpar
We finally emphasize that the above limit model also governs the 
macroscopic evolution in the borderline regime and in the quasi-stationary limit. Specifically, 
the case $\tau\sim\nu^{p_\#}$ corresponds to $\si_\#=\si_*$ and $\si^\#=\si^*$, whereas for 
$h_\#\geq h_\threshold=h_\pm\at{0}$ we set $\si_\#=\si^\#=0$. We also expect that our proof 
strategy as summarized in \S\ref{sect:overview} can be generalized to 
these extreme cases but important details would be different. In the quasi-stationary limit,
our arguments can even be simplified since there is no subcritical regime at all. In the borderline regime, however,
many of the asymptotic estimates in \S\ref{sect:AuxResults} and \S\ref{sect:limit} must be formulated more carefully
since the mass can now be concentrated near $x_*$ or $x^*$ and because some of error terms decay, as $\nu\to0$, no longer 
exponentially but only algebraically. 
%
\subsection{Overview on the proof strategy}\label{sect:overview}
%
The heuristic derivation of the limit dynamics in \S\ref{sect:HeuristicDynamics} relies on two crucial assumptions
for $t\geq t_0$, namely
$\at{a}$ that the dissipation $\calD\at{t}$ is pointwise small, and $\at{b}$ that $\dot{\si}\at{t}$ is pointwise of order $1$.
In numerical simulations one observes such a nice behavior but our convergence proof is based on weaker statements, which are, however, sufficient for passing to the limit $\nu\to0$.
Specifically, below we only show
 $\at{a}$ that the moment
$\xi\at{t}+m_0\at{t}$ remains small, and $\at{b}$ 
that  the dynamical multiplier $\si$ is Lipschitz continuous up to small error terms. 
\par
A further technical difficulty is that the constants in many of the asymptotic estimates derived below 
degenerate if $\si$ approaches one of the critical values $\{\si_\#, \si^\#\}$ because the Kramers time scale \eqref{Eqn:KramersScale} is then of order $1$. Our strategy in \S\ref{sect:AuxResults} and \S\ref{sect:limit} is 
therefore as follows. We introduce an artificial parameter $\eps>0$ and establish most of our results concerning 
the effective dynamics for $0<\nu\ll1$ under the assumptions $\at{i}$ that $\eps$ is small but independent of $\nu$
and $\at{ii}$ that $\si$ remains outside of the $\eps$-neighborhood of $\si_*$ and/or $\si^*$. In this way we obtain 
strong results for both the subcritical and the supercritical evolution but have only incomplete control over the dynamics 
whenever $\si\at{t}\approx\si_\#$ or $\si\at{t}\approx\si^\#$. In the final step we then pass to the limit 
$\nu\to0$ along sequences $\at{\eps_n}_{n\in\Nset}$ and $\at{\nu_n}_{n\in\Nset}$ with $\eps_n\to0$ and 
$0<\nu_n\leq \bar{\nu}\at{\eps_n}\to0$, where the critical value $\bar{\nu}\at{\eps_n}$ will be identified in the 
proof of Theorem \ref{Thm:Compactness}. 
\bigpar
We proceed with a more detailed overview about the basic ideas 
for the rigorous justification of the limit dynamics as used in 
\S\ref{sect:AuxResults} and \S\ref{sect:limit}.
%
\par {\bf{Mass dissipation estimates.}} 
In  order to control the amount of mass that is 
concentrated near the stable peak positions, we establish in \S\ref{sect:PandMConstants}, 
see Lemma \ref{Lem:MD.AuxEstimate} and also \eqref{Eqn:MD.AuxFormula}, the estimate 
\begin{align}
\label{Eqn:MassDissAbtract}
\int_{J}\varrho\pair{t}{x}\dint{x}\leq C_1 \calD\at{t} +C_2\int_{I}\varrho\pair{t}{x}\dint{x},
\end{align}
where the constants $C_1$ and $C_2$ depend crucially on $\nu$, the instantaneous value of $\si$, 
and the choice of the intervals $J\subseteq{I}\subseteq\Rset$. More precisely, denoting by $\ga_\si$ 
the (re-scaled) equilibrium solution of the linearized Fokker Plank equation with frozen $\si$ -- 
see \eqref{Eqn:DefEquilibriumDensities} for a precise definition -- the constant $C_1$ from \eqref{Eqn:MassDissAbtract} is basically 
the Poincar\'e constant of $\ga_\si$ restricted to $I$ and $C_2$ measures the equilibrium mass 
inside $J$. For our analysis it is essential to quantify these constants  
for small values of $\nu$. While the computation of $C_2$ is rather straight forward, the asymptotic 
analysis of $C_1$ is more involved and requires to estimate the so-called Muckenhoupt constants of 
$\ga_\si|_I$; see \eqref{Eqn:DefMuckenhouptConstants} and the proofs in \S\ref{sect:AsymptoticsForCP}. In \S\ref{sect:MassInPeaks} we finally formulate two particular mass-dissipation estimates that cover 
different ranges for $\si$ and correspond to different choices of $J$ and $I$. Specifically, in Lemma \ref{Lem:MD.MassOutsideTwoPeaks} we assume that $H_\si$ is a genuine double-well potential and control 
the mass outside the two stable peaks, whereas Lemma \ref{Lem:MD.MassOutsideOnePeak} estimates 
the mass outside a single peak located at the global minimum of $H_\si$. 
%
\par {\bf{Dynamical stability estimates.}} 
Since we lack pointwise estimates for $\calD\at{t}$, mass-dissipation estimates like \eqref{Eqn:MassDissAbtract} are not sufficient for 
passing to the limit $\nu\to0$. Our arguments are therefore 
also based on pointwise upper bounds
for $\xi\at{t}+m_0\at{t}$ as these imply the 
dynamical stability of localized peaks  (in a weak sense). In \S\ref{sect:PeakStability} we 
start these stability investigations and derive some preliminary results 
using moment ODEs as well as the maximum principle for linear Fokker-Planck equations. 
In particular, in Lemma \ref{Lem:PS.NoMassInSpinodalRegion1} and Lemma \ref{Lem:PS.NoMassInSpinodalRegion2} we 
control the evolution of $\xi\at{t}+m_0\at{t}$ under the assumption that $\si$ remains confined to certain ranges
depending on $\eps$.
%
\par {\bf{Monotonicity relations.}} 
Further building blocks for the limit $\nu\to0$ are discussed in \S\ref{sect:Montonicity} and provide monotonicity
relations that control the evolution of the partial masses $m_\pm$ up to small error terms. In the proof of Lemma \ref{Lem:ChangeOfMasses} 
we deduce from certain moment equations that $m_-$ and $m_+$ are essentially decreasing at all times $t$ 
with $\si\at{t}>\si_\#+\eps$ and $\si\at{t}<\si^\#-\eps$, respectively. These results can be regarded as 
the rigorous analogue to the informal discussion about the orders of magnitude in Kramers 
formula \eqref{Eqn:KramersFormula}. In particular, they guarantee that there is virtually no mass flux 
as long as $\si$ remains confined to the subcritical range $\oointerval{\si_\#+\eps}{\si^\#-\eps}$ and 
give moreover rise to monotonicity relations between the 
dynamical control $\ell$ and the dynamical multiplier $\si$; see Lemma \ref{Lem:EllAndSigma}.  
\par {\bf{Approximation by stable peaks.}} 
In \S\ref{sect:AppByStablePeaks} we continue our investigations on the dynamical peak stability and provide
a corresponding approximation result that holds for all values of $\si\at{t}$ and ensures additionally that almost 
all mass is contained in a single peak as long as $\si\at{t}$ is strictly supercritical. More precisely, in 
Lemma \ref{Lem:ZetaDeltaBounds} we show that $\calD\at{t_0}\leq\tau^\beta$ with
$\beta\in\oointerval{0}{1}$ implies that the moment
\begin{align}
\label{Eqn:DefZeta}
\zeta\at{t} := \xi\at{t}+m_0\at{t}+
\left\{
\begin{array}{lcl}
m_+\at{t}&\text{for}&\si\at{t}\in\ocinterval{-\infty}{\si_\#-\eps},\\
0&\text{for}&\si\at{t}\in\oointerval{\si_\#-\eps}{\si^\#+\eps},\\
m_-\at{t}&\text{for}&\si\at{t}\in\cointerval{\si^\#+\eps}{+\infty},
\end{array}\right.
\end{align}
is bounded by $C\nu^2$ for all times $t\in\ccinterval{t_0}{T}$. The proof combines
all auxiliary results from \S\ref{sect:AuxResults} for the different $\si$-ranges and relies additionally on the following key observation:
The $\fspaceL^1$-bound for the dissipation ensures that each time interval with length of order $1$
contains at least one time $t$ at which 
the instantaneous state $\varrho\pair{t}{\cdot}$ can be controlled by mass-dissipation estimates.
\par {\bf{Continuity estimates for $\si$.}} 
Since our asymptotic results for small $\nu$ do not provide uniform bounds for
$\dot{\si}\at{t}$, we describe in \S\ref{sect:Continuity} the evolution of $\si$ by combining the moment estimates for
$\zeta$ with the monotonicity relations from \S\ref{sect:AuxResults} and the dynamical  constraint \eqref{Eqn:FP2P}. 
In particular, Lemma \ref{Lem:LipschitzSigma} reveals that $\si$ is almost Lipschitz continuous in the sense that 
$\abs{\si\at{t_2}-\si \at{t_1}}$ can be bounded by $C\abs{t_2-t_1}+E$ for some
constant $C$ independent of $\nu$, where the error term $E$ depends nicely on $\nu$ and $\eps$. 
These Lipschitz-type estimates not only allow us to verify the dynamical constraint in the limit $\nu=0$ but 
also provide the compactness of $\si$ as $\nu\to0$.
\par {\bf{Passage to the limit.}} 
In  \S\ref{sect:compactness} we finally take the limit $\nu\to0$. We first justify the limit model from 
\S\ref{sect:limitmodel} along subsequences $\eps_n,\nu_n\to0$ and discuss afterwards the case of 
well-prepared initial data, see Theorems \ref{Thm:Compactness} and \ref{Thm:Convergence}, respectively.
%
%
\section{Auxiliary results}\label{sect:AuxResults}
%
%
%
The quantities $c$, $C$, and $\alpha$ always denote positive but generic constants (so their
value may change from line to line) which are 
independent of $\nu$ but can depend on $H$, $\ell$, $T$, the constant from Assumption \ref{Ass:RegularityOfInitialData}, and other
parameters to be introduced below.  Notice that 
the scaling law between $\tau$ and $\nu$, see Assumption \ref{Ass:Tau}, implies that 
a given positive quantity is exponentially small with respect to $\nu$ if and only if
it  is bounded by $ C\tau^\alpha$ for some constants $\al$ and $C$ independent of $\nu$.

\subsection{Mass-dissipation estimates}\label{sect:MassDiss}
In this section we derive mass-dissipation estimates, this means we show that small dissipation requires the total 
mass of the system to be concentrated near either both or one of the stable peak positions
$X_-\at{\si}$ and $X_+\at\si$. These estimates become important in \S\ref{sect:limit}
because they guarantee (in combination with the $\fspaceL^1$-bound for $\calD$) 
that for each time $t_1$ there exists another time $t_2\in\ccinterval{t_1}{t_1+\tau^\beta}$ with $0<\beta<1$ such that $\varrho\pair{t_2}{\cdot}$ consists of two narrow peaks located at
$X_-\at{\si\at{t_2}}$ and $X_+\at{\si\at{t_2}}$.
In the present section, however, all arguments and results hold pointwise in $t$ 
and thus we omit the time dependence in all quantities.
\bigpar
For the following considerations
we introduce, for each $\si\in\Rset$, the relative equilibrium density
\begin{align}
\label{Eqn:DefEquilibriumDensities}
\ga_\si\at{x}:=\exp\at{-\frac{H_\sigma\at{x}}{\nu^2}},\quad
z_\si:=\int_\Rset\ga_\si\at{x}\dint{x}
\end{align}
see Figure \ref{Fig:Equilibria} for an illustration, and denote
by $\ga_{\si,-}$ and 
$\ga_{\si,+}$ the restriction of $\ga_\si$ to the intervals
\begin{align}
\label{Eqn:DefSiIntervals}
I_{\si,-}:=\oointerval{-\infty}{X_0\at{\si}}\qquad\text{and}\qquad
I_{\si,+}:=\oointerval{X_0\at{\si}}{+\infty},
\end{align}
respectively. The functions $\ga_\si$ are naturally related to states with small dissipation. In fact,
$\ga_\si/z_\si$ is the global equilibrium of the linear Fokker-Planck equation \eqref{Eqn:FP1} 
with $\si\at{t}=\si=\const$, and the modified energy functional
\begin{align}
\label{Eqn:DefEnergyFunctional.A}
\calE_\si\at{\varrho}:=\nu^2\int_{\Rset}\varrho\at{x}\ln\varrho\at{x}\dint{x}+\int_\Rset H_\si\at{x}\varrho\at{x}\dint{x},
\end{align}
 just gives the relative entropy of $\varrho$ with respect to $\ga_\si$, that is
\begin{align}
\label{Eqn:DefEnergyFunctional.B}
\calE_\si\at{\varrho}=\nu^2 \int_{\Rset}\varrho\at{x}\ln\at{\frac{\varrho\at{x}}{\ga_\si\at{x}}}\dint{x}.
\end{align}
Notice that the definition of the modified energy $\calE_\si$ involves $H_\si$ instead of $H$, so
the total energy $\calE$ from \eqref{Eqn:DefEnergy} is given by
$\calE_\si+\si\ell$ .  

\begin{figure}[ht!]%
\centering{%
\includegraphics[height=.275\textwidth, draft=\figdraft]{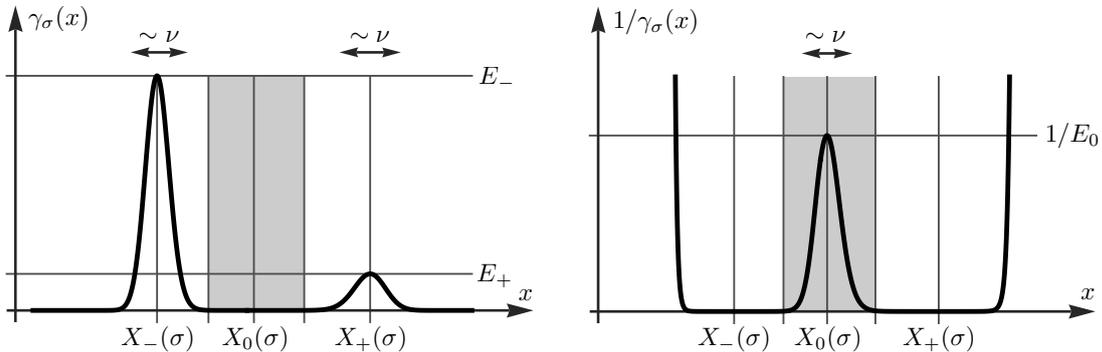}%
}%
\caption{The relative equilibrium density $\gamma_\sigma$ for $\si_*<\si<0$, where
$E_j:=\exp\at{-H_\si\at{X_j\at{\si}}/\nu^2}$.
For $0<\nu\ll1$, the density $\gamma_\sigma$ exhibits two peaks located
at $X_-\at{\si}$ and $X_+\at{\si}$. The width of these peaks is of order $\DO{\nu}$  and 
the mass ratio between the peaks scales
with $\exp\at{\at{h_-\at{\si}-h_+\at{\si}}/\nu^2}$. The inverse density $1/\gamma_\sigma$
has a peak at $X_0\at{\si}$ and grows very rapidly for $x\to\pm\infty$.
}%
\label{Fig:Equilibria}
\end{figure}%

%

\subsubsection{On Poincar\'e and Muckenhoupt constants}\label{sect:PandMConstants}
We now summarize some well-known facts about $\fspaceL^1$-measures which then
allow us to establish mass-dissipation estimates in \S\ref{sect:MassInPeaks}. 
Within this section, let $I=\oointerval{x_-}{x_+}$ be some (bounded or unbounded) interval, 
$\ga$
be a positive $\fspaceL^1$-function on the interval $I$, and $C_P\at\ga$ the
Poincar\'e constant of $\ga$. The latter reads
\begin{align}
\label{Eqn:DefPoincareConstants}
\frac{1}{C_P\at\ga} := \inf_{w\in\fspaceL^2\at{\ga\dint{x}}}\frac{\D\int_{I}\at{w^\prime\at{x}}^2\ga\at{x}\dint{x}}{\D\int_I\bat{w\at{x}-w_\av}^2\ga\at{x}\dint{x}},
\qquad w_\av:=\frac{\D\int_I w\at{x}\ga\at{x}\dint{x}}{\D\int_I\ga\at{x}\dint{x}},
\end{align}
where $w^\prime$ abbreviates the derivative of $w$ with respect to $x$ and
$\fspaceL^2\at{\ga\dint{x}}:=\{w\;:\;\int_Iw\at{x}^2 \ga\at{x}\dint{x}<\infty\}$.
For each $x_0\in{I}$, we also introduce 
the one-sided Muckenhoupt constants $C_M^\pm\pair{\ga}{x_0}$ by
\begin{align}
\label{Eqn:DefMuckenhouptConstants}
\begin{split}
C_M^-\pair\ga{x_0} &:= 
\sup_{x\in\ocinterval{x_-}{x_0}}\at{\int_{x}^{x_0}\frac{1}{\ga\at{y}}\dint{y}}
\at{\int_{x_-}^{x}\ga\at{y}\dint{y}},
\\
C_M^+\pair\ga{x_0}&:= 
\sup_{x\in\cointerval{x_0}{x_+}}
\at{\int_{x_0}^{x}\frac{1}{\ga\at{y}}\dint{y}}
\at{\int_{x}^{x_+}\ga\at{y}\dint{y}}.
\end{split}
\end{align}
It is known, see the discussion in \cite{Fou05,Sch12},
that $\ga$ admits a finite Poincar\'e constant if and only if
the Muckenhoupt constants are bounded. More precisely, 
a lower bound is given by 
\begin{align*}
C_P\at{\ga}\geq \tfrac12\max\Big\{C_M^-\pair{\ga}{x_\med},\,C_M^+\pair{\ga}{x_\med}\Big\},
\end{align*}
where the median $x_\med$ is uniquely defined by $\int_{x_-}^{x_\med}\ga\at{y}\dint{y}=
\int_{x_\med}^{x^+}\ga\at{y}\dint{y}$, and an upper bound can be formulated as follows. 
\begin{lemma}[Muckenhoupt  constants bound Poincar\'e constant]
\label{Lem:MD.PoincareVersusMuckenhoupt}
We have 
\begin{align*}
C_P\at{\ga}\leq 4 \max\Big\{C_M^-\pair{\ga}{x_0},\,C_M^+\pair{\ga}{x_0}\Big\}
\end{align*}
for all $\ga$ and any $x_0\in{I}$.
\end{lemma}
\begin{proof}
The proof can be found in \cite[Proposition 5.21]{Sch12}.
\end{proof}
We also mention that the Muckenhoupt constants can easily be estimated for logarithmically concave functions $\ga$.
\begin{lemma}[$C^{\pm}_M$ for logarithmically concave $\ga$]
\label{Lem:Muckenhoupt.LogConcave}
Let $\ga\at{x}=\exp\bat{-V\at{x}}$.
For any convex and strictly increasing potential $V:\cointerval{x_0}{+\infty}\to\Rset$ we have
\begin{align*}
C_M^+\pair{\ga}{x_0}\leq \sup_{x\geq x_0}\frac{x-x_0}{ V^\prime\at{x}}.
\end{align*}
Similarly,  the estimate
\begin{align*}
C_M^-\pair{\ga}{x_0}\leq \sup_{x\leq x_0}\frac{x-x_0}{ V^\prime\at{x}}
\end{align*}
holds provided that $V:\ocinterval{-\infty}{x_0}\to\Rset$ is convex and  strictly decreasing. 
\end{lemma}
\begin{proof}
By symmetry it is sufficient to study the first case only. For $x\geq x_0$ we estimate
\begin{align*}
\int_{x_0}^x \frac{1}{\ga\at{y}}\dint{y}=\int_{x_0}^x \exp\bat{V\at{y}}\dint{y}\leq \exp\bat{V\at{x}}\at{x-{x_0}}.
\end{align*}
Moreover, employing Taylor expansion of $V$ at $x$ as well as the monotonicity of $V^\prime$ we find
\begin{align*}
\int_x^\infty \ga\at{y} \dint{y}&=\int_x^\infty \exp\bat{-V\at{y}}\dint{y}
\\&\leq
\exp\bat{-V\at{x}}\int_x^\infty \exp\bat{-V^\prime\at{x}\at{y-x}}\dint{y}
=\frac{\exp\bat{-V\at{x}}}{V^\prime\at{x}},
\end{align*}
and the claim follows immediately.
\end{proof}

The mass-dissipation estimates derived below rely on asymptotic expressions for the
Muckenhoupt constants of $\ga_\si$ and the following observation.

\begin{lemma}[variant of the Poincar\'e inequality]
\label{Lem:MD.AuxEstimate} 
For any $\ga$, the estimate
\begin{align*}
\int_J w\at{x}^2\ga\at{x}\dint{x}\leq 2 C_P\at{\ga}\int_I \at{w^\prime\at{x}}^2\ga\at{x}\dint{x}  +2 C_J\at{\ga}
\int_{I} w\at{x}^2\ga\at{x}\dint{x},\qquad
C_J\at\ga:=\frac{\D\int_{J}\ga\at{x}\dint{x}}{\D\int_{I}\ga\at{x}\dint{x}}
\end{align*}
holds for all $w\in\fspaceL^2\at{\ga\dint{x}}$ and any subinterval $J\subseteq{I}$.
\end{lemma}
\begin{proof}

 Thanks to $2ab\leq \eta a^2+\eta^{-1}b^2$ and H\"older's inequality we have
\begin{align*}
2w_\av\int_{J}w\at{x}\ga\at{x}\dint{x}\leq 
\eta w_\av ^2 + \eta^{-1}
\at{\int_J w\at{x}^2\ga\at{x}\dint{x}}\at{\int_J\ga\at{x}\dint{x}},
\end{align*}
where $w_\av$ is defined in \eqref{Eqn:DefPoincareConstants}, and with 
$\eta := 2 \int_J\ga\at{x}\dint{x}$ we estimate
\begin{align*}
\int_I\bat{w\at{x}-w_\av}^2\ga\at{x}\dint{x}&\geq
\int_J\bat{w\at{x}-w_\av}^2\ga\at{x}\dint{x}\\&=
\int_Jw\at{x}^2\ga\at{x}\dint{x}
+w_\av^2\int_J\ga\at{x}\dint{x}
-2w_\av\int_{J}w\at{x}\ga\at{x}\dint{x}
\\&\geq
\tfrac{1}{2} \int_Jw\at{x}^2\ga\at{x}\dint{x}- w_\av^2\int_J\ga\at{x}\dint{x}.
\end{align*}
H\"older's inequality also implies
\begin{align*}
w_\av^2\leq\frac{\D\int_I w\at{x}^2\ga\at{x}\dint{x}}{\D\int_{I}\ga\at{x}\dint{x}},
\end{align*} 
and combining the latter two estimates with the Poincar\'{e} estimate for $w$ and $\ga$, see again \eqref{Eqn:DefPoincareConstants}, we arrive at the desired result. 
\end{proof}

\subsubsection{Asymptotics of Poincar\'e constants \texorpdfstring{for $\ga_\si$}{}}
\label{sect:AsymptoticsForCP}

In this section we derive upper bounds for the Poincar\'e constants of the functions $\ga_\si$ and $\ga_{\si,\pm}$
which have been introduced in \eqref{Eqn:DefEquilibriumDensities} and \eqref{Eqn:DefSiIntervals}, respectively. 
To this end we first mention that standard methods from asymptotic analysis
allows one to justify the following statements concerning the $\nu$-dependence for any fixed value of $\si$:
\begin{enumerate}
\item For $\si>\sigma^*$ or $\si<\si_*$, the effective potential 
$H_\si$ from \eqref{Eqn:EffectivePotential} is a single-well potential which grows quadratically at infinity, and this 
implies 
\begin{align*}
C_P\at{\ga_\si}\sim\nu^2\bat{1+\Do{1}},
\end{align*}
where $\Do{1}$ means as usual arbitrary small for small $\nu$.
\item For $\si\in\oointerval{\si_*}{\si^*}$, the effective potential is 
a genuine double-well potential with energy barriers \eqref{Eqn:DefEnergyBarriers}. The Poincar\'e constant of $\ga_\si$ is therefore given by
\begin{align}
\label{Eqn:ExamplePoimcareConstant}
C_P\at{\ga_\si}\sim\exp\at{\frac{\min\big\{h_-\at\si,\,h_+\at{\si}\big\}}{\nu^2}}\bat{1+\Do{1}},
\end{align}
while the Poincar\'e constants of $\ga_{\si,-}$ and $\ga_{\si,+}$ are of order $\DO{1}$. Notice that \eqref{Eqn:ExamplePoimcareConstant}
involves the same exponential terms as Kramers formula  \eqref{Eqn:KramersFormula}, this means
$C_P\at{\ga_\si}$ determines the time scale on which probabilistic transitions between the different wells of $H_\si$ become relevant.
In particular, we have 
\begin{enumerate}
\item $C_P\at{\ga_\si}\ll1/\tau$  for supercritical $\si$
as the minimal energy barrier $\min\{h_-\at\si,\,h_+\at\si\}$
is smaller than the critical value $h_\#=h_-\at{\si^\#}=h_+\at{\si_\#}$, but
\item $C_P\at{\ga_\si}\gg 1/\tau$  for subcritical $\si$ since both $h_-\at\si$ and $h_+\at\si$ exceed $h_\#$.
\end{enumerate} 
\end{enumerate}
For our purposes, the above asymptotic results are not sufficient 
since both the leading order constants and the next-to-leading order corrections depends 
on $\si$. In the subsequent series of lemmata we therefore establish 
(non-optimal) upper bounds for the Poincar\'e constant $C_P$ that hold uniform in certain ranges of $\si$.
\begin{figure}[ht!]%
\centering{%
\includegraphics[height=.25\textwidth, draft=\figdraft]{\figfile{muckenhoupt}}%
}%
\caption{Examples of the effective potential $H_\si$ with $\si<\si_*$ 
(left panel) and $\si_*<\si<0$ (right panel). The Muckenhoupt constants
$C^{\pm}_M\bpair{\ga_\si}{X_-\at\si}$ are estimated 
in the proofs of Lemma \ref{Lem:MD.Muckenhoupt.B} 
and Lemma \ref{Lem:MD.Muckenhoupt.C} 
}%
\label{Fig:Muckenhoupt}
\end{figure}%

\begin{lemma}[Poincar\'e constants of $\ga_{\si,\pm}$ if $H_\si$ is a double-well potential]
\label{Lem:MD.Muckenhoupt.A}
For each $\eps$ with $0<\eps<\tfrac{1}{2}\at{\si^*-\si_*}$
there exists a constant $C$, which depends only on $\eps$ and $H$,
such that
\begin{align*}
C_P\at{\ga_{\si,\pm}}\leq C
\end{align*}
holds for all $0<\nu\leq1$ and $\si\in\ccinterval{\si_*+\eps}{\si^*-\eps}$.
\end{lemma}
\begin{proof}
Let $\si\in\ccinterval{\si_*+\eps}{\si^*-\eps}$ be given. Since the effective potential $H_\si$ is strongly convex and strictly decreasing on the interval $\ocinterval{-\infty}{X_-\at{\si}}$,
Lemma \ref{Lem:Muckenhoupt.LogConcave} provides
\begin{align*}
C_M^-\at{\ga_{\si,-},X_-\at{\si}}&=\sup_{x\leq X_-\at{\si}}
\at{\int_x^{X_-\at\si}\exp\at{+\frac{H_\si\at{y}}{\nu^2}}\dint{y}}
\at{\int_{-\infty}^x\exp\at{-\frac{H_\si\at{y}}{\nu^2}}\dint{y}}
\\&\leq
\nu^2\sup\limits_{x\leq X_-\at{\si}}\frac{x-X_-\at{\si}}{H_\si^\prime\at{x}}
=
\nu^2\sup\limits_{x\leq X_-\at{\si}}\frac{x-X_-\at{\si}}{H^\prime\at{x}-H^\prime\at{X_-\at{\si}}}
\\&\leq
\frac{\nu^2}{\inf_{x\leq X_-\at{\si}}H^{\prime\prime}\at{x}}\leq
\frac{\nu^2}{\inf_{x\leq X_-\at{\si^*-\eps}}H^{\prime\prime}\at{x}} =C\nu^2.
\end{align*}
 Moreover, $H_\si$ is 
strictly increasing on the interval $\ccinterval{X_-\at{\si}}{X_0\at{\si}}$, and thus we estimate
\begin{align*}
C_M^+\at{\ga_{\si,-},X_-\at{\si}}&=
\sup_{x\in\ccinterval{X_-\at{\si}}{X_0\at{\si}}}
\at{\int_{X_-\at\si}^{x}\exp\at{+\frac{H_\si\at{y}}{\nu^2}}\dint{y}}
\at{\int_x^{X_0\at{\si}}\exp\at{-\frac{H_\si\at{y}}{\nu^2}}\dint{y}}
\\&\leq
\sup_{x\in\ccinterval{X_-\at{\si}}{X_0\at{\si}}}
\at{\exp\at{+\frac{H_\si\at{x}}{\nu^2}}\bat{x-X_-\at{\si}}}
\at{\exp\at{-\frac{H_\si\at{x}}{\nu^2}}\bat{X_0\at{\si}-x}}
\\&=
\sup_{x\in\ccinterval{X_-\at{\si}}{X_0\at{\si}}}
\bat{x-X_-\at{\si}}
\bat{X_0\at{\si}-x}
 \\&\leq 
\bat{X_0\at{\si}-X_-\at{\si}}^2\leq{C}. 
\end{align*}
From Lemma \ref{Lem:MD.PoincareVersusMuckenhoupt} we now conclude that
\begin{align*}
C_P\at{\ga_{\si,-}}\leq\max\big\{C\nu^2,\,C\big\},
\end{align*}
and the corresponding estimate for $\ga_{\si,+}$ follows by symmetry.

\end{proof}
\begin{lemma}[Poincar\'e constant of $\ga_\si$ if $H_\si$ is a single-well potential]
\label{Lem:MD.Muckenhoupt.B}
For each $\eps>0$ there exists a constant $C$, which depends only on $\eps$ and $H$, such that
\begin{align*}
C_P\at{\ga_{\si}}\leq C
\end{align*}
holds for all $0<\nu\leq{1}$ and $\si\in\ccinterval{\si_*-1/\eps}{\si_*}\cup\ccinterval{\si^*}{\si^*+1/\eps}$.
\end{lemma}
\begin{proof}
Let $\overline{X}>x_*$ be arbitrary but fixed and choose $\si\in\ccinterval{\si_*-1/\eps}{\si_*}$. The potential $H_\si$ is strongly convex and strictly decreasing on the interval
$\ocinterval{-\infty}{X_-\at{\si}}$, and hence we show
\begin{align*}
C_M^-\at{\ga_{\si},X_-\at{\si}}&\leq
\frac{\nu^2}{\inf_{x\leq X_-\at{\si_*}}H^{\prime\prime}\at{x}} =C\nu^2
\end{align*}
as in the proof of Lemma \ref{Lem:MD.Muckenhoupt.A}. In order to estimate
\begin{align*}
C_M^+\at{\ga_{\si},X_-\at{\si}}&=
\sup_{x\geq X_-\at{\si}}
\at{\int_{X_-\at\si}^x\exp\at{+\frac{H_\si\at{y}}{\nu^2}}\dint{y}}
\at{\int_{x}^{+\infty}\exp\at{-\frac{H_\si\at{y}}{\nu^2}}\dint{y}}
\end{align*}
we notice that $H_\si$ is strongly 
convex and strictly increasing on $\cointerval{\overline{X}}{+\infty}$, 
see Figure \ref{Fig:Muckenhoupt}. We therefore find
\begin{align*}
\inf_{x\geq \overline{X}} \frac{H_\si^\prime\at{x}}{x-\overline{X}}\geq\inf_{x\geq \overline{X}}
 \frac{H_\si^\prime\at{\bar{X}}+c\bat{x-\overline{X}}}{x-\overline{X}}\geq{c},
\end{align*}
with $c:=\inf_{x\geq\overline{X}}H^{\prime\prime}\at{x}>0$, and Lemma \ref{Lem:Muckenhoupt.LogConcave} yields
\begin{align*}
\sup\limits_{x\geq\overline{X}}\at{\int_{\overline{X}}^x\frac{1}{\ga_\si\at{y}}\dint{y}}\at{\int_x^{+\infty}\ga_\si\at{y}\dint{y}}\leq{C}\nu^2.
\end{align*}
For $x\geq \overline{X}$ we therefore obtain
\begin{align*}
\at{\int_{X_-\at{\si}}^x\frac{1}{\ga_\si\at{y}}\dint{y}}\at{\int_x^{+\infty}\ga_\si\at{y}\dint{y}}\leq{C}_\si+C\nu^2,
\end{align*}
where
\begin{align*}
{C}_\si:=\at{\int_{X_-\at{\si}}^{\overline{X}}\frac{1}{\ga_\si\at{y}}\dint{y}}\at{\int_{\overline{X}}^{+\infty}\ga_\si\at{y}\dint{y}}.
\end{align*}
Moreover, for $x\in\ccinterval{X_-\at{\si}}{\overline{X}}$ we estimate
\begin{align*}
\at{\int_{X_-\at{\si}}^x\frac{1}{\ga_\si\at{y}}\dint{y}}\at{\int_x^{+\infty}\ga_\si\at{y}\dint{y}}
&\leq
\at{\int_{X_-\at{\si}}^x\frac{1}{\ga_\si\at{y}}\dint{y}}\at{\int_x^{\overline{X}}\ga_\si\at{y}\dint{y}}+C_\si
\\&
\leq \bat{x-X_-\at{\si}}\at{\overline{X}-x}+C_\si
\\&\leq \bat{\overline{X}-X_-\at{\si}}^2+C_\si\leq C+C_\si,
\end{align*}
where we used that $\ga_\si$ is strictly decreasing on $\ccinterval{X_-\at{\si}}{\overline{X}}$ and that 
the assumed bounds for $\si$ as well as the monotonicity of $X_-$ guarantee
$X_-\at{\si_*-1/\eps}\leq X_-\at{\si}<X_-\at{\si_*}=x_{**}<\overline{X}$. Combining all estimates derived so far with
Lemma \ref{Lem:MD.PoincareVersusMuckenhoupt} gives
\begin{align*}
C_P\at{\ga_\si}\leq C\at{1+\nu^2}+C_\si,
\end{align*}
and thus it remains to bound $C_\si$. To this end we employ the
monotonicity properties of $H_\si$ and $H_\si^\prime$ to find
\begin{align*}
C_\si&=\at{\int_{X_-\at{\si}}^{\overline{X}}\exp\at{\frac{H_\si\at{y}-H_\si\at{\overline{X}}}{\nu^2}}\dint{y}
}\at{\int_{\overline{X}}^{+\infty}\exp\at{\frac{H_\si\at{\overline{X}}-H_\si\at{y}}{\nu^2}}\dint{y}}
\\&\leq
\bat{\overline{X}-X_-\at{\si}}\int_{\overline{X}}^{+\infty}\exp\at{\frac{-H_\si^\prime\at{\overline{X}}\at{y-\overline{X}}}{\nu^2}}\dint{y}\leq C\frac{\nu^2}{H_\si^\prime\at{\overline{X}}}\leq{C}.
\end{align*}
The discussion in the case of $\si\in\ccinterval{\si^*}{\si^*+1/\eps}$ is analogous.
\end{proof}

\begin{lemma}[Poincar\'e constant of $\ga_\si$ if $H_\si$ is a supercritical double-well potential]
\label{Lem:MD.Muckenhoupt.C}
For each $\eps$ with $0<\eps<\min\{\si_\#-\si_*,\,\si^*-\si^\#\}$ there exist constants $\alpha$ and $C$ which depend only on $\eps$ and $H$ such that
\begin{align*}
C_P\at{\ga_{\si}}\leq C\tau^{\alpha-1}
\end{align*}
holds for all $\si\in\ccinterval{\si_*}{\si_\#-\eps}\cup\ccinterval{\si^\#+\eps}{\si^*}$ and 
all sufficiently small $\nu>0$.
\end{lemma}
\begin{proof} 
By symmetry and continuity, it is sufficient to consider the case  
$\si\in\ocinterval{\si_*}{\si_\#-\eps}$. As in the proof of Lemma \ref{Lem:MD.Muckenhoupt.A} we first estimate
\begin{align*}
C_M^-\at{\ga_{\si},X_-\at{\si}}&\leq
\frac{\nu^2}{\inf_{x\leq X_-\at{\si_\#}}H^{\prime\prime}\at{x}} =C_0\nu^2.
\end{align*}
Afterwards we choose $\ol{X}$ sufficiently large such that 
\begin{align*}
H_\si\at{X_0\at{\si}}\leq H_\si\at{\ol{X}},\qquad X_+\at{\si}\leq\ol{X}
\end{align*}
holds for all $\si\in\ocinterval{\si_*}{\si_\#-\eps}$, and as in the proof of 
Lemma \ref{Lem:MD.Muckenhoupt.B} we verify that
\begin{align*}
\sup\limits_{x\geq \overline{X}}\at{\int_{\overline{X}}^x\frac{1}{\ga_\si\at{y}}\dint{y}}\at{\int_x^{+\infty}\ga_\si\at{y}\dint{y}}\leq{C_1}\nu^2,
\end{align*}
where we used that  $H_\si$ is
strictly increasing and convex in the interval $\cointerval{\ol{X}}{\infty}$, see Figure \ref{Fig:Muckenhoupt}.
Due to the monotonicity properties of $H_\si$ and $H_\si^\prime$, and thanks to our choice of $\ol{X}$, we further 
obtain
\begin{align*}
\at{\int_{X_-\at{\si}}^{\overline{X}}\frac{1}{\ga_\si\at{y}}\dint{y}}\at{\int_{\overline{X}}^{+\infty}\ga_\si\at{y}\dint{y}}&=\at{\int_{X_-\at{\si}}^{\overline{X}}\frac{\ga_\si\at{\ol{X}}}{\ga_\si\at{y}}\dint{y}}\at{\int_{\overline{X}}^{+\infty}\frac{\ga_\si\at{y}}{\ga_\si\at{\ol{X}}}\dint{y}}
\\&\leq
\at{\overline{X}-X_-\at{\si}}\int_{\overline{X}}^{+\infty}\exp\at{\frac{-H_\si^\prime\bat{\overline{X}}\at{y-\overline{X}}}{\nu^2}}\dint{y
}\\&
\leq C_2\nu^2
\end{align*}
as well as
\begin{align*}
\at{\int_{X_-\at{\si}}^{X_+\at{\si}}\frac{1}{\ga_\si\at{y}}\dint{y}}\at{\int_{X_0\at{\si}}^{\overline{X}}\ga_\si\at{y}\dint{y}}&\leq
\at{C\exp\at{\frac{H_\si\bat{X_0\at\si}}{\nu^2}}}
\at{C\exp\at{-\frac{H_\si\bat{X_+\at\si}}{\nu^2}}}
\\&\leq C_3\exp\at{\frac{h_+\at{\si}}{\nu^2}}.
\end{align*}
We now abbreviate
\begin{align*}
f_\si\at{x}:=\at{\int_{X_-\at{\si}}^{x}\frac{1}{\ga_\si\at{y}}\dint{y}}\at{\int_{x}^{+\infty}\ga_\si\at{y}\dint{y}}
\end{align*}
and discuss four different cases: 
With $x\geq\overline{X}$ we estimate
\begin{align*}
f_\si\at{x}
&=
\at{\int_{X_-\at{\si}}^{\overline{X}}\frac{1}{\ga_\si\at{y}}\dint{y}}\at{\int_{x}^{+\infty}\ga_\si\at{y}\dint{y}}+
\at{\int_{\overline{X}}^x\frac{1}{\ga_\si\at{y}}\dint{y}}\at{\int_{x}^{+\infty}\ga_\si\at{y}\dint{y}}
&\leq \bat{C_2+C_1}\nu^2.
\end{align*}
For 
$x\in\ccinterval{X_+\at{\si}}{\overline{X}}$ we find
\begin{align*}
f_\si\at{x}
&\leq
\at{\int_{X_-\at{\si}}^{x}\frac{1}{\ga_\si\at{y}}\dint{y}}\at{\int_{x}^{\overline{X}}\ga_\si\at{y}\dint{y}}+C_2\nu^2 \\
&\leq 
\at{\int_{X_+\at{\si}}^{x}\frac{1}{\ga_\si\at{y}}\dint{y}}\at{\int_{x}^{\overline{X}}\ga_\si\at{y}\dint{y}}
+ C_3\exp\at{\frac{h_+\at{\si}}{\nu^2}}+C_2\nu^2,
\end{align*}
and since $H_\si$ is strictly increasing on the interval $\ccinterval{X_+\at{\si}}{\overline{X}}$, there exists a constant $C_4$
such that
\begin{align*}
f_\si\at{x}\leq
C_4\at{1+\nu^2+\exp\at{\frac{h_+\at{\si}}{\nu^2}}}.
\end{align*}
In the case of $x\in\ccinterval{X_0\at{\si}}{X_+\at\si}$ we verify
\begin{align*}
f_\si\at{x}
&\leq
\at{\int_{X_-\at{\si}}^{x}\frac{1}{\ga_\si\at{y}}\dint{y}}\at{\int_{x}^{\overline{X}}\ga_\si\at{y}\dint{y}}+C_2\nu^2
\\&\leq C_3\exp\at{\frac{h_+\at{\si}}{\nu^2}}+C_2\nu^2,
\end{align*}
and for $x\in\ccinterval{X_-\at\si}{X_0\at{\si}}$ we finally get
\begin{align*}
f_\si\at{x}
&\leq
\at{\int_{X_-\at{\si}}^{x}\frac{1}{\ga_\si\at{y}}\dint{y}}\at{\int_{x}^{\overline{X}}\ga_\si\at{y}\dint{y}}+C_2\nu^2
\\
&\leq
\at{\int_{X_-\at{\si}}^{x}\frac{1}{\ga_\si\at{y}}\dint{y}}\at{\int_{x}^{X_0\at{\si}}\ga_\si\at{y}\dint{y}}
+C_3\exp\at{\frac{h_+\at{\si}}{\nu^2}}+C_2\nu^2
\\&\leq C_5\at{1+\nu^2+\exp\at{\frac{h_+\at{\si}}{\nu^2}}},
\end{align*}
where the last inequality holds since 
$H_\si$ is strictly increasing on the interval $\ccinterval{X_-\at{\si}}{X_0\at{\si}}$.
Taking the supremum over all $x\geq X_-\at{\si}$ we now obtain,
thanks to Lemma \ref{Lem:MD.PoincareVersusMuckenhoupt}, the bound
\begin{align*}
C_P\at{\ga}\leq \max\big\{C_M^-\at{\ga_{\si},X_-\at{\si}},\,
C_M^+\at{\ga_{\si},X_-\at{\si}}\big\}&\leq C\at{1+\nu^2+\exp\at{\frac{h_+\at{\si}}{\nu^2}}}.
\end{align*}
The claim now follows 
because Assumption \ref{Ass:Tau} implies 
\begin{align*}
\tau=\exp\at{-\frac{h_\#\bat{1+\Do{1}}}{\nu^2}}
\end{align*}
and since we have $0\leq h_+\at{\si}\leq h_\#-c$ for some $c>0$ depending on $\eps$.
\end{proof}

\subsubsection{Estimates for the mass near the stable peak positions}\label{sect:MassInPeaks}
%
In order to establish the mass-dissipation estimates,  we 
introduce the dissipation functional
\begin{align}
\label{Eqn:DefDissipationFunctional}
\calD_\si\at\varrho&:=\int_\Rset
\frac{\Bat{\nu^2\partial_x\varrho\at{x}+\bat{H^\prime\at{x}-\si}\varrho\at{x}}^2}{\varrho\at{x}}\dint{x}.
\end{align}
This definition is consistent with \eqref{Eqn:DefDissipation} and implies 
\begin{align}
\label{Eqn:MD.AuxFormula}
\calD_\si\at{\varrho}=4\nu^4\int_\Rset\bat{\partial_xw}^2\ga_\si\dint{x},\quad
\int_{J}\varrho\dint{x}=\int_Jw^2\ga_\si\dint{x}
\qquad\text{for}\quad
\varrho=w^2\ga_\si.
\end{align}
Our first mass-dissipation estimate implies for each $\si\in\oointerval{\si_*}{\si^*}$ that the mass is concentrated
near the stable peak positions $X_-\at\si$ and $X_+\at\si$ provided that the dissipation is sufficiently small.
\begin{lemma}[upper bound for the mass outside the stable peaks]
\label{Lem:MD.MassOutsideTwoPeaks}
For each $\eps$ and any $\eta$
with
\begin{align*}
0<\eps<\tfrac12\at{\si^*-\si_*},\qquad
0< \eta <  \min\Big\{ x^*-X_-\at{\si^*-\eps},\,X_+\at{\si_*+\eps}-x_*\Big\}
\end{align*}
there exist constants $\alpha$ and $C$, which depend only on $\eps$ and $\eta$, such that
\begin{align*}
\int_{\Rset\setminus B_{\eta}\at{X_-\at\si}\cup B_{\eta}\at{X_+\at\si}}\varrho\at{x}\dint{x}\leq C\tau^\alpha\at{\frac{\calD_\si\at\varrho}{\tau}+1}
\end{align*} 
for all $\si\in\oointerval{\si_*+\eps}{\si^*-\eps}$, any smooth probability measure $\varrho$, and all sufficiently small $\nu>0$.
\end{lemma}

\begin{proof}  
Due to the imposed bounds for $\eta$, the definitions of $\ga_\si$ and $I_{\si,\pm}$ -- see equations \eqref{Eqn:DefEquilibriumDensities} and \eqref{Eqn:DefSiIntervals} as well as Figure \ref{Fig:Equilibria} -- imply the existence of constants $C>0$ and $0<\alpha<1$ such that
\begin{align*}
\frac{\D \int_{I_{\si,-}\setminus B_\eta\at{X_-\at\si}}\ga_\si\at{x}\dint{x}}
{\D \int_{I_{\si,-}}\ga_\si\at{x}\dint{x}}
+
\frac{\D \int_{I_{\si,+}\setminus B_\eta\at{X_+\at\si}}\ga_\si\at{x}\dint{x}}{
\D \int_{I_{\si,+}}\ga_\si\at{x}\dint{x}}
\leq C\tau^\alpha
\end{align*}
holds for all sufficiently small $\nu>0$. In other words, the mass of the equilibrium density $\ga_\si/z_\si$ is 
almost completely located near the stable peak positions $X_-\at{\si}$ and $X_+\at{\si}$, see Figure \ref{Fig:Equilibria}. Using Lemma \ref{Lem:MD.AuxEstimate} twice 
with 
\begin{align*}
\ga=\ga_{\si,\pm},\qquad I=I_{\si,\pm}, \qquad w^2=\varrho/\ga_\si, \qquad J=I_{\si,\pm}\setminus
B_\eta\bat{X_\pm\at\si}
\end{align*}
we therefore arrive -- see also \eqref{Eqn:MD.AuxFormula} and recall that $\int_{I_{\si,\pm}}w^2\ga_\si\dint{x}\leq\int_\Rset\varrho\dint{x}=1$ -- 
at the estimate
\begin{align*}
\int_{\Rset\setminus \bat{ B_\eta\at{X_-\at\si} \cup B_\eta\at{X_+\at\si}} }\varrho\at{x}\dint{x}
\leq 2\Bat{C_P\at{\ga_{\si,-}}+ C_P\at{\ga_{\si,+}}}\frac{\calD_\si\at{\varrho}}{4\nu^4}+
C\tau^\alpha.
\end{align*}
The assertion now follows since Lemma \ref{Lem:MD.Muckenhoupt.A} provides
$
C_P\at{\ga_{\sigma,\pm}}\leq C$
and because
 Assumption \ref{Ass:Tau} yields $\nu^{-4}\tau\leq \tau^{\alpha}$ for all sufficiently small $\nu>0$. 
\end{proof}

The second mass-dissipation estimate applies to strictly supercritical $\sigma$ and reveals
that the dissipation controls the mass near the global minimizer of $H_\si$, which is 
$X_-\at{\si}$ or $X_+\at{\si}$ for $\si<\si_\#$ or $\si>\si^\#$, respectively, see Remark \ref{Rem:PropertiesOfXi}.
\begin{lemma}[upper bound for the mass outside the most stable peak]
\label{Lem:MD.MassOutsideOnePeak}
For each $\eps>0$ and any $\eta$ 
with
\begin{align*}
0< \eta <  \Big\{ x^*-X_-\bat{\si^\#-\eps},\,X_+\bat{\si_\#+\eps}-x_*\Big\}
\end{align*}
there exist  constants $\alpha$ and $C$, which depend only on $\eps$ and $\eta$, such that the implications
\begin{align*}
\si\in\ccinterval{\si^{\#}+\eps}{\si^\#+1/\eps}\quad\implies\quad
\int_{ B_{\eta}\at{X_+\at\si}}\varrho\dint{x}\geq 1-C\tau^\alpha\at{\frac{\calD_\si\at\varrho}{\tau}+1}
\end{align*}
and
\begin{align*}
\si\in\ccinterval{\si_{\#}-1/\eps}{\si_{\#}-\eps}\quad\implies\quad
\int_{B_{\eta}\at{X_-\at\si}}\varrho\dint{x}\geq 1-C\tau^\alpha\at{\frac{\calD_\si\at\varrho}{\tau}+1}
\end{align*}
hold for any smooth probability measure $\varrho$ and all sufficiently small $\nu$.
\end{lemma}

\begin{proof} We only prove the first implication; the second one follows by analogous arguments.
By Lemma \ref{Lem:MD.Muckenhoupt.B} and Lemma \ref{Lem:MD.Muckenhoupt.C}, there exist positive constants $C$ and $\alpha$ such that
\begin{align*}
\frac{C_P\at{\ga_\si}\tau}{\nu^4}\leq C\tau^\alpha.
\end{align*}
Making $\alpha$ smaller and $C$ larger (if necessary) we can also assume -- see again Figure \ref{Fig:Equilibria} -- that
\begin{align*}
\frac{\D\int_{\Rset\setminus B_\eta\at{X_+\at\si}}\ga_\si\at{x}\dint{x}}
{\D\int_{\Rset}\ga_\si\at{x}\dint{x}}
\leq C\tau^\alpha
\end{align*}
for all sufficiently small $\nu$. The assertion now follows by applying
Lemma \ref{Lem:MD.AuxEstimate} with
$\ga=\ga_\si$, $I=\Rset$, and $J=\Rset\setminus B_\eta\at{X_+\at\si}$.
\end{proof}
%
%
\subsection{Dynamical stability of peaks}\label{sect:PeakStability}
%
%
In our convergence proof we have 
to guarantee that any solution to the nonlocal Fokker-Planck equation \eqref{Eqn:FP1}+\eqref{Eqn:FP2}
can -- at each sufficiently large time $t$ and depending on the value of $\si\at{t}$ -- be approximated by either two or one
stable peaks located at $X_-\bat{\si\at{t}}$ and/or
$X_+\bat{\si\at{t}}$. In view of the mass-dissipation 
estimates from \S\ref{sect:MassDiss} it is clear that such an approximation is possible if the dissipation is small but our approach lacks pointwise estimates for $\calD\at{t}$. 
As already mention in \S\ref{sect:overview}, we therefore control the approximation error by 
certain combinations of the moment $\xi$ and the 
partial masses $m_-$, $m_0$, and $m_+$ because these quantities can be bounded pointwise in time. We also recall
that $\xi$ and the masses $m_i$ are defined in \eqref{Eqn:Def.MomentXi} and \eqref{Eqn:Def.PartialMasses}, respectively, and that $m_-\at{t}+m_0\at{t}+m_+\at{t}=1$ holds by construction. 
\par
In this section we derive upper bounds for $\xi\at{t}+m_0\at{t}$ and discuss
the evolution of $m_-$ and $m_+$ afterwards in \S\ref{sect:Montonicity}. 
We start with some auxiliary results which hold pointwise in time and do not rely on dynamical arguments.
\begin{lemma}[dissipation bounds $\xi$]
\label{Lem:PS.AuxFormula}
There exists a constant $C$ such that 
\begin{align*}
\xi\at{t}\leq \calD\at{t}+C\nu^2
\end{align*}
holds for all $t\geq0$ and $\nu>0$.
\end{lemma}
\begin{proof}
From the definition of the dissipation, see 
 \eqref{Eqn:DefDissipation}, we infer that 
\begin{align*}
\calD\at{t}&=\int_\Rset\at{\bat{H^\prime\at{x}-\si}^2\varrho\pair{t}{x}+\nu^4\frac{\bat{\partial_x \varrho\pair{t}{x}}^2}{\varrho\pair{t}{x}}+2\nu^2 \bat{H^\prime\at{x}-\si} \partial_x \varrho\pair{t}{x}}\dint{x}
\\&\geq\int_\Rset\bat{H^\prime\at{x}-\si}^2\varrho\pair{t}{x}\dint{x}-2\nu^2
\int_\Rset H^{\prime\prime}\at{x} \varrho\pair{t}{x}\dint{x},
\end{align*}
and this gives the desired result since Assumption \ref{Ass:Potential} 
implies  $\abs{H^{\prime\prime}\at{x}}\leq{C}$ for all $x\in\Rset$.
\end{proof}
\begin{lemma}[relations between $\ell$, $\si$, and $m_\pm$]
\label{Lem:XiControlsL}
For each $\eps$ with $0<\eps<\tfrac12\at{\si^*-\si_*}$ there exists a constant $C$,
which depends on $\eps$ but not on $\nu$, such that the implications
\begin{align*}
\si\at{t}\in\ocinterval{-\infty}{\si^*-\eps}\quad\implies\quad
\babs{\ell\at{t}-X_-\bat{\si\at{t}}}\leq C\sqrt{\xi\at{t}+m_0\at{t}+m_+\at{t}}
\end{align*}
as well as
\begin{align*}
\si\at{t}\in\ccinterval{\si_*+\eps}{\si^*-\eps}\quad\implies\quad
\babs{\ell\at{t}-m_-\at{t}X_-\bat{\si\at{t}}-m_+\at{t}X_+\bat{\si\at{t}}}\leq C\sqrt{\xi\at{t}+m_0\at{t}}
\end{align*}
and
\begin{align*}
\si\at{t}\in\cointerval{\si_*+\eps}{+\infty}\quad\implies\quad
\babs{\ell\at{t}-X_+\bat{\si\at{t}}}\leq C\sqrt{\xi\at{t}+m_-\at{t}+m_0\at{t}}
\end{align*}
hold for all $0\leq{t}\leq{T}$ and all $\nu>0$.
\end{lemma}
\begin{proof}
We only prove the first implication; the derivations of the second and the third one are similar.
From the definitions of the partial masses $m_i$ with $i\in\{-,0,+\}$, see again \eqref{Eqn:Def.PartialMasses}, 
and the functions $X_i$, see Remark \ref{Rem:PropertiesOfXi}, we infer that 
\begin{align*}
\ell\at{t}-X_-\bat{\si\at{t}}=
-\bat{m_0\at{t}+m_+\at{t}}X_-\at{\si\at{t}}+\int_{-\infty}^{x^*}\bat{x-X_-\at{\si\at{t}}}\varrho\pair{t}{x}\dint{x}+
\int_{x^*}^{+\infty}x\varrho\pair{t}{x}\dint{x}.
\end{align*}
Thanks to $\si\at{t}\leq \si^*-\eps$ and the uniform bounds for $\abs{\si\at{t}}$, see Lemma \ref{Lem:PropertiesOfSolutions}, we have
\begin{align*}
-C\leq X_-\at{\si\at{t}}\leq x^*-c
\end{align*}
for some positive constants $c$ and $C$. In view of the properties of $H^{\prime}$ and $H^{\prime\prime}$, see Assumption \ref{Ass:Potential}, we therefore get
\begin{align*}
\abs{x-X_-\bat{\si\at{t}}}\leq C\abs{H^\prime\at{x}-\si\at{t}}\qquad \text{for all}\quad x\leq x^*
\end{align*}
and hence
\begin{align*}
\abs{\ell\at{t}-X_-\bat{\si\at{t}}}\leq C\at{m_0\at{t}+m_+\at{t}+\int_{-\infty}^{x^*}\babs{H^\prime\at{x}-\si\at{t}}\varrho\pair{t}{x}\dint{x}+
\int_{x^*}^{+\infty}x\varrho\pair{t}{x}\dint{x}}.
\end{align*}
Moreover, H\"older's inequality yields
\begin{align*}
\int_{-\infty}^{x^*}\babs{H^\prime\at{x}-\si\at{t}}\varrho\pair{t}{x}\dint{x}\leq
\at{m_-\at{t}\int_{-\infty}^{x^*}\babs{H^\prime\at{x}-\si\at{t}}^2\varrho\pair{t}{x}\dint{x}}^{1/2}
\leq\sqrt{\xi\at{t}}
\end{align*}
thanks to $m_-\at{t}\leq1$, as well as
\begin{align*}
\int_{x^*}^{+\infty}x\varrho\pair{t}{x}\dint{x}\leq
\bat{m_0\at{t}+m_+\at{t}}^{1/2}\at{\int_{x^*}^{+\infty}x^2\varrho\pair{t}{x}\dint{x}}^{1/2}
\leq{C}\sqrt{m_0\at{t}+m_+\at{t}}
\end{align*}
due to the uniform moment estimates from Lemma \ref{Lem:PropertiesOfSolutions}. The first implication now follows
from these results since we have $m_i\at{t}\leq \sqrt{m_i\at{t}}$ and because $\sqrt{a}+\sqrt{b}\leq \sqrt2\sqrt{a+b}$ holds for all $a,b\geq 0$.
\end{proof}
The assertions and the proof of Lemma \ref{Lem:XiControlsL} can easily be generalized to other moments. 
\begin{remark}
\label{Rem:XiControlsMoments}
For any continuous moment weight $\psi$ that grows at most linearly and each $\eps$ as in Lemma \ref{Lem:XiControlsL} 
there exists a constant $C$, which depends on $\eps$ and $\psi$ but not on $\nu$, such that
\begin{align*}
\abs{\int_\Rset\psi\at{x}\varrho\pair{t}{x}\dint{x}-\sum_{j\in\{-,+\}}m_j\at{t}\psi\Bat{X_j\bat{\si\at{t}}}}\leq C\sqrt{\xi\at{t}+m_0\at{t}}
\end{align*}
holds for all sufficiently small $\nu$ as long as $\si\at{t}\in\ccinterval{\si_*+\eps}{\si^*-\eps}$. Moreover, similar results 
hold in the cases $\si\at{t}\in\ocinterval{-\infty}{\si^*-\eps}$ and $\si\at{t}\in\cointerval{\si_*+\eps}{+\infty}$.
\end{remark}
%
%
\subsubsection{Evolution of the moment \texorpdfstring{$\xi$}{}}
%
%
We next study the dynamics of the moment $\xi$ from \eqref{Eqn:Def.MomentXi} and establish an upper bound 
in terms of the auxiliary mass
\begin{align*}
m_\eta\at{t}&:=\int_{x^*-\eta}^{x_*+\eta} 
\varrho\pair{t}{x}\dint{x},
\end{align*}
where $\eta>0$ denotes a free parameter. To this end we derive the moment balance for $\xi$ from the Fokker-Planck equation \eqref{Eqn:FP1}+\eqref{Eqn:FP2}, estimate the arising integrals to obtain appropriate bounds for
$\dot\xi$, and finally apply simple ODE arguments.
\begin{lemma}[pointwise estimate for $\xi$] 
\label{Lem:PS.BoundForF}
For each $\eta>0$ there exists a constant $C$, which depends on $\eta$ but not on $\nu$, such that
\begin{align*}
\sup\limits_{t\in\ccinterval{t_1}{t_2}}\xi\at{t}\leq \xi\at{t_1} + C\Bat{\nu^2+\sup\limits_{t\in\ccinterval{t_1}{t_2}}m_\eta\at{t}}
\end{align*}
holds for all $0\leq t_1<t_2<T$ and all sufficiently small $\nu>0$.
\end{lemma}
\begin{proof}
Using the abbreviation $\psi\pair{t}{x}:=\bat{H^\prime\at{x}-\si\at{t}}^2$
as well as
\eqref{Eqn:FP2} and integration by parts, we easily compute
\begin{align*}
\tau\dot{\xi}\at{t}&=-2\tau\dot\si\at{t}\int_{\Rset} 
\bat{H^\prime\at{x}-\si\at{t}}\varrho\pair{t}{x}\dint{x}+
\int_{\Rset} 
\psi\pair{t}{x}\tau\partial_t\varrho\pair{t}{x}\dint{x}
\\
&=2\tau^2\dot\si\at{t}\dot{\ell}\at{t}+
\int_{\Rset} 
\psi\pair{t}{x}\partial_x\Bat{\nu^2\partial_x\varrho\pair{t}{x}+
\bat{H^\prime\at{x}-\si\at{t}}\varrho\pair{t}{x}}\dint{x}
\\
&=2\tau^2\dot\si\at{t}\dot{\ell}\at{t}
+\nu^2\int_{\Rset} 
\psi^{\prime\prime}\pair{t}{x}\varrho\pair{t}{x}\dint{x}
-2
\int_{\Rset} 
H^{\prime\prime}\at{x}\psi\pair{t}{x}\varrho\pair{t}{x}\dint{x},
\end{align*}
as well as
\begin{align*}
\tau\dot\si\at{t}=\tau\ddot{\ell}\at{t}+
\nu^2\int_{\Rset} 
H^{\prime\prime\prime}\at{x}\varrho\pair{t}{x}\dint{x}
-
\int_{\Rset} 
H^{\prime\prime}\at{x}\bat{H^\prime\at{x}-\si\at{t}}\varrho\pair{t}{x}\dint{x}.
\end{align*}
In view of 
\begin{align*}
\abs{H^{\prime\prime}\at{x}}+\abs{H^{\prime\prime\prime}\at{x}}\leq C
,\qquad 
\abs{ \psi^{\prime\prime}\pair{t}{x}}+\abs{H^{\prime}\at{x}}\leq C\at{1+\abs{x}}
,\qquad 
\abs{ \psi\pair{t}{x}}\leq C\at{1+x^2}
\end{align*} 
and
\begin{align*}
\nabs{\dot\ell\at{t}}+\nabs{\ddot\ell\at{t}}+\abs{\si\at{t}}+\int_{\Rset}\at{1+x^2}\varrho\pair{t}{x}\dint{x}\leq{C},
\end{align*} 
see Assumption \ref{Ass:Potential}, Assumption \ref{Ass:Constraint}, and Lemma \ref{Lem:PropertiesOfSolutions}, we therefore find
\begin{align*}
\tau\dot{\xi}\at{t}
&\leq C\at{\nu^2+\tau}-
2
\int_{\Rset} 
H^{\prime\prime}\at{x}\psi\pair{t}{x}\varrho\pair{t}{x}\dint{x}.
\end{align*}
Moreover, since $H$ is smooth by Assumption \ref{Ass:Potential} there exist constants $c$ and $C$ such that
\begin{align*}
H^{\prime\prime}\at{x}\geq c \quad\text{for all}\quad
x\in\Rset\setminus{\ccinterval{x^*-\eta}{x_*+\eta}},\qquad
\abs{H^{\prime\prime}\at{x}}\leq C\quad\text{for all}\quad x\in \ccinterval{x^*-\eta}{x_*+\eta},
\end{align*}
and this 
implies
\begin{align*}
\int_{\Rset} 
H^{\prime\prime}\at{x}\psi\pair{t}{x}\varrho\pair{t}{x}\dint{x}
&\geq
\int\limits_{\Rset\setminus \ccinterval{x^*-\eta}{x_*+\eta}}
H^{\prime\prime}\at{x}\psi\pair{t}{x}\varrho\pair{t}{x}\dint{x}+
\int\limits_{x^*-\eta}^{x_*+\eta}
H^{\prime\prime}\at{x}\psi\pair{t}{x}\varrho\pair{t}{x}\dint{x}
\\
&\geq
c\int\limits_{\Rset\setminus \ccinterval{x^*-\eta}{x_*+\eta}}\psi\pair{t}{x}\varrho\pair{t}{x}\dint{x}-
Cm_\eta\at{t} = c\,\xi\at{t}-C M_\eta,
\end{align*}
where $M_\eta$ is shorthand for $\sup_{t\in\ccinterval{t_1}{t_2}}m_\eta\at{t}$. 
Combining all partial results we finally get 
\begin{align*}
\tau\dot{\xi}\at{t}&\leq -c \xi\at{t}+C\at{\nu^2 + \tau+M_\eta}
\end{align*}
for all $t\in\ccinterval{t_1}{t_2}$, and the comparison principle for scalar ODE finishes the proof.
\end{proof}

\subsubsection{Conditional stability estimates}
%
%
We are now able to establish partial results on the dynamical stability of peaks. More precisely, 
assuming that the dynamical multiplier $\si\at{t}$ from \eqref{Eqn:FP2} remains confined to certain 
intervals we derive estimates
that control the evolution of the moment $\xi\at{t}+m_0\at{t}$. In the proof we employ -- 
apart from the upper bounds for $\xi$ derived in Lemma \ref{Lem:PS.BoundForF} -- 
local comparison principles for linear Fokker-Planck equations in order to show that only a very small amount of 
mass can flow into the unstable interval $\oointerval{x^*}{x_*}$. 
In this context we recall that $\norm{\varrho\pair{t}{\cdot}}_\infty\leq C\nu^{-2}$ holds for all times 
$t\geq t_*$, where  
\begin{align*}
 t_*=\nu^2\tau \quad\text{and}\quad t_*=0
\end{align*}
for generic and well-prepared initial data, respectively, see Proposition \ref{Lem:PropertiesOfSolutions} and Remark \ref{Rem:WellPreparedInitialData}.  
\begin{figure}[ht!]%
\centering{%
\includegraphics[height=.275\textwidth, draft=\figdraft]{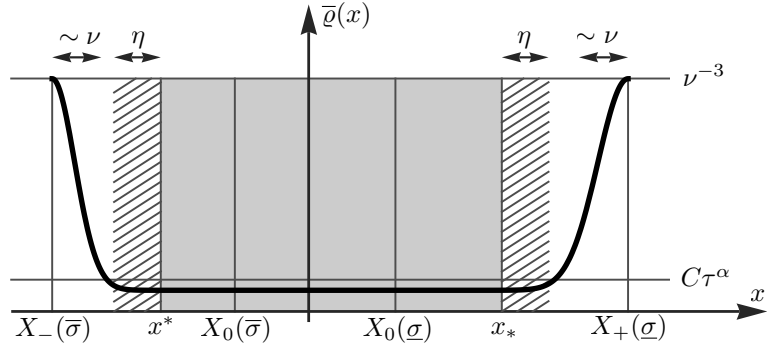}%
}%
\caption{Cartoon of the supersolution $\ol\varrho$ and the characteristic lengths from the proof of Lemma \ref{Lem:PS.NoMassInSpinodalRegion1};
the hatched regions indicate intervals of length $\eta$. The strictly decreasing and increasing branches are given by the rescaled equilibrium densities $\ga_{\ol\si}$ and $\ga_{\ul\si}$, respectively.
For $0<\nu\ll1$, $\bar\varrho$ is therefore very small in $\ccinterval{x^*-\eta}{x_*+\eta}$ and exhibits boundary layers with 
width of order $\nu$ near $X_-\at{\ol\si}$ and $X_+\at{\ul\si}$.
}%
\label{Fig:NoMassInSpinodalRegion1}
\end{figure}%
\begin{lemma}[first conditional estimate for $\xi+m_0$] 
\label{Lem:PS.NoMassInSpinodalRegion1}
For each $\eps$ with $0<\eps<\tfrac12\at{\si^*-\si_*}$ there exists a positive
constant
$C$, which depends on $\eps$ but not on $\nu$, such that
the implication

\begin{align*}
\si\at{t}\in\ccinterval{\si_*+\eps}{\si^*-\eps}
\;\;\text{for all}\;\; t\in\ccinterval{t_1}{t_2}
\quad \implies \quad
\sup_{t\in\ccinterval{t_1}{t_2}}\bat{\xi\at{t}+m_0\at{t}}\leq 
C\Bat{\xi\at{t_1}+m_0\at{t_1}+\nu^2}
\end{align*}
holds for all $\tsnu\leq t_1<t_2\leq T$ and all sufficiently small $\nu>0$.
\end{lemma}

\begin{proof}
Within this proof we regard \eqref{Eqn:FP1} as a non-autonomous but linear PDE for $\varrho$, that means
we ignore \eqref{Eqn:FP2} and regard $\si$ as a given function of time. 
\par
{\ul{\emph{Preliminaries:}}} We first choose $\ul\si,\,\ol\si\in\oointerval{\si_*}{\si^*}$ 
such that
\begin{align*}
h_+\at{\ul\si}=h_-\at{\ol\si}=\tfrac12\min\big\{h_+\at{\si_*+\eps},\;h_-\at{\si^*-\eps}\big\},
\end{align*}
and the monotonicity properties of $h_-$ and $h_+$ -- see Figure \ref{Fig:EffectivePotential} -- 
ensure that $\ul\si\in\oointerval{\si_*}{\si_*+\eps}$ and $\ol\si\in\oointerval{\si^*-\eps}{\si^*}$.
Employing the monotonicity of  $X_-$, $X_0$, and $X_+$ 
we verify -- see  Remark \ref{Rem:PropertiesOfXi}  and Figure \ref{Fig:NoMassInSpinodalRegion1} -- the order relations
\begin{align*}
X_-\bat{\si^*-\eps}<X_-\at{\ol\si}<x^*<X_0\at{\ol\si}<X_0\bat{\si^*-\eps}
\end{align*}
and
\begin{align*}
X_0\bat{\si_*+\eps}<X_0\at{\ul\si}<x_*<
X_+\at{\ul\si}<X_+\bat{\si_*+\eps}
\end{align*}
and thus we can choose $\eta>0$ sufficiently small such that
the distance between any two adjacent points in these chains is larger than $\nu$. In particular, there exists a constant $C$ which depends only on $H$ and $\eps$ such that
\begin{align}
\label{Lem:PS.NoMassInSpinodalRegion1.Eqn2}
\chi_{\ccinterval{X_-\at{\ol\si}}{X_+\at{\ul\si}}}\at{x}\leq
C\at{\bat{H^\prime\at{x}-\si}^2 + \chi_{\ccinterval{x^*}{x_*}}\at{x}}
\end{align}
holds for all $x\in\Rset$ and any $
\si\in
\ccinterval{\si_*+\eps}{\si^*-\eps}$, where $\chi_I$ denotes as usual the indicator function of the interval $I$.
\par
{\ul{\emph{Construction of a supersolution:}}} We define a local supersolution $\ol\varrho$ on the interval $\ccinterval{X_-\at{\ol\si}}{X_+\at{\ul\si}}$  
by combining rescaled versions of the monotone
branches of $\ga_{\ol\si}$ and $\ga_{\ul\si}$, where the latter are defined in \eqref{Eqn:DefEquilibriumDensities}. More precisely, we set
\begin{align*}
\ol{\varrho}\at{x}:=\nu^{-3}
\left\{
\begin{array}{lcl}
\D\exp\at{\frac{H_{\ol\si}\at{X_-\at{\ol\si}}-H_{\ol\si}\at{x}}{\nu^2}}
&\text{for}&X_-\at{\ol\si}\leq x \leq X_0\at{\ol\si},\\
\D\exp\at{\frac{-h_-\at{\ol\si}}{\nu^2}}
&\text{for}&X_0\at{\ol\si}\leq x\leq X_0\at{\ul\si},\\
\D
\exp\at{\frac{H_{\ul\si}\at{X_+\at{\ul\si}}-H_{\ul\si}\at{x}}{\nu^2}}
&\text{for}&X_0\at{\ul\si}\leq x\leq X_+\at{\ul\si}.
\end{array}
\right.
\end{align*}
Our choice of $\ol\si$ and $\ul\si$ implies that $\ol\varrho$ is continuous with
\begin{align*}
\ol{\varrho}\bat{X_-\at{\ol\si}}=\ol\varrho\bat{X_+\at{\ul\si}}=\nu^{-3}
\end{align*}
and the smallness of $\eta$ guarantees the existence of constants $\al$ and $C$ such that
\begin{align*}
\ol{\varrho}\at{x}\leq C\tau^\alpha\qquad \text{for all}\qquad x\in\ccinterval{x^*-\eta}{x_*+\eta}
\end{align*}
and all sufficiently small $\nu$.
Moreover, $\ol\varrho$ is continuously differentiable and piecewise
twice continuously differentiable with
\begin{align*}
\partial_x\Bat{\nu^2 \partial_x \ol\varrho\at{x} + \bat{H^\prime\at{x}-\si\at{t}}\ol\varrho\at{x}}  =
\left\{\begin{array}{lcl}
\bat{\ol\si-\si\at{t}}\partial_x\ol\varrho\at{x}&\text{for}& {X_-\at{\ol\si}}<x<{X_0\at{\ol\si}},\\
H^{\prime\prime}\at{x}\ol\varrho\at{x}&\text{for}& X_0\at{\ol\si}<x<X_0\at{\ul\si},\\
\bat{\ul\si-\si\at{t}}\partial_x\ol\varrho\at{x}&\text{for}& {X_0\at{\ul\si}}<x<{X_+\at{\ul\si}}.
\end{array}
\right.
\end{align*}
Combining this with 
\begin{align*}
\ul\si\leq\si\at{t}\leq \ol\si\quad\text{for}\quad
t\in\ccinterval{t_1}{t_2},\qquad
H^{\prime\prime}\at{x}\leq0\quad\text{for}\quad
x\in\ccinterval{X_0\at{\ol\si}}{X_0\at{\ul\si}}
\end{align*}
and
\begin{align*}
\partial_x\ol\varrho\at{x}\leq0\quad \text{for}\quad
x\in\ccinterval{X_-\at{\ol\si}}{X_0\at{\ol\si}}
,\qquad
\partial_x\ol\varrho\at{x}\geq0\quad \text{for}\quad
x\in\ccinterval{X_0\at{\ul\si}}{X_+\at{\ul\si}}
\end{align*}
gives
\begin{align*}
\partial_x\Bat{\nu^2 \partial_x \ol\varrho\at{x} + \bat{H^\prime\at{x}-\si\at{t}}\ol\varrho\at{x}}  \leq0 =\tau\partial_t\bar\varrho\at{x},
\end{align*}
and we conclude that $\ol\varrho$ is in fact a supersolution to the linear Fokker-Planck equation
\eqref{Eqn:FP1} on the time-space domain
$\ccinterval{t_1}{t_2}\times\ccinterval{X_-\at{\ol\si}}{X_+\at{\ul\si}}$.
\par
{\ul{\emph{Moment estimates:}}} We 
next define three solutions $\varrho_-$, $\varrho_0$, and $\varrho_+$ 
to the linear PDE \eqref{Eqn:FP1} on the time interval $\ccinterval{t_1}{t_2}$ by imposing
the initial conditions
\begin{align*}
\varrho_-\pair{t_1}{x}&=\varrho\pair{t_1}{x}\chi_{\ocinterval{-\infty}{X_-\at{\ol\si}}}\at{x},\\
\varrho_0\pair{t_1}{x}&=\varrho\pair{t_1}{x}\chi_{\ocinterval{X_-\at{\ol\si}}{X_+\at{\ul\si}}}\at{x},\\
\varrho_+\pair{t_1}{x}&=\varrho\pair{t_1}{x}\chi_{\cointerval{X_+\at{\ul\si}}{+\infty}}\at{x}.
\end{align*}
All three functions are nonnegative and satisfy $\varrho_-+\varrho_-+\varrho_+=\varrho$ by construction.
Thanks to $t_1\geq t_*$ and the $\fspaceL^\infty$-estimates from Lemma \ref{Lem:PropertiesOfSolutions} we
therefore get
\begin{align*}
\varrho_\pm\pair{t}{x}\leq\varrho\pair{t}{x}\leq \frac{C}{\nu^2}\qquad
\text{for all}\quad
\pair{t}{x}\in\ccinterval{t_1}{t_2}\times\Rset,
\end{align*}
and this implies
\begin{align*}
\varrho_\pm\pair{t}{X_-\at{\ol\si}}\leq \nu^{-3}=\ol\varrho\at{X_-\at{\ol\si}},\qquad
\varrho_\pm\pair{t}{X_+\at{\ul\si}}\leq\nu^{-3}= \ol\varrho\at{X_+\at{\ul\si}}
\end{align*}
for all $t\in\ccinterval{t_1}{t_2}$ and all sufficiently small $\nu$. Since we also have $0=\varrho_\pm\pair{t_1}{x}<\ol\varrho\at{x}$ for all $x\in\ccinterval{X_-\at{\ol\si}}{X_+\at{\ul\si}}$, the comparison principle for linear and parabolic PDEs yields
\begin{align*}
\varrho_\pm\pair{t}{x}\leq\ol\varrho\pair{t}{x}\qquad
\text{for all}\quad \pair{t}{x}\in\ccinterval{t_1}{t_2}\times\ccinterval{X_-\at{\ol\si}}{X_+\at{\ul\si}}
\end{align*}
and hence
\begin{align}
\notag
\int_{x^*-\eta}^{x_*+\eta}\bat{\varrho_-\pair{t}{x}+\varrho_+\pair{t}{x}}\dint{x}\leq
2 \int_{x^*-\eta}^{x_*+\eta}\ol\varrho\at{x}\dint{x}\leq{C}\tau^\alpha\qquad
\text{for all}\quad t\in\ccinterval{t_1}{t_2}.
\end{align}
On the other hand, using
the mass conservation property of \eqref{Eqn:FP1} we estimate
\begin{align*}
\int_{x^*-\eta}^{x_*+\eta}\varrho_0\pair{t}{x}
\dint{x}\leq
\int_{-\infty}^{+\infty}\varrho_0\pair{t}{x}\dint{x}= 
\int_{-\infty}^{+\infty}\varrho_0\pair{t_1}{x}\dint{x}=
\int_{X_-\at{\ol\si}}^{X_+\at{\ul\si}}\varrho\pair{t_1}{x}\dint{x},
\end{align*}
and 
\eqref{Lem:PS.NoMassInSpinodalRegion1.Eqn2} combined with the definition of $\xi$ and $m_0$, see \eqref{Eqn:Def.MomentXi} and \eqref{Eqn:Def.PartialMasses}, provides
\begin{align*}
\int_{X_-\at{\ol\si}}^{X_+\at{\ul\si}}\varrho\pair{t_1}{x}\dint{x}\leq {C}\bat{\xi\at{t_1}+m_0\at{t_1}}.
\end{align*}
In view of $\varrho=\varrho_-+\varrho_0+\varrho_+$ and by
taking the supremum over $t$ we finally get
\begin{align}
\label{Lem:PS.NoMassInSpinodalRegion1.Eqn4}
\sup\limits_{t\in\ccinterval{t_1}{t_2}}
\int_{x^*-\eta}^{x_*+\eta}\varrho\pair{t}{x}
\dint{x}\leq C\at{ \int_{X_-\at{\ol\si}}^{X_+\at{\ul\si}}\varrho\pair{t_1}{x}\dint{x}+\tau^\alpha}\leq
C\Bat{\xi\at{t_1}+m_0\at{t_1}+\nu^2},
\end{align}
where we used that $\tau^\al\leq\nu^2$.
The desired result is now provided by \eqref{Lem:PS.NoMassInSpinodalRegion1.Eqn4} and
Lemma \ref{Lem:PS.BoundForF}.
\end{proof}
Lemma \ref{Lem:PS.NoMassInSpinodalRegion1} allows us to control the evolution of 
$\xi+m_0$ as long as $\si$ is strictly between the bifurcation values $\si_*$ and $\si^*$, that means 
as long as the effective potential $H_\si$ has two proper wells. We next derive two similar results that cover time intervals in which $H_\si$ is either a single-well or a degenerate double-well potential. The corresponding pointwise upper bounds 
for  $\xi\at{t}+m_0\at{t}$, however, depend additionally on either $m_-\at{t_1}$ or $m_+\at{t_1}$.
\begin{figure}[ht!]%
\centering{%
\includegraphics[height=.275\textwidth, draft=\figdraft]{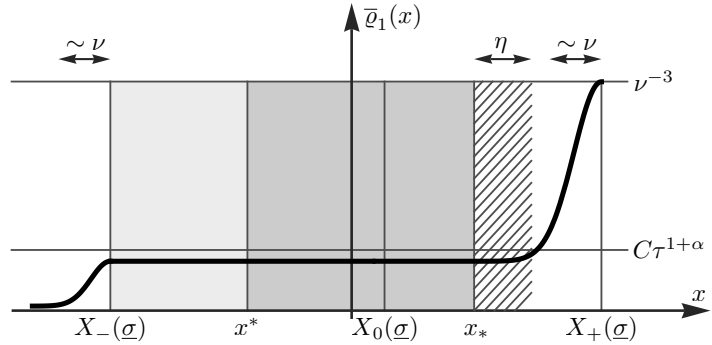}%
}%
\caption{Cartoon of the function $\ol\varrho_1$ from the proof of Lemma \ref{Lem:PS.NoMassInSpinodalRegion2}, which provides
a local supersolution 
on both $\ocinterval{-\infty}{X_-\at{\ul\si}}$ and 
$\ccinterval{x^*}{X_+\at{\ul\si}}$ but not on 
the interval
$\ccinterval{X_-\at{\ul\si}}{x^*}$ (light gray area). The corresponding defect term, however, can be compensated by a small time-dependent function 
$\ol\varrho_2$.}%
\label{Fig:NoMassInSpinodalRegion2}
\end{figure}%
\begin{lemma}[second and third conditional estimate for $\xi+m_0$] 
\label{Lem:PS.NoMassInSpinodalRegion2}
For each $\eps>0$ there exists a positive constant $C$,
which depends on $\eps$ but not on $\nu$, such that
the implications
\begin{align*}
\si\at{t}\geq \si^\#+\eps\;\;\text{for all}\;\; t\in\ccinterval{t_1}{t_2} \;\;\; \implies \;\;\;
\sup\limits_{t\in\ccinterval{t_1}{t_2}}\bat{\xi\at{t}+m_0\at{t}}\leq {C}\bat{\xi\at{t_1}+m_0\at{t_1}+m_-\at{t_1}+
\nu^2}
\end{align*}
and 
\begin{align*}
\si\at{t}\leq \si_\#-\eps
\;\;\text{for all}\;\; t\in\ccinterval{t_1}{t_2}
\;\;\;\implies\;\;\;
\sup\limits_{t\in\ccinterval{t_1}{t_2}}\!\bat{\xi\at{t}+m_0\at{t}}\leq 
{C}\Bat{\xi\at{t_1}+m_0\at{t_1}+m_+\at{t_1}+\nu^2}
\end{align*}
hold for all $\tsnu\leq t_1<t_2\leq T$ and all sufficiently small $\nu>0$.
\end{lemma}
\begin{proof}
As in the proof of Lemma \ref{Lem:PS.NoMassInSpinodalRegion1}, we justify the 
first implication by constructing appropriate supersolutions
to the linear PDE \eqref{Eqn:FP1} and by splitting the 
function $\varrho$; the second implication can be proven along the same lines.
\par%
{\ul{\emph{Construction of a supersolution:}}}
We set
\begin{align*}
\ul\si:=\si^\#+\tfrac12\eps,
\end{align*}
assume without loss of generality that $\ul\si<\si^*$ (for
$\ul\si\geq\si^*$, our arguments can even be simplified considerably), and
 fix $\eta>0$ sufficiently small 
such that
\begin{align*}
2\eta\leq X_+\at{\ul\si}-x_*,\qquad 2\eta \leq X_+\bat{\si^\#+\eps}-X_+\at{\ul\si}.
\end{align*}
We also define -- see Figure 
\ref{Fig:NoMassInSpinodalRegion2} for an illustration -- a
piecewise smooth and continuously differentiable function 
$\ol\varrho_1:\ocinterval{-\infty}{X_+\at{\ul\si}}\to\Rset$ by 
\begin{align*}
\ol{\varrho}_1\at{x}:=\nu^{-3}
\left\{
\begin{array}{lcl}
\D\exp\at{\frac{H_{\ol\si}\at{X_-\at{\ul\si}}-h_+\at{\ul\si}-H_{\ul\si}\at{x}}{\nu^2}}
&\text{for}&-\infty< x \leq X_-\at{\ul\si},\\
\D\exp\at{\frac{-h_+\at{\ul\si}}{\nu^2}}
&\text{for}&X_-\at{\ul\si}\leq x\leq X_0\at{\ul\si},\\
\D
\exp\at{\frac{H_{\ul\si}\at{X_+\at{\ul\si}}-H_{\ul\si}\at{x}}{\nu^2}}
&\text{for}&X_0\at{\ul\si}\leq x\leq X_+\at{\ul\si},
\end{array}
\right.
\end{align*}
and find -- thanks to $h_+\at{\ul\si}>h_+\at{\si^\#}=h_\#$ and our choice of $\eta$ -- two positive constants $C$, $\al$
such that
\begin{align*}
\sup\limits_{x\in\ocinterval{-\infty}{x_*+\eta}}\ol\varrho_1\at{x}\leq C\tau^{1+\al}
\end{align*}
holds all sufficiently small $\nu>0$.
Moreover, 
by direct computations as in the proof of Lemma \ref{Lem:PS.NoMassInSpinodalRegion1} we arrive at
\begin{align*}
\partial_x\Bat{ \nu^2\partial_{x}\ol\varrho_1\at{x} +\bat{H^\prime\at{x}-\si\at{t}}\ol\varrho_1\at{x}}\leq C\tau^{1+\al}\chi_{\ccinterval{X_-\at{\ul\si}}{x^*}}\at{x}
\end{align*}
for all $t$ with $\si\at{t}>\ul\si$, and conclude
that $\varrho_1$ satisfies the differential inequality for a supersolution on $\ocinterval{-\infty}{X_+\ul\si}$ up to very small error terms. In order to eliminate the latter,
we denote by
$\ol\varrho_2:\ccinterval{t_1}{t_2}\times\Rset\to\Rset$ the solution to the linear and inhomogeneous initial value problem
\begin{align*}
\tau\partial_t \ol\varrho_2\pair{t}{x} = \partial_x\Bat{\nu^2
\partial_x \ol\varrho_2\pair{t}{x}+\bat{H^\prime\at{x}-\si\at{t}}\ol\varrho_2\pair{t}{x}}+C\tau^{1+\al}\chi_{\ccinterval{X_-\at{\ul\si}}{x^*}}\at{x},\qquad \ol\varrho_2\pair{t_1}{x}=0.
\end{align*}
The function $\ol\varrho\pair{t}{x}:=\ol\varrho_1\at{x}+\ol\varrho_2\pair{t}{x}$ is then a time-dependent supersolution to the linear Fokker-Planck equation \eqref{Eqn:FP1} on the time-space domain $\ccinterval{t_1}{t_2}\times\ocinterval{-\infty}{X_+\at{\ul\si}}$. 
Moreover, since $\ol\varrho_2$ is non-negative and satisfies
\begin{align*}
\tau\frac{\dint}{\dint{t}}\int_\Rset\ol\varrho_2\pair{t}{x}\dint{x} \leq C\tau^{1+\al},
\end{align*}
we also get
\begin{align*}
\ol\varrho\pair{t}{X_+\at{\ul\si}}\geq\ol\varrho_1\at{X_+\at{\ul\si}}=\nu^{-3}
\end{align*}
as well as
\begin{align*}
\int_{-\infty}^{x_*+\eta}\ol\varrho\pair{t}{x}\dint{x}\leq 
\int_{-\infty}^{x_*+\eta}\ol\varrho_1\at{x}\dint{x}+
\int_{\Rset}\ol\varrho_2\pair{t}{x}\dint{x}\leq C\tau^\al
\end{align*}
for all $t\in\ccinterval{t_1}{t_2}$.
\par
{\ul{\emph{Moment estimates:}}}
The initial conditions
\begin{align*}
\varrho_{-0}\pair{t_1}{x}=\varrho\pair{t_1}{x}\chi_{\ocinterval{-\infty}{X_+\at{\ul\si}}}\at{x},\qquad 
\varrho_{+}\pair{t_1}{x}=\varrho\pair{t_1}{x}\chi_{\cointerval{X_+\at{\ul\si}}{+\infty}}\at{x},
\end{align*}
define two solutions $\varrho_{-0}$ and $\varrho_+$ to \eqref{Eqn:FP1}, which are defined on $\ccinterval{t_1}{t_2}\times\Rset$ and satisfy $\varrho_{-0}+\varrho_{+}=\varrho$ along with 
\begin{align*}
0\leq\varrho_{-0}\pair{t}{x},\,\varrho_+\pair{t}{x}\leq \nu^{-3}
\end{align*}
for all sufficiently small $\nu>0$, see Lemma \ref{Lem:PropertiesOfSolutions}. The comparison principle  -- applied to \eqref{Eqn:FP1} and with respect to the time-space domain 
$\ccinterval{t_1}{t_2}\times\ocinterval{-\infty}{X_+\at{\ul\si}}$ -- now yields $\varrho_+\leq\ol\varrho$ and hence
\begin{align*}
\int_{x^*-\eta}^{x_*+\eta}\varrho_{+}\pair{t}{x}\dint{x}
\leq
\int_{x^*-\eta}^{x_*+\eta}\ol\varrho\pair{t}{x}\dint{x}\leq C\tau^\al,
\end{align*}
while the estimate
\begin{align*}
\int_{x^*-\eta}^{x_*+\eta}\varrho_{-0}\pair{t}{x}\dint{x}&\leq
\int_{-\infty}^\infty\varrho_{-0}\pair{t_1}{x}\dint{x}=
\int_{-\infty}^{X_+\at{\ul\si}}\varrho\pair{t_1}{x}\dint{x}
\\&\leq \int_{x^*}^{x_*}\varrho\pair{t_1}{x}\dint{x}
+
C\int_\Rset\bat{H^\prime\at{x}-\si\at{t_1}}\varrho\pair{t_1}{x}\dint{x}\leq C\at{m_0\at{t_1}+\xi\at{t_1}}
\end{align*}
holds due to our choice of $\ul\si$ and $\eta$. In summary, we have
\begin{align*}
\sup_{t\in\ccinterval{t_1}{t_2}}\int_{x^*-\eta}^{x_*+\eta}\varrho\pair{t}{x}\dint{x}\leq
C\bat{m_0\at{t_1}+\xi\at{t_1}+\tau^\al},
\end{align*}
and the desired result is a consequence of Lemma \ref{Lem:PS.BoundForF}.
\end{proof}
%
%
\subsection{Monotonicity relations}
\label{sect:Montonicity}
In this section we complement the stability estimates from \S\ref{sect:PeakStability} by
dynamical monotonicity relations for the partial masses 
$m_-$, $m_+$  and the dynamical multiplier $\si$. These results allow us to bound the moment $\zeta$ from \eqref{Eqn:DefZeta} for all sufficiently large $t$, see \S\ref{sect:AppByStablePeaks}. 
\subsubsection{Mass transfer between the stable regions}
We first investigate the evolution of the partial masses $m_-$ and $m_+$ 
for $t\geq t_*$ by means of appropriate
moment equations. The resulting estimates imply for $0<\nu\ll1$ that the mass flux from 
the left stable interval $\ocinterval{-\infty}{x^*}$ towards
the right one $\cointerval{x_*}{+\infty}$ is -- up to small correction terms  -- 
positive for $\si\at{t}>\si_\#$ but negative for $\si\at{t}<\si^\#$, and hence that there is 
essentially no mass transfer in the subcritical regime $\si\at{t}\in\oointerval{\si_\#}{\si^\#}$. 
These findings perfectly agree with the large deviations results that we obtained in \S\ref{sect:HeuristicDynamics}
by analyzing the orders of magnitude in Kramers formula \eqref{Eqn:KramersFormula}.
\begin{figure}[ht!]%
\centering{%
\includegraphics[height=.3\textwidth, draft=\figdraft]{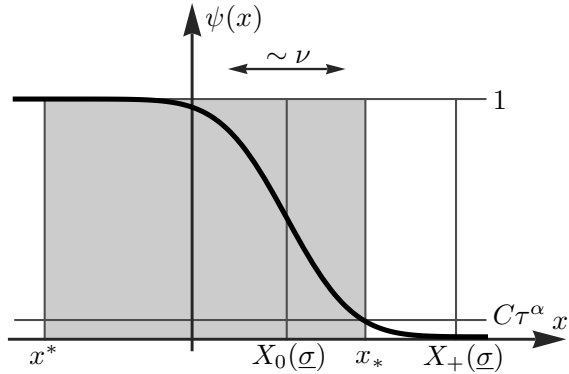}%
}%
\caption{Cartoon of the moment weight $\psi$ that is used in the proof of Lemma \ref{Lem:ChangeOfMasses}.
The strictly decreasing branch of $\psi$ on the interval $\ccinterval{x^*}{X_+\at{\ul\si}}$ is given by the rescaled and shifted primitive of $-1/\ga_{\ul\si}$ and has effective width of order $\nu$. For $0<\nu\ll1$, the function $\psi$ is therefore close to the indicator function of $\oointerval{-\infty}{X_0\at{\ul\si}}$.
}%
\label{Fig:ChangeOfMasses}
\end{figure}%
\begin{lemma}[monotonicity estimates for $m_\pm$]
\label{Lem:ChangeOfMasses}
For each $\eps$ with
\begin{align*}
0<\eps<\min\{ \si^*-\si_\#,\;\si^\#-\si_*\}
\end{align*}
there exist constants $\alpha$ and $C$, which depend on $\eps$ but not on $\nu$, such that
the implications
\begin{align*}
\si\at{t}\geq \si_\#+\eps
\quad\text{for all}\quad t\in\ccinterval{t_1}{t_2}
\quad\implies \quad
\sup\limits_{t\in\ccinterval{t_1}{t_2}} m_-\at{t}\leq m_-\at{t_1}+m_0\at{t_1}+C\tau^\alpha
\end{align*}
and
\begin{align*}
\si\at{t}\leq \si^\#-\eps
\quad\text{for all}\quad t\in\ccinterval{t_1}{t_2}
\quad\implies \quad
\sup\limits_{t\in\ccinterval{t_1}{t_2}} m_+\at{t}\leq m_+\at{t_1}+m_0\at{t_1}+C\tau^\alpha
\end{align*}
hold for $\tsnu\leq t_1<t_2\leq T$ and all sufficiently small $\nu>0$.
\end{lemma}
\begin{proof}
We demonstrate the first implication only; the second one follows analogously. Our strategy in this proof is
to control the evolution of a certain upper bound for $m_-$, namely
of the moment
\begin{align*}
\ol{m}_-\at{t}:=\int_\Rset \psi\at{x}\varrho\pair{t}{x}\dint{x}.
\end{align*}
Here, the weight $\psi$ is defined as 
piecewise constant continuation
of an appropriately rescaled and shifted primitive of $-1/\ga_{\ul\si}$, where $\ul\si$ is shorthand for $\si_\#+\eps$.
More precisely, in view of \eqref{Eqn:DefEquilibriumDensities} we set
\begin{align*}
\psi\at{x}:=
\left\{
\begin{array}{lcl}
1&\;\text{for}\;& x\leq{x^*},\\
\displaystyle\frac{\displaystyle\int_x^{X_+\at{\ul\si}}\exp\at{\frac{H_{\ul\si}\at{y}}{\nu^2}}\dint{y}}{\displaystyle\int_{x^*}^{X_+\at{\ul\si}}\exp\at{\frac{H_{\ul\si}\at{y}}{\nu^2}}\dint{y}}
&\;\text{for}\;& 
{x^*}\leq{x}\leq{X_+\at{\ul\si}},\\
0&\;\text{for}\;& x\geq{X_+\at{\ul\si}},\\
\end{array}
\right.
\end{align*}
and refer to Figure \ref{Fig:ChangeOfMasses} for an illustration. Since the function
$\psi$ is continuous as well as piecewise twice continuously differentiable, we derive 
-- using \eqref{Eqn:FP1} as well as integration by parts -- the moment balance
\begin{align*}
\tau\dot{\ol{m}}_-\at{t}
&=-
\int_\Rset\psi^\prime\at{x}\Bat{\nu^2\partial_x\varrho\pair{t}{x}+\bat{H^\prime\at{x}-\sigma\at{t}}\varrho\pair{t}{x}}\dint{x}
\\&=-
\int_{x^*}^{X_+\at{\ul\si}}\psi^\prime\at{x}\Bat{\nu^2\partial_x\varrho\pair{t}{x}+\bat{H^\prime\at{x}-\sigma\at{t}}\varrho\pair{t}{x}}\dint{x}
\\&=
\mathrm{b.t.}+\int_{x^*}^{X_+\at{\ul\si}}\Bat{\nu^2\psi^{\prime\prime}\at{x}+\bat{\sigma\at{t}-H^\prime\at{x}}\psi^\prime\at{x}}\varrho\pair{t}{x}\dint{x}.
\end{align*}
The boundary terms are given by
\begin{align*}
\mathrm{b.t.}=\nu^2 \psi^\prime\bat{x^*+0}\varrho\bpair{t}{x^*}-
\nu^2\psi^\prime\bat{X_+\at{\ul\si}-0}\varrho\bpair{t}{X_+\at{\ul\si}},
\end{align*}
and the notation $\pm0$ indicates that the boundary values of $\psi^\prime$ are taken with respect to the interval $\ccinterval{x^*}{X_+\at{\ul\si}}$.
\par
It remains to estimate $\dot{\ol{m}}_-\at{t}$ for all $t\in\ccinterval{t_1}{t_2}$, that means for
$\si\at{t}\geq\ul\si$. We first infer from the above definition of $\psi$ and the properties of the effective potential
$H_\si$ that
\begin{align*}
\psi^\prime\at{x}\leq0,\qquad
\nu^2\psi^{\prime\prime}\at{x}+\bat{\ul\si-H^\prime\at{x}}\psi^\prime\at{x}=0 
\end{align*}
holds for all $x\in\ccinterval{x^*}{X_+\at{\ul\si}}$, and this implies
\begin{align*}
\tau\dot{\ol{m}}_-\at{t}&\leq
\nu^2\babs{\psi^\prime\bat{X_+\at{\ul\si}-0}}\varrho\bpair{t}{X_+\at{\ul\si}}.
\end{align*}
We then observe that the asymptotic properties of $\ga_\si$ -- see Figure \ref{Fig:Equilibria} -- imply 
\begin{align*}
\int_{x^*}^{X_+\at{\ul\si}}\exp\at{\frac{H_{\ul\si}\at{y}}{\nu^2}}\dint{y}\geq c\nu\exp\at{
\frac{
H_{\ul{\si}}\bat{X_0\at{\ul{\si}}}
}{\nu^2}
}
\end{align*}
for some sufficiently small constant $c>0$, and in view of
$h_+\at{\ul\si}>h_+\at{\si_\#}=h_\#$  and the scaling relation from Assumption \ref{Ass:Tau} we verify that the estimates
\begin{align*}
\abs{\psi^\prime\at{X_+\at{\ul\si}-0}}&\leq \frac{C}{\nu}
\exp\at{\frac{H_{\ul{\si}}\bat{X_+\at{\ul{\si}}}-H_{\ul{\si}}\bat{X_0\at{\ul{\si}}}}{\nu^2}}
\\&=
\frac{C}{\nu}
\exp\at{\frac{-h_+\at{\ul{\si}}}{\nu^2}}
\\&=
\frac{C}{\nu}
\exp\at{\frac{-h_\#}{\nu^2}}\exp\at{\frac{h_\#-h_+\at{\ul{\si}}}{\nu^2}}
\leq C \tau^{1+\alpha} 
\end{align*}
and
\begin{align*}
\sup_{x\geq x_*}\psi\at{x}=\psi\at{x_*}\leq 
C
\exp\at{\frac{H_{\ul{\si}}\bat{x_*}-H_{\ul{\si}}\bat{X_0\at{\ul{\si}}}}{\nu^2}}\leq C\tau^\alpha
\end{align*}
hold for some positive constants $\alpha$, $C$ and all sufficiently small $\nu>0$. Moreover,  
Lemma \ref{Lem:PropertiesOfSolutions} ensures that $\varrho\bpair{t}{X_+\at{\ul\si}}\leq C\nu^{-2}$. Combining all estimates derived so far we finally obtain
\begin{math}
\dot{\ol{m}}_-\at{t}\leq C\tau^{\alpha}
\end{math}, %
and hence 
\begin{align*}
\sup_{t\in\ccinterval{t_1}{t_2}}m_-\at{t}\leq 
\sup_{t\in\ccinterval{t_1}{t_2}}\ol{m}_-\at{t}\leq \ol{m}_-\at{t_1}+C\tau^\al\leq m_-\at{t_1}+m_0\at{t_1}+C\tau^\al,
\end{align*}
where we used that $t_2-t_1\leq{T}$ and $\ol{m}_-\at{t_1}\leq m_-\at{t_1}+m_0\at{t_1}+\sup_{x\geq x_*}\psi\at{x}$.
\end{proof}
%
%
\subsubsection{Dynamical relations \texorpdfstring{between $\si$ and $\ell$}{}}%
%
%
As an important consequence of the monotonicity relations for $m_-$ and $m_+$ we now
establish, up to some small error terms, monotonicity relations between the dynamical control $\ell$ and the dynamical multiplier $\si$. These
results have three important implications, which can informally be summarized as follows:
\begin{enumerate}
\item If $\si\at{t}\approx\si_\#$ or $\si\at{t}\approx\si^\#$ holds for all $t$ in some interval $\ccinterval{t_1}{t_2}$, 
then $\ell$ must be essentially decreasing or increasing, respectively, on this interval. The dynamical 
constraint \eqref{Eqn:FP2P} then implies in the limit $\nu\to0$ that the phase fraction $\mu= m_+-m_-$ is decreasing and increasing for $\si=\si_\#$
and $\si=\si^\#$, respectively.
\item If $\ccinterval{t_1}{t_2}$ is some time interval such that $\si$ behaves nicely and
\begin{enumerate}
\item crosses $\si_\#$ from above in the sense of $\si\at{t_2}<\si_\#<\si\at{t_1}<\si^\#$, or
\item crosses $\si^\#$ from below via $\si_\#<\si\at{t_1}<\si^\#<\si\at{t_2}$,
\end{enumerate}
then $t_2-t_1$ can be bounded from below by $\abs{\si\at{t_2}-\si\at{t_1}}$. This implies, roughly speaking,
that solutions for small $\nu$ cannot change too rapidly from subcritical $\si$ to supercritical $\si$.
\item 
If $\ccinterval{t_1}{t_2}$ is some time interval such that $\si$ stays inside the subcritical range $\oointerval{\si_\#}{\si^\#}$,
then $\abs{\si\at{t_2}-\si\at{t_1}}$ can be bounded from above by $\abs{\ell\at{t_2}-\ell\at{t_1}}$, and 
this gives rise to Lipschitz estimates for subcritical $\si$ in the limit $\nu\to0$.
\end{enumerate}
\begin{lemma}[conditional monotonicity relations]
\label{Lem:EllAndSigma}
Let $\eps$ be fixed with $0<\eps<\tfrac12\at{\si^\#-\si_\#}$. Then the implications
\begin{align*}
\si\at{t}\in\bccinterval{\tfrac12\nat{\si_*+\si_\#}}{ \si^\#-\eps}
\quad\text{for all}\quad t\in\ccinterval{t_1}{t_2}
\quad\implies \quad
g\bat{\si\at{t_1}-\si\at{t_2}}\leq \ell\at{t_1}-\ell\at{t_2}+\mathrm{e.t.}
\end{align*}
and
\begin{align*}
\si\at{t}\in\bccinterval{ \si_\#+\eps}{\tfrac12\nat{\si^\#+\si^*}}
\quad\text{for all}\quad t\in\ccinterval{t_1}{t_2}
\quad\implies \quad
g\bat{\si\at{t_2}-\si\at{t_1}}\leq \ell\at{t_2}-\ell\at{t_1}+\mathrm{e.t.}
\end{align*}
hold with error terms
\begin{align*}
\mathrm{e.t.}:=C\Bat{\sqrt{\xi\at{t_1}+m_0\at{t_1}}+\sqrt{\xi\at{t_2}+m_0\at{t_2}}}+C\tau^\al
\end{align*}
for all $\tsnu\leq{t_1}\leq{t_2}\leq T$ and all sufficiently small $\nu>0$. Here, 
the constants $\alpha$, $C$ depend on $\eps$ but not on $\nu$,
and 
$g$ is the increasing and piecewise linear function $g\at{s}=C_{\sgn\at{s}}s$, where the constants $C_-,\,C_+>0$ are 
independent of both $\eps$ and $\nu$. 
\end{lemma}
\begin{proof}
We derive the first implication only; the arguments for the second one are similar.
For the proof we suppose
that $\si\at{t}\in\ccinterval{ \tfrac12\nat{\si_*+\si_\#}}{ \si^\#-\eps}$ holds for all $t\in\ccinterval{t_1}{t_2}$ and 
set $x_{\pm}\at{t}:=X_{\pm}\bat{\si\at{t}}$ as well as
\begin{align*}
\tilde\ell\at{t}:=m_-\at{t} x_-\at{t}+m_+\at{t}x_+\at{t},\qquad
\bar{e}:=\sqrt{\xi\at{t_1}+m_0\at{t_1}}+\sqrt{\xi\at{t_2}+m_0\at{t_2}}.
\end{align*} 
Lemma \ref{Lem:XiControlsL} yields $\babs{\ell\at{t_i}-\tilde{\ell}\at{t_i}}\leq C\bar{e}$,
and we conclude that 
\begin{align}
\label{Lem:EllAndSigma.Eqn1}
\begin{split}
C\bar{e}+\ell\at{t_1}-\ell\at{t_2}&\geq \tilde\ell\at{t_1}-\tilde\ell\at{t_2}\\
&= \sum_{j\in\{-,+\}} \bat{m_j\at{t_1}-m_j\at{t_2}}x_j\at{t_1}+
\sum_{j\in\{-,+\}} m_j\at{t_2}\bat{x_j\at{t_1}-x_j\at{t_2}}.
\end{split}
\end{align}
Thanks to $m_-\at{t}+m_0\at{t}+m_+\at{t}=1$, see \eqref{Eqn:Def.PartialMasses}, we find
\begin{align*}
m_-\at{t_1}-m_-\at{t_2}=-\bat{m_0\at{t_1}-m_0\at{t_2}}-\bat{m_+\at{t_1}-m_+\at{t_2}}
\end{align*} 
and hence 
\begin{align*}
\sum_{j\in\{-,+\}} \bat{m_j\at{t_1}-m_j\at{t_2}}x_j\at{t_1}&\geq
\bat{m_+\at{t_1}-m_+\at{t_2}}\bat{x_+\at{t_1}-x_-\at{t_1}}-C\bar{e},
\end{align*}
where we used that $\abs{x_\pm\at{t_1}}\leq{C}$ and $m_0\at{t_1},m_0\at{t_2}\leq\bar{e}$.
Moreover, Lemma \ref{Lem:ChangeOfMasses} provides constants $\alpha$ and $C$ such that 
\begin{align*}
m_+\at{t_1}-m_+\at{t_2}\geq - C\bat{\bar{e}+\tau^\alpha}
\end{align*}
holds for all sufficiently small $\nu$, and in view of $x_+\at{t_1}>x_-\at{t_1}$, see Remark \ref{Rem:PropertiesOfXi}, we arrive at
\begin{align}
\label{Lem:EllAndSigma.Eqn2}
\sum_{j\in\{-,+\}} \bat{m_j\at{t_1}-m_j\at{t_2}}x_j\at{t_1}&\geq -C\bat{\bar{e}+\tau^\alpha}.
\end{align}
It remains to estimate the second sum on the right hand side of \eqref{Lem:EllAndSigma.Eqn1} depending on the sign of $\si\at{t_1}-\si\at{t_2}$.
In the case of $\si\at{t_1}-\si\at{t_2}\geq0$ we have $x_j\at{t_1}>x_j\at{t_2}$ thanks to the monotonicity of $X_\pm$, so the Mean Value Theorem implies
\begin{align*}
x_j\at{t_1}-x_j\at{t_2}=X_j^{\prime}\at{\tilde{\si}}\at{\si\at{t_1}-\si\at{t_2}}\geq C_+
\bat{\si\at{t_1}-\si\at{t_2}},\qquad C_+:=\min\limits_{\tilde{\si}\in{S}
}\min\limits_{j\in\{-,\,+\}}
X_j^\prime\at{\tilde\si},
\end{align*}
where $S$ abbreviates the interval $\ccinterval{ \tfrac12\nat{\si_*+\si_\#}}{ \tfrac12\nat{\si^\#+\si^*}}$. Similarly, for
$\si\at{t_1}-\si\at{t_2}<0$  we get
\begin{align*}
x_j\at{t_2}-x_j\at{t_1}\leq C_-
\bat{\si\at{t_2}-\si\at{t_1}},\qquad C_-:=\max\limits_{\tilde{\si}\in{S}
}\max\limits_{j\in\{-,\,+\}}
X_j^\prime\at{\tilde\si}.
\end{align*}
In summary, in both cases we have
\begin{align*}
x_j\at{t_1}-x_j\at{t_2}\geq  C_{\sgn\at{\si\at{t_1}-\si\at{t_2}}} 
\bat{\si\at{t_1}-\si\at{t_2}}=
g\bat{\si\at{t_1}-\si\at{t_2}}
\end{align*} 
and hence 
\begin{align*}
\sum_{j\in\{-,+\}} m_j\at{t_2}\bat{x_j\at{t_1}-x_j\at{t_2}}&\geq
g\bat{\si\at{t_1}-\si\at{t_2}}\bat{m_-\at{t_1}+m_+\at{t_1}}
\\&\geq 
g\bat{\si\at{t_1}-\si\at{t_2}}-C\bar{e}.
\end{align*}
The desired implication now follows by combining the latter estimate with \eqref{Lem:EllAndSigma.Eqn1} and  
\eqref{Lem:EllAndSigma.Eqn2}.
\end{proof}
%
%
\section{ Justification of the limit dynamics }\label{sect:limit}
%
In this section we finally combine all partial results from \S\ref{sect:AuxResults}
in order to justify the limit model from \S\ref{sect:limitmodel}. To this end we fix $\eps_*>0$ such that
\begin{align}
\label{Eqn:Conv.Delta}
\eps_*\leq \tfrac12\min\{\si^*-\si^\#,\;\si_\#-\si_*,\;\si^\#,\;-\si_\#\},\qquad \sup_{t\in\ccinterval{0}{T}}\abs{\si\at{t}}\leq\frac{1}{\eps_*}.
\end{align}  
holds for all $0<\nu\leq1$, and assume from now on that $0<\eps<\eps_*$. Notice that the uniform bounds from Lemma \ref{Lem:PropertiesOfSolutions} ensure that such an $\eps_*>0$ does in fact exist.
%
\subsection{Approximation by stable peaks}\label{sect:AppByStablePeaks}

Heuristically it is clear that the small-parameter dynamics of the nonlocal Fokker-Planck equation can be  described by the rate-independent limit model from \S\ref{sect:limitmodel} if and only if
the state $\varrho\pair{t}{\cdot}$ of the system can be approximated by
\begin{enumerate}
\item two narrow peaks located at $X_-\at{t}$ and $X_+\at{t}$ as long as
$\si\at{t}\in\oointerval{\si_\#}{\si^\#}$,
\item a single narrow peak located at $X_-\at{t}$ or $X_+\at{t}$ for $\si\at{t}\in\oointerval{-\infty}{\si_\#}$ or $\si\at{t}\in\oointerval{\si^\#}{+\infty}$, respectively.
\end{enumerate}
In this section we establish an $\eps$-variant of this 
approximation result. More precisely, we now prove that the moment
$\zeta\at{t}$ from \eqref{Eqn:DefZeta} is small for all times $t\geq{t_0}$ provided that the dissipation is small
at time $t_0$ and that $\nu$ is sufficiently small. This conclusion is in fact at the very core of our approach
as it allows us to convert the $\fspaceL^1$-bound for the dissipation into moment estimates that hold 
pointwise in time.
\begin{figure}[ht!]%
\centering{%
\includegraphics[height=.275\textwidth, draft=\figdraft]{\figfile{persistence}}%
}%
\caption{Schematic representation of the intervals $J_i$ as well as the $\si$-domains for the cases $N_j$ and $P_j$ as used in the proof of Lemma \ref{Lem:ZetaDeltaBounds}. For $\eps\to0$, we have 
$J_{-2}\to\ocinterval{-\infty}{\si_\#}$, $J_0\to\ccinterval{\si_\#}{\si^\#}$,  $J_{+2}\to\cointerval{\si^\#}{+\infty}$
as well as $J_{-1}\to\{\si_\#\}$ and $J_{+1}\to\{\si^\#\}$.
}%
\label{Fig:ZetaDeltaBounds}
\end{figure}%
\begin{lemma}[pointwise upper bound for $\zeta$]
\label{Lem:ZetaDeltaBounds}
For each $\eps\in\ccinterval{0}{\eps_*}$ there exist positive constants $\beta<1$ and
$C$, which depend on $\eps$ but not on $\nu$, such that the implication
\begin{align}
\label{Lem:ZetaDeltaBounds.I0}
\calD\at{t_0}\leq \tau^\beta \quad \implies \quad
\sup\limits_{t\in\ccinterval{t_0}{T}}\zeta\at{t}\leq C\nu^2
\end{align}
holds for all $\tsnu\leq t_0\leq{T}$ and all sufficiently small $\nu>0$, where $\eps_*$ and $t_*$
have been introduced in \eqref{Eqn:Conv.Delta} and Proposition \ref{Lem:PropertiesOfSolutions}.
\end{lemma}
\begin{proof}
We consider the intervals 
\begin{align*}
J_{-2} = \ocinterval{-\infty}{\si_\#-\eps},\qquad
J_{-1} = \ccinterval{\si_\#-\eps}{\si_\#-\tfrac12\eps},
\end{align*}
as well as $J_0 := \ccinterval{\si_\#-\tfrac12\eps}{\si^\#+\tfrac12\eps}$ and
\begin{align*}
J_{+1} = \ccinterval{\si^\#+\tfrac12\eps}{\si^\#+\eps},\qquad
J_{+2} = \cointerval{\si^\#+\eps}{+\infty}.
\end{align*}
These intervals and the different cases considered in this proof are illustrated in Figure \ref{Fig:ZetaDeltaBounds}. 
We also recall that the two definitions \eqref{Eqn:DefDissipation} and \eqref{Eqn:DefDissipationFunctional} of the dissipation are consistent in the sense that $\calD\at{t}=\calD_{\si\at{t}}\bat{\varrho\pair{t}{\cdot}}$ holds for all $t\geq0$.
\par
{\ul{\emph{Part 1:}}} We first prove statements like \eqref{Lem:ZetaDeltaBounds.I0} under the assumption that $\si$ remains confined to 
at most two or three neighboring intervals from $\{J_{-2},\, J_{-1},\, J_{0},\, J_{+1},\, J_{+2}\}$,
and start with the case
\begin{align}
\tag{$N_{-}$}\label{Lem:ZetaDeltaBounds.CM}
\si\at{t}\in J_{-2}\cup{J_{-1}}\quad\text{for all}\quad t\in\ccinterval{t_1}{t_2},
\end{align}
where $t_0\leq t_1<t_2\leq{T}$. Lemma \ref{Lem:PS.NoMassInSpinodalRegion2} and Lemma \ref{Lem:ChangeOfMasses} then provide constants $\alpha_1$ and $C_1$ such that
\begin{align}
\label{Lem:ZetaDeltaBounds.I0.xEqn2}
\sup_{t\in\ccinterval{t_1}{t_2}}\bat{\xi\at{t}+m_0\at{t}}
\leq C_1\Bat{\xi\at{t_1}+m_+\at{t_1}+m_0\at{t_1}+\nu^2}
\end{align}
as well as
\begin{align}
\label{Lem:ZetaDeltaBounds.I0.xEqn3}
\sup_{t\in\ccinterval{t_1}{t_2}}m_+\at{t}
\leq m_+\at{t_1}+m_0\at{t_1}+C_1\tau^{\alpha_1},
\end{align}
and by Lemma \ref{Lem:MD.MassOutsideOnePeak} there exist constants $\alpha_2$ and $C_2$ such that
\begin{align*}
m_0\at{t_1}+m_+\at{t_1}\leq C_2\tau^{\alpha_2}\Bat{\tau^{-1}\calD\at{t_1} +1}.
\end{align*} 
Moreover, Lemma \ref{Lem:PS.AuxFormula} yields a constant $C_3$ with
\begin{align}
\notag
\xi\at{t_1}\leq \calD\at{t_1}+C_3\nu^2,
\end{align} 
and combining all these estimates we finally arrive at
\begin{align*}
\sup_{t\in\ccinterval{t_1}{t_2}}\zeta\at{t}&\leq
\sup_{t\in\ccinterval{t_1}{t_2}}\bat{\xi\at{t}+m_0\at{t}+m_+\at{t}}
\\&\leq  \bat{C_1+C_2\at{C_1+1}\tau^{\al_2-1}}\calD\at{t_1} + C_1\tau^{\al_1}+ C_2\at{C_1+1}\tau^{\al_2}+
C_1\at{C_3+1}\nu^2.
\end{align*}
We next choose $\beta_1\in\oointerval01$ sufficiently large such that $\alpha_2+\beta_1-1>0$ and this guarantees that the implication
\begin{align}
\label{Lem:ZetaDeltaBounds.I1}
\eqref{Lem:ZetaDeltaBounds.CM}
\quad \text{and}\quad \calD\at{t_1}\leq\tau^{\beta_1}
\qquad\implies\qquad 
\sup_{t\in\ccinterval{t_1}{t_2}}\zeta\at{t}
\leq
\sup_{t\in\ccinterval{t_1}{t_2}}
\bat{\xi\at{t}+m_0\at{t}+m_+\at{t}}
\leq C_4 \nu^2
\end{align}
holds for all sufficiently small $\nu>0$, where $C_4$ can be chosen as
$C_4:=1+C_1\bat{C_3+1}$.
\par
The arguments for the case 
\begin{align}
\tag{$N_{+}$}\label{Lem:ZetaDeltaBounds.CP}
\si\at{t}\in J_{+1}\cup{J_{+2}}\quad\text{for all}\quad t\in\ccinterval{t_1}{t_2}
\end{align}
are entirely similar. In particular, possibly changing all constants introduced so far, we readily demonstrate that
\begin{align}
\label{Lem:ZetaDeltaBounds.I2}
\eqref{Lem:ZetaDeltaBounds.CP}
\quad \text{and}\quad \calD\at{t_1}\leq\tau^{\beta_1}
\qquad\implies\qquad 
\sup_{t\in\ccinterval{t_1}{t_2}}\zeta\at{t}
\leq
\sup_{t\in\ccinterval{t_1}{t_2}}
\bat{\xi\at{t}+m_-\at{t}+m_0\at{t}}
\leq C_4\nu^2
\end{align}
holds for all sufficiently small $\nu>0$. 
\par
We next study the case
\begin{align}
\tag{$N_0$}\label{Lem:ZetaDeltaBounds.CZ}
\si\at{t}\in J_{-1}\cup{J_0}\cup{J_{+1}}\quad\text{for all}\quad t\in\ccinterval{t_1}{t_2},
\end{align}
and first observe that Lemma \ref{Lem:PS.NoMassInSpinodalRegion1} provides a constant $C_5$
such that
\begin{align}
\label{Lem:ZetaDeltaBounds.I0.xEqn1}
\sup_{t\in\ccinterval{t_1}{t_2}}\bat{\xi\at{t}+m_0\at{t}}
\leq C_5 \bat{\xi\at{t_1}+m_0\at{t_1}+\nu^2}.
\end{align}
By Lemma \ref{Lem:MD.MassOutsideTwoPeaks} we find further constants $\alpha_6$ and $C_6$ such that
\begin{align*}
m_0\at{t_1}\leq C_6\tau^{\alpha_6}\Bat{\tau^{-1}\calD\at{t_1} +1},
\end{align*} 
and we choose $\beta_2\in\oointerval{0}{1}$ sufficiently close to $1$ such that $\alpha_6+\beta_2-1>0$. This ensures
(using also \eqref{Lem:ZetaDeltaBounds.I0.xEqn3})
that the implication
\begin{align}
\label{Lem:ZetaDeltaBounds.I3}
\eqref{Lem:ZetaDeltaBounds.CZ}
\quad \text{and}\quad \calD\at{t_1}\leq\tau^{\beta_2}
\qquad\implies\qquad 
\sup_{t\in\ccinterval{t_1}{t_2}}\zeta\at{t}=
\sup_{t\in\ccinterval{t_1}{t_2}}\bat{\xi\at{t}+m_0\at{t}}\leq C_7\nu^2
\end{align}
holds for all sufficiently small $\nu>0$ and $C_7:=C_5\at{C_3+3}$.
\par
{\ul{\emph{Part 2:}}} We set
\begin{align*}
\beta:=\max\big\{\beta_1,\,\beta_2\big\} \in\oointerval{0}{1}, \qquad C:=\max\big\{C_4,\,C_7\big\}.
\end{align*}
Our next goal is to demonstrate that whenever the systems passes for $t\in\ccinterval{t_3}{t_4}\subseteq\ccinterval{\tsnu}{T}$ 
through one of the intervals ${J_{-1}}$ or ${J_{+1}}$, then there exist
at least one time $t$ in between $t_3$ and $t_4$ such that $\calD\at{t}\leq\tau^\beta$.
To this end, we have to discuss the four cases
\begin{align}
\tag{$P_{-0}$}\label{Lem:ZetaDeltaBounds.TMZ}
\si\at{t}\in {J_{-1}}\quad\text{for all}\quad t\in\ccinterval{t_3}{t_4},\qquad
\si\at{t_3} =\si_\#-\eps,\qquad
\si\at{t_4}=\si_\#-\tfrac12\eps
\end{align}
and
\begin{align}
\tag{$P_{0-}$}\label{Lem:ZetaDeltaBounds.TZM}
\si\at{t}\in {J_{-1}}\quad\text{for all}\quad t\in\ccinterval{t_3}{t_4},\qquad
\si\at{t_3} =\si_\#-\tfrac12\eps,\qquad
\si\at{t_4}=\si_\#-\eps
\end{align}
as well as
\begin{align}
\tag{$P_{+0}$}\label{Lem:ZetaDeltaBounds.TPZ}
\si\at{t}\in {J_{+1}}\quad\text{for all}\quad t\in\ccinterval{t_3}{t_4},\qquad
\si\at{t_3}=\si^\#+\eps,\qquad \si\at{t_4}=\si^\#+\tfrac12\eps
\end{align}
and
\begin{align}
\tag{$P_{0+}$}\label{Lem:ZetaDeltaBounds.TZP}
\si\at{t}\in {J_{+1}}\quad\text{for all}\quad t\in\ccinterval{t_3}{t_4},\qquad
\si\at{t_3}=\si^\#+\tfrac12\eps,\qquad \si\at{t_4}=\si^\#+\eps
\end{align}
but by symmetry it is sufficient to study \eqref{Lem:ZetaDeltaBounds.TMZ} and \eqref{Lem:ZetaDeltaBounds.TZM} only. To discuss the case \eqref{Lem:ZetaDeltaBounds.TZM}, we suppose that 
\begin{align*}
\zeta\at{t_3}=\xi\at{t_3}+m_0\at{t_3}\leq C\nu^2,
\end{align*}
and notice that our arguments for the case \eqref{Lem:ZetaDeltaBounds.CZ} -- see \eqref{Lem:ZetaDeltaBounds.I0.xEqn1} with $t_1=t_3$, $t_2=t_4$ -- imply
\begin{align*}
\xi\at{t_4}+m_0\at{t_4}\leq C_5\at{C+1}\nu^2.
\end{align*}
Lemma \ref{Lem:EllAndSigma} combined with the uniform bounds for $\nabs{\dot\ell\at{t}}$ from Assumption \ref{Ass:Constraint}  yields
constants $c_8$ and $C_9$ such that
\begin{align*}
\tfrac12\eps=\si\at{t_3}-\si\at{t_4}\leq C_8\abs{t_4-t_3}+C_9\nu
\end{align*}
and we conclude that there exists a positive constant $c_{10}$ such that
$t_4-t_3\geq c_{10}$ holds for all sufficiently small $\nu>0$. Moreover, since Lemma \ref{Lem:PropertiesOfSolutions} provides
$\int_{t_3}^{t_4}\calD\at{t}\dint{t}\leq {C_{11}}\tau$ we find at least one
time $t\in\ccinterval{t_3}{t_4}$ (which depends on $\nu$ and $\eps$) such that
\begin{align*}
\calD\at{t}\leq\frac{C_{11}}{c_{10}}\tau\leq\tau^\beta
\end{align*}
for all sufficiently small $\nu>0$. In summary, the implication
\begin{align}
\label{Lem:ZetaDeltaBounds.I4}
(P_{0\mp}) \quad\text{and}\quad
\quad \zeta\at{t_3}\leq C\nu^2
\qquad\implies\qquad \calD\at{t}\leq\tau^\beta \quad \text{for some}\quad t\in\ccinterval{t_3}{t_4}
\end{align}
holds for all sufficiently small $\nu>0$. 
\par
In the case 
\eqref{Lem:ZetaDeltaBounds.TMZ}
we assume that 
\begin{align*}
\zeta\at{t_3}=\xi\at{t_3}+m_0\at{t_3}+m_+\at{t_3}\leq C\nu^2.
\end{align*}
Similar to the above discussion for the case \eqref{Lem:ZetaDeltaBounds.CM}, we exploit  
Lemma \ref{Lem:PS.NoMassInSpinodalRegion2} and Lemma \ref{Lem:ChangeOfMasses} -- see  \eqref{Lem:ZetaDeltaBounds.I0.xEqn2} and \eqref{Lem:ZetaDeltaBounds.I0.xEqn3} with $t_1=t_3$,
$t_2=t_4$ --
and show that there is a constant $C_{12}$ such that
\begin{align*}
\xi\at{t_4}+m_0\at{t_4}+m_+\at{t_4}\leq C_{12}\nu^2
\end{align*}
holds for all sufficiently small $\nu>0$. From this and Lemma \ref{Lem:XiControlsL} we further infer that 
there is a constant $C_{13}$ such that
\begin{align*}
\abs{X_-\bat{\si\at{t_4}}-X_-\bat{\si\at{t_3}}}\leq \abs{\ell\at{t_4}-\ell\at{t_3}}+C_{13}\nu,
\end{align*}
and the properties of $X_-$ and $\ell$ imply that $t_4-t_3\geq c_{14}$ holds for all sufficiently small $\nu>0$ 
and some constant $c_{14}$.  In particular, using
$\int_{t_3}^{t_4}\calD\at{t}\dint{t}\leq {C_{11}}\tau$ once more, we show that the implication
\begin{align}
\label{Lem:ZetaDeltaBounds.I5}
(P_{\mp0}) \quad\text{and}\quad
\quad \zeta\at{t_3}\leq C\nu^2
\qquad\implies\qquad \calD\at{t}\leq\tau^\beta \quad \text{for some}\quad t\in\ccinterval{t_3}{t_4}
\end{align}
holds for all sufficiently small $\nu>0$. We finally recall that 
\begin{align}
\label{Lem:ZetaDeltaBounds.Time}
(P_{0\mp}) \quad\text{or}\quad(P_{\mp0}) \quad\implies\quad
t_4-t_3\geq\min\{c_{10},\,c_{14}\}>0
\end{align}
holds for all sufficiently small $\nu>0$.
\par
{\ul{\emph{Part 3:}}} 
We finally return to the time interval $\ccinterval{t_0}{T}$ and
establish a recursive argument that allows us to finish the proof after a finite number of iterations.
More precisely, we show that we are either done
since the assertion \eqref{Lem:ZetaDeltaBounds.I0} is satisfied for given $t_0<T$ or can replace $t_0$ by a larger time 
$\bar{t}_0\in\oointerval{t_0}{T}$ with  
\begin{align}
\label{Lem:ZetaDeltaBounds.Cond}
\sup\limits_{t\in\ccinterval{t_0}{\bar{t}_0}}\zeta\at{t}\leq C\nu^2
\qquad\text{and}\qquad
\calD\at{\bar{t}_0}\leq\tau^\beta.
\end{align}
Suppose at first that $\si\at{t_0}\in J_{-2}$. If $\si\at{t}\in J_{-2}\cup J_{-1}$ holds for
all $t\in\ccinterval{t_0}{T}$, then we are done as  \eqref{Lem:ZetaDeltaBounds.I1} with $t_1=t_0$ and 
$t_2=T$ implies \eqref{Lem:ZetaDeltaBounds.I0}. Otherwise
we consider the times
\begin{align*}
t_4:=\inf\Big\{t\in\ccinterval{t_0}{T}\;:\;\si\at{t}=\si_\#-\tfrac12\eps\Big\},\qquad
t_3:=\sup\Big\{t\in\ccinterval{t_0}{t_4}\;:\;\si\at{t}=\si_\#-\eps\Big\},
\end{align*}
which are well-defined as $\si$ is continuous. By construction,
the intervals $\ccinterval{t_0}{t_3}$ and $\ccinterval{t_3}{t_4}$ 
corresponds to the cases \eqref{Lem:ZetaDeltaBounds.CM} and \eqref{Lem:ZetaDeltaBounds.TMZ}, respectively,
and the existence of $\bar{t}_0\in\ccinterval{t_3}{t_4}$ with \eqref{Lem:ZetaDeltaBounds.Cond} is a consequence of
\eqref{Lem:ZetaDeltaBounds.I1} and \eqref{Lem:ZetaDeltaBounds.I5}. Similarly, 
the case $\si\at{t_0}\in J_{+2}$ can be traced back to the cases
\eqref{Lem:ZetaDeltaBounds.CP} and \eqref{Lem:ZetaDeltaBounds.TPZ},
and $\bar{t}_0$ is provided by \eqref{Lem:ZetaDeltaBounds.I2} and \eqref{Lem:ZetaDeltaBounds.I5}.
\par
Now suppose that $\si\at{t_0}\in J_{0}$. If $\si\at{t}\in J_{-1}\cup J_{0}\cup J_{+1}$ holds for all $t\in\ccinterval{t_0}{T}$,
then we are done as \eqref{Lem:ZetaDeltaBounds.I0} follows from \eqref{Lem:ZetaDeltaBounds.I3} with
$t_1=t_0$ and $t_2=T$. Otherwise we find
times $t_3<t_4$ such that $\ccinterval{t_0}{t_3}$ corresponds to \eqref{Lem:ZetaDeltaBounds.CZ} and $\ccinterval{t_3}{t_4}$ to either \eqref{Lem:ZetaDeltaBounds.TZM} or \eqref{Lem:ZetaDeltaBounds.TZP}, and the existence of 
$\bar{t}_0$ is implied by \eqref{Lem:ZetaDeltaBounds.I3} and \eqref{Lem:ZetaDeltaBounds.I4}.
\par
For $\si\at{t_0}\in J_{\pm 1}$, we are either done via $\si\at{t}\in{J_{\pm1}}$ for all $\ccinterval{t_0}{T}$, or we
find a time
$t_1$ such that $\si\at{t}\in{J_{\pm1}}$ for all $t\in\ccinterval{t_0}{t_1}$ and 
either $\si\at{t_1}=\si_\#\pm\eps$ or $\si\at{t_1}=\si_\#\pm\tfrac12\eps$. 
Depending on the value of $\si\at{t_1}$ we can now argue as for
$\si\at{t_0}\in{J_{\pm2}}$ or $\si\at{t_0}\in{J_{0}}$. 
\par
We have now established the recursive argument from above and  argue iteratively. In particular, between 
two subsequent iterations $t_0\leadsto \bar{t}_0$ and $\bar{t}_0\leadsto \bar{\bar{t}}_0$ the system runs through 
at least one of the four cases $\{P_{-0}$, $P_{0-}$, $P_{0+}$, $P_{+0}\}$ and 
\eqref{Lem:ZetaDeltaBounds.Time} provides a lower bound for $\bar{\bar{t}}_0-t_0$.
\end{proof}
%
%
\subsection{Continuity estimates \texorpdfstring{for $\si$}{for the multiplier}}
\label{sect:Continuity}
%
As further key ingredient to the derivation of the limit model we next show
that the dynamical multiplier $\si$ from \eqref{Eqn:FP2} is, up to some error terms, globally Lipschitz continuous in time.
These estimates become important when establishing the limit $\nu\to0$
because they imply the existence of convergent subsequences as well as the Lipschitz continuity of any limit function.
\begin{lemma}[Lipschitz continuity of $\si$ up to small error terms]
\label{Lem:LipschitzSigma}
For each $\eps\in\ccinterval{0}{\eps_*}$ there exist constants $\alpha$ and $C$, which depend on $\eps$ but not on $\nu$, as well as a constant $C_0$, which is independent of both $\eps$ and $\nu$, such that
\begin{align*}
\babs{\si\at{t_2}-\si\at{t_1}}\leq C_0\bat{\abs{t_2-t_1}+\eps}+C\bat{\tau^\alpha+\sup_{t\in\ccinterval{t_1}{t_2}}\sqrt{\zeta\at{t}}}
\end{align*}
holds for all $\tsnu\leq t_1\leq t_2\leq T$ and all sufficiently small $\nu$,
where $\eps_*$ and $t_*$
have been introduced in \eqref{Eqn:Conv.Delta} and Proposition \ref{Lem:PropertiesOfSolutions}.
\end{lemma}
\begin{proof}
\emph{\ul{Step 0:}}
We introduce appropriate cut offs in $\si$-space. More precisely, we define
\begin{align*}
\si_{-2}\at{t}:=\Pi_{\oointerval{-\infty}{\si_\#-\eps}}\si\at{t},\qquad
\si_{0}\at{t}:=\Pi_{\oointerval{\si_\#+\eps}{\si^\#-\eps}}\si\at{t},\qquad
\si_{+2}\at{t}:=\Pi_{\oointerval{\si^\#+\eps}{+\infty}}\si\at{t},
\end{align*}
as well as
\begin{align*}
\si_{-1}\at{t}:=\Pi_{\oointerval{\si_\#-\eps}{\si_\#+\eps}}\si\at{t},\qquad
\si_{+1}\at{t}:=\Pi_{\oointerval{\si^\#-\eps}{\si^\#+\eps}}\si\at{t},
\end{align*}
where the nonlinear projectors $\Pi_{\oointerval{\ul\si}{\ol\si}}$ are given by
$\Pi_{\oointerval{\ul\si}{\ol\si}}\at\si:=\max\{\min\{\si,\ol\si\},\ul\si\}$.  These definitions
imply
\begin{align}
\label{Lem:LipschitzSigma.AR}
\sum_{j=-2}^{+2}\si_{j}\at{t}=
\si\at{t}+2\bat{\si^\#-\si_\#},
\end{align}
and since $\si$ is (for any given $\nu>0$) continuous in time, 
all projected functions $\si_j$ depend continuously on $t$ as well.
\par
\emph{\ul{Step 1:}}
To show that $\si_0$ is almost Lipschitz continuous, we assume without loss of generality that $\si_0\at{t_1}<\si_0\at{t_2}$
and consider at first the special case of $\si_0\at{t}=\si\at{t}\in\ccinterval{\si_0\at{t_1}}{\si_0\at{t_2}}$ for all 
$t\in\ccinterval{t_1}{t_2}$. Under this assumption, Lemma \ref{Lem:EllAndSigma} provides constants $\alpha$, $C$ and $C_0$ such that
\begin{align}
\label{Lem:LipschitzSigma.EstZ}
\abs{\si_0\at{t_2}-\si_0\at{t_1}}\leq
C_0\abs{\ell\at{t_2}-\ell\at{t_1}}+C\at{\tau^\alpha+\sqrt{\zeta\at{t_1}}+\sqrt{\zeta\at{t_2}}}
\end{align}
holds for all sufficiently small $\nu>0$. In the general case, 
we introduce two times $\hat{t}_1$ and $\hat{t}_2$, which both depend on $\nu$,  by
\begin{align}
\label{Lem:LipschitzSigma.Eqn1}
\hat{t}_1:=\max\big\{t\in\ccinterval{t_1}{t_2}\;:\; \si_0\at{t}=\si_0\at{t_1}\big\},\qquad
\hat{t}_2:=\min\big\{t\in\ccinterval{\hat{t}_1}{t_2}\;:\; \si_0\at{t}=\si_0\at{t_2}\big\},
\end{align}
and notice that the Intermediate Value Theorem (applied to the continuous function $\si_0$)
ensures that $\si_0$ is a bijective map between the intervals
$\ccinterval{\hat{t}_1}{\hat{t}_2}$ and $\ccinterval{\si_0\at{t_1}}{\si_0\at{t_2}}$.
In particular, our result for the special case applied to the interval $\ccinterval{\hat{t}_1}{\hat{t}_2}$
combined with $\abs{\hat{t}_2-\hat{t}_1}\leq\abs{t_2-t_1}$ yields again \eqref{Lem:LipschitzSigma.EstZ}.
\par
\emph{\ul{Step 2:}}
We next derive a Lipschitz estimate for $\si_{+2}$. As above, we suppose that 
$\si_{+2}\at{t_1}<\si_{+2}\at{t_2}$ and consider at first the special case 
of $\si_{+2}\at{t}=\si\at{t}\in\ccinterval{\si_{+2}\at{t_1}}{\si_{+2}\at{t_2}}$ for all 
$t\in\ccinterval{t_1}{t_2}$. From Lemma \ref{Lem:XiControlsL} we then infer that
\begin{align*}
\babs{\ell\at{t_i}-X_+\bat{\si_{+2}\at{t_i}}}\leq C\sqrt{\zeta\at{t_i}},\qquad i=1,2,
\end{align*}
for some constant $C$ and all sufficiently small $\nu>0$, and hence we get
\begin{align}
\label{Lem:LipschitzSigma.Eqn2}
\babs{X_+\bat{\si_{+2}\at{t_2}}-X_+\bat{\si_{+2}\at{t_1}}}\leq
\babs{\ell\at{t_2}-\ell\at{t_1}}+C\at{\sqrt{\zeta\at{t_1}}+\sqrt{\zeta\at{t_2}}}.
\end{align}
On the other hand, thanks to $\si_{+2}\at{t_i}\geq\si^\#+\eps>\si_*$ and the properties of $X_+$ -- see
Remark \ref{Rem:PropertiesOfXi} -- we have
\begin{align*}
\babs{ \si_{+2}\at{t_2}-\si_{+2}\at{t_1}}\leq{C_0}
\babs{X_+\bat{\si_{+2}\at{t_2}}-X_+\bat{\si_{+2}\at{t_1}}},
\end{align*}
and combining this with \eqref{Lem:LipschitzSigma.Eqn2} gives
\begin{align}
\label{Lem:LipschitzSigma.EstP}
\babs{ \si_{+2}\at{t_2}-\si_{+2}\at{t_1}}\leq{C_0}
\babs{\ell\at{t_2}-\ell\at{t_1}}+C\at{\sqrt{\zeta\at{t_1}}+\sqrt{\zeta\at{t_2}}}.
\end{align} 
In the general case we
introduce again two times $\hat{t}_1$ and $\hat{t}_2$ by using
\eqref{Lem:LipschitzSigma.Eqn1} with $\si_{+2}$ instead of $\si_0$, and argue as above.
Moreover, the estimate 
\begin{align}
\label{Lem:LipschitzSigma.EstM}
\babs{ \si_{-2}\at{t_2}-\si_{-2}\at{t_1}}\leq{C_0}
\babs{\ell\at{t_2}-\ell\at{t_1}}+C\at{\sqrt{\zeta\at{t_1}}+\sqrt{\zeta\at{t_2}}}.
\end{align} 
can be proven similarly.
\par
\emph{\ul{Step 3:}} By construction, we have
\begin{align*}
\abs{\si_{-1}\at{t_2}-\si_{-1}\at{t_1}}\leq 2\eps,\qquad
\abs{\si_{+1}\at{t_2}-\si_{+1}\at{t_1}}\leq 2\eps,
\end{align*}
and combining this with
 the algebraic relation \eqref{Lem:LipschitzSigma.AR}
as well as the estimates \eqref{Lem:LipschitzSigma.EstZ}, \eqref{Lem:LipschitzSigma.EstP}, and \eqref{Lem:LipschitzSigma.EstM}, we arrive at 
\begin{align*}
\babs{\si\at{t_2}-\si\at{t_1}}\leq C_0\Bat{\babs{\ell\at{t_2}-\ell\at{t_1}}+\eps}+C\at{\tau^\al+\sup_{t\in\ccinterval{0}{T}}\sqrt{\zeta\at{t}}}.
\end{align*}
Assumption \ref{Ass:Constraint} finally implies
\begin{align*}
\abs{\ell\at{t_2}-\ell\at{t_1}}\leq\abs{t_2-t_1} \sup\limits_{t\in\ccinterval{t_1}{t_2}}\babs{\dot{\ell}{\at{t}}}
\leq{C_0}\abs{t_2-t_1}
\end{align*}
and hence the desired result.
\end{proof}

%
\subsection{Passage to the limit \texorpdfstring{$\nu\to0$}{}}\label{sect:compactness}
%
We finally pass to the limit $\nu\to0$ and verify the validity of the limit model  
as formulated in Definition \ref{Def:LimitModel}. We therefore write 
\begin{align*}
\text{$\tau_\nu$ instead of $\tau$,\quad $\varrho_\nu$ instead
of $\varrho$, \quad $\si_\nu$ instead of $\si$, \quad$m_{j,\,\nu}$ instead of $m_j$,\quad
$\zeta_{\eps,\,\nu}$ instead of $\zeta$, }
\end{align*}
and define the phase fraction by 
$\mu_{\nu}:=m_{+,\,\nu}-m_{-,\,\nu}$.

\begin{theorem}[convergence to limit model along subsequences]
\label{Thm:Compactness} %
There
exists a sequence $\at{\nu_n}_{n\in\Nset}$ with $\nu_n\to0$ as $n\to\infty$ 
as well as 
two Lipschitz functions $\si_0,\,\mu_0\in\fspaceC^{0,\,1}\at{\ccinterval{0}{T}}$ such that
\begin{align}
\label{Thm:Compactness.Conv}
\norm{\si_{\nu_n}-\si_0}_{\fspaceC\at{I}}\quad +\quad
\norm{\mu_{\nu_n}-\mu_0}_{\fspaceC\at{I}}\qquad\xrightarrow{\;\;n\to\infty\;\;}\qquad 0
\end{align}
on each compact interval $I\subset\ocinterval{0}{T}$. Moreover, the convergence
\begin{align*}
\varrho_{\nu_n}\pair{t}{x}\quad \xrightharpoonup{{\;\;n\to\infty\;\;}}\quad 
\frac{1-\mu_0\at{t}}{2}\delta_{X_-\at{\si_0\at{t}}}\at{x}+\frac{1+\mu_0\at{t}}{2}\delta_{X_+\at{\si_0\at{t}}}\at{x}
\end{align*}
holds for all $t>0$ with respect to the weak$\star$ topology and the triple $\triple{\ell}{\si_0}{\mu_0}$ 
is a solution to the limit model in the sense of Definition \ref{Def:LimitModel}.
\end{theorem}

\begin{proof}
\par
\emph{\ul{Convergence of $\si$:}} 
We choose a sequence $\at{\eps_n}_{n\in\Nset}$ with $0<\eps_n<\eps_*$ for all $n$ and
$\eps_n\to0$ as $n\to\infty$. According to 
Lemma \ref{Lem:ZetaDeltaBounds} and Lemma \ref{Lem:LipschitzSigma}, there exist -- for any given $n$ --
positive constant $C_0$, $C_n$, $\alpha_n$,
and $\beta_n<1$ such that
\begin{align*}
\calD_{\nu}\at{t_0}\leq\tau_{\nu}^{\beta_n}\quad\implies\quad
\sup\limits_{t\in\ccinterval{t_0}{T}}\zeta_{\eps_n,\nu}\at{t}\leq C_n\nu^2
\end{align*}
and 
\begin{align*}
\babs{\si_\nu\at{t_2}-\si_\nu\at{t_1}}\leq  C_0\Bat{ \abs{t_2-t_1}+\eps_n}+ C_n\Bat{\tau_{\nu}^{\alpha_n}+\sup\limits_{t\in\ccinterval{t_1}{t_2}}\sqrt{\zeta_{\eps_n,\nu}\at{t}}}
\end{align*}
holds for all $n$, all times $t_0, t_1, t_2\in\ocinterval{0}{T}$, and all sufficiently small $\nu>0$, where $C_0$ is in fact
independent of $n$.  Moreover, making $C_0$ larger (if necessary) we can also assume that
\begin{align*}
{C_0}\tau_\nu\geq \int_{\nu^2\tau_\nu}^T\calD_{\nu}\at{t}\dint{t}\geq \tau_\nu^{\beta_n}\Babs{\big\{t\in\ccinterval{\nu^2\tau_\nu}{T}\;:\; \calD_{\nu}\at{t}>\tau_\nu^{\beta_n}\big\}}
\end{align*}
holds for all $\nu>0$ and $n\in\Nset$,
and hence there exists for any choice of $\nu$ and $n$ a time \begin{align*}
S_{n,\nu}\in\ccinterval{\nu^2\tau_\nu}{\nu^2\tau_\nu+C_0\tau_\nu^\at{1-\beta_n}}\quad\text{such that}\quad \calD_\nu\at{S_{n,\nu}}\leq\tau_\nu^{\beta_n}.
\end{align*}
For each $n$ we next choose 
$\nu_n>0$ sufficiently small such that 
\begin{align*}
\max\big\{C_n\nu_n^2,\;
C_n\tau_{\nu_n}^{\alpha_n}+ C_n^{3/2}\nu_n,\;
\tau_{\nu_n}^{\beta_n},\; \nu_n^2\tau_{\nu_n}+C_0\tau_{\nu_n}^{(1-\be_{\nu_n})}\big\}\leq \eps_n.
\end{align*}
In particular, using the abbreviations $\si_n:=\si_{\nu_n}$, $m_{j,\,n}:=m_{j,\,\nu_n}$, $\zeta_n:=\zeta_{\eps_n,\,\nu_n}$, and
$S_n:=S_{n,\,\nu_n}$ we obtain 
\begin{align}
\label{Thm:Compactness.Eqn0}
\sup\limits_{t\in\ccinterval{S_n}{T}}\zeta_{n}\at{t}\leq \eps_n,\qquad  S_n\leq\eps_n
\end{align}
as well as
\begin{align}
\label{Thm:Compactness.Eqn1}
\babs{\si_n\at{t_2}-\si_n\at{t_1}}\leq  C_0 \abs{t_2-t_1}+\at{C_0+1}\eps_n \quad
\text{for all}\quad t_1, t_2\in\ccinterval{S_n}{T}.
\end{align}
Let $t_0>0$ be fixed and notice that $S_n\leq t_0$ for almost all $n$. A refined version of the 
Arzel\`a-Ascoli Theorem -- see Proposition 3.3.1 in \cite{AGS05} --
guarantees 
the existence of a continuous function $\si_0$ defined on $\ccinterval{t_0}{T}$ along with a not relabeled 
subsequence such that $\norm{\si_n-\si_0}_{\fspaceC\nat{\ccinterval{t_0}{T}}}\to0$ as $n\to\infty$. Moreover,
by the usual diagonal argument
we can extract a further subsequence such that
$\norm{\si_n-\si_0}_{\fspaceC\at{I}}\to0$ for any compact $I\subset\ocinterval{0}{T}$,
and the estimate \eqref{Thm:Compactness.Eqn1} implies that $\si_0$ is Lipschitz continuous on the whole interval $\ccinterval{0}{T}$.
\par
\emph{\ul{Convergence of $\mu$ and $\varrho$:}} In what follows, we denote by $C_0$ any generic constant independent of $n$, and assume (without saying so explicitly) 
that $n$ is sufficiently large. We also 
define
\begin{align*}
\ul\si:=\tfrac12\bat{\si_*+\si_\#},\qquad
\ol\si:=\tfrac12\bat{\si^\#+\si^*},
 \end{align*}
and introduce a bounded function
$\mu_0:\ccinterval{0}{T}\to\Rset$ as follows:
For any $t$ with $\si_0\at{t}\in \ocinterval{-\infty}{\ul\si}$ we set 
$\mu_0\at{t}:=-1$ and notice that
$\sum_{j\in\{-,0,+\}}m_{j,\,n}\at{t}=1$ gives $\mu_n\at{t}-\mu_0\at{t}=2m_{+,n}\at{t}+m_{0,n}\at{t}$.
Moreover, $\eps\leq\eps_*\leq\ul{\si}-\si_*$ combined with \eqref{Eqn:DefZeta} implies 
\begin{align*}
m_{0,\,n}\at{t}+m_{+,\,n}\at{t}\leq \zeta_n\at{t},
\end{align*}
and thus we find
\begin{align}
\label{Thm:Compactness.Eqn3a}
\babs{\mu_n\at{t}-\mu_0\at{t}}\leq C_0\zeta_n\at{t}.
\end{align}
Similarly, for any $t$ with $\si_0\at{t}\in \cointerval{\ol\si}{+\infty}$ we set $\mu_0\at{t}:=+1$ and
find again \eqref{Thm:Compactness.Eqn3a}. For times $t$ with $\si_0\at{t}\in\oointerval{\ul\si}{\ol\si}$, we employ
Lemma \ref{Lem:XiControlsL} -- applied with $\eps=\min\{\si^*-\ol\si,\ul\si-\si_*\}$, which does not depend on $n$ --
to find
\begin{align}
\label{Thm:Compactness.Eqn4a}
\Babs{\ell\at{t}-\frac{1-\mu_n\at{t}}{2}X_-\bat{\si_n\at{t}}-\frac{1+\mu_n\at{t}}{2}X_+\at{\si_n\at{t}}}\leq
C_0\zeta_n\at{t}.
\end{align}
We then define $m_0\at{t}\in\Rset$  as the unique solution to
\begin{align}
\label{Thm:Compactness.Eqn4b}
\ell\at{t}=\frac{1-\mu_0\at{t}}{2}X_-\bat{\si_0\at{t}}-\frac{1+\mu_0\at{t}}{2}X_+\at{\si_0\at{t}},
\end{align}
and show that the properties of $X_\pm$ from Remark \ref{Rem:PropertiesOfXi} imply
\begin{align}
\label{Thm:Compactness.Eqn3b}
\babs{\mu_n\at{t}-\mu_0\at{t}}\leq C_0\Bat{\zeta_n\at{t} + \abs{\si_n\at{t}-\si_0\at{t}}}
\end{align}
thanks to \eqref{Thm:Compactness.Eqn3a}, \eqref{Thm:Compactness.Eqn4a}, and \eqref{Thm:Compactness.Eqn4b}.
\par
In summary, we have now defined
$\mu_0\at{t}$ for all $t\in\ccinterval{0}{T}$, and \eqref{Thm:Compactness.Eqn3a}, \eqref{Thm:Compactness.Eqn3b} combined with
\eqref{Thm:Compactness.Eqn0} and $S_n\to0$ ensure that $\mu_0$ depends in fact continuously on $t$ and satisfies
$\norm{\mu_n-\mu_0}_{\fspaceC\at{I}}\to0$ as $n\to\infty$ on any compact interval $I\subset\ocinterval{0}{1}$. 
Moreover, the claimed weak$\star$ convergence of $\varrho_{n}$ is a direct consequence of Remark \ref{Rem:XiControlsMoments} and
$\xi_n\at{t}+m_{0,\,n}\at{t}\leq\eps_n\to0$.
\par
\emph{\ul{Verification of limit dynamics:}}
Using Lemma \ref{Lem:XiControlsL}
once more 
we find
\begin{align*}
X_-\bat{\si\at{t}}=\ell\at{t}
\quad\text{for} \quad 
\si_{0}\at{t}<\si_\#
,\qquad
X_+\bat{\si\at{t}}=\ell\at{t}\quad\text{for} \quad 
\si_{0}\at{t}>\si^\#.
\end{align*}
and hence
\begin{align}
\label{Thm:Compactness.Eqn6}
\si_{0}\at{t}<\si_\#\implies \mu_0\at{t}=-1,\qquad\qquad
\si_{0}\at{t}>\si^\#\implies\mu_0\at{t}=+1
\end{align}  
for all $t\in\ccinterval{0}{T}$.
Combining these results with \eqref{Thm:Compactness.Eqn4b} we readily verify the algebraic relations
\begin{align}
\label{Thm:Compactness.Eqn5}
\btriple{\ell\at{t}}{\si_0\at{t}}{\mu_0\at{t}}\in\Om,\qquad \ell\at{t}=\calL\bpair{\si_0\at{t}}{\mu_0\at{t}},
\end{align}
where $\Om$ and $\calL$ are defined in \eqref{Eqn:LimDyn.Def1}+\eqref{Eqn:LimDyn.Def2}. The pointwise relations \eqref{Thm:Compactness.Eqn5} combined with
$\si_0,\ell\in\fspaceC^{0,\,1}\at{\ccinterval{0}{T}}$ and the smoothness of the functions $X_i$ also
imply that $\mu_0$ belongs in fact to $\fspaceC^{0,\,1}\at{\ccinterval{0}{T}}$.
\par
The dynamical relations from Definition \ref{Def:LimitModel}
follow because passing to the limit $n\to\infty$ in Lemma \ref{Lem:ChangeOfMasses} shows that the implications
\begin{align}
\label{Thm:Compactness.Eqn7}
\si_0\at{t}<\si^\# \implies\dot{\mu}_0\at{t}\leq0,\qquad \qquad 
\si_0\at{t}>\si_\# \implies\dot{\mu}_0\at{t}\geq0
\end{align}
hold for almost all $t\in\ccinterval{0}{T}$. In particular, the combination of \eqref{Thm:Compactness.Eqn6} and \eqref{Thm:Compactness.Eqn7} gives
\begin{align*}
\si\at{t}\notin\{\si_\#,\,\si^\#\}\implies \dot{\mu}_0\at{t}=0,\qquad \dot{\mu}_0\at{t}>0\implies \si\at{t}=\si^\#,\qquad
 \dot{\mu}_0\at{t}<0\implies \si\at{t}=\si_\#
\end{align*}
for almost all $t\in\ccinterval{0}{T}$.
\end{proof}

Notice that Theorem \ref{Thm:Compactness} neither implies $\si_{\nu_n}\at{0}\to \si_0\at{0}$ nor
$\mu_{\nu_n}\at{0}\to \mu_0\at{0}$. This is not surprising because we expect, as explained within \S\ref{sect:prelim},
that each solution for $\nu>0$ and generic initial data exhibits a small initial transition layer. More precisely, if the mass at time $t=0$ is not yet concentrated in
two narrow peaks, the systems undergoes a fast initial relaxation process during which $\si_\nu$ and $\mu_\nu$ may change rapidly.
After this process, that means at some time of order at most ${\tau_\nu^{1-\beta}}$, $0<\beta<1$, the dissipation is of order ${\tau_\nu^\beta}$ and our peak stability estimates imply that afterwards the state $\varrho_\nu$ can be described by two narrow peaks, which in turn are either transported by the dynamical constraint
or exchange mass by a Kramers-type phase transition. 
\par
The above arguments reveal that the limit functions
$\si_0$ and $\mu_0$ can (and in general they do) depend on the subsequence, or equivalently, on the 
microscopic details of the initial data. For well-prepared initial data, however, 
we can improve our result as follows.

\begin{theorem}[convergence for well-prepared initial data]
\label{Thm:Convergence}
For well-prepared initial data in the sense of Definition \ref{Def:WellPreparedData}, we can choose
$I=\ccinterval{0}{T}$ in \eqref{Thm:Compactness.Conv}. In particular, 
the whole family  $\bat{\triple{\ell}{\si_\nu}{\mu_\nu}}_{\nu>0}$ converges as $\nu\to0$ to a solution of the limit 
model.
\end{theorem}
\begin{proof}
By assumption, there exist values $\si_\ini\in\Rset$ and $\mu_\ini\in\ccinterval{-1}{1}$ such that
$\si_{\nu}\at{0}\to \si_\ini$ as well as $\mu_{\nu}\at{0}\to \mu_\ini$ as $\nu\to0$.
Now let $\at{\pair{\si_n}{\mu_n}}_{n\in\Nset}$ be any sequence as provided by
Theorem \ref{Thm:Compactness}. Since the initial data are well-prepared, 
we can choose $S_{n}=0$ in the proof of Theorem \ref{Thm:Compactness}, see also Remark \ref{Rem:WellPreparedInitialData}. This implies
\begin{align*}
\norm{\si_n-\si_0}_{\fspaceC\at{\ccinterval{0}{T}}}+
\norm{\mu_n-\mu_0}_{\fspaceC\at{\ccinterval{0}{T}}}\quad\xrightarrow{\nu\to0}\quad 0
\end{align*}
and hence $\si_0\at{0}=\si_\ini=\lim_{\nu\to0}\si_\nu\at{0}$ and  $\mu_{0}\at{0}= \mu_\ini=\lim_{\nu\to0}\mu_\nu\at{0}$. 
Since the 
limit model has precisely one solution with initial data $\pair{\si_\ini}{\mu_\ini}$, see Proposition \ref{App:Prop:WellPosednessLimitModel}, we conclude that each sequence from
Theorem \ref{Thm:Compactness} has the same limit, and standard arguments (compactness+uniqueness of accumulation points=convergence)
provide the claimed convergence.
\end{proof}
%
%
%
\appendix
\section{Solutions to the nonlocal Fokker-Planck equation}\label{app:ExistenceAndUniqueness}
In this appendix we show that the initial value problem to the nonlocal Fokker-Planck equation \eqref{Eqn:FP1} and
\eqref{Eqn:FP2} is well-posed
with state space
\begin{align*}
\fspace{P}^2\at\Rset:=\Big\{\text{probability measures on $\Rset$ with bounded variance}\Big\}.
\end{align*}
To this end we suppose that the final time $T$ with $0<T<\infty$ is fixed 
and denote by $\ell\in\fspaceC^1\at{\ccinterval{0}{T}}$ and 
$\varrho_\ini\in\fspace{P}^2\at\Rset$ a given control function and some prescribed initial data, respectively.
Moreover, in what follows we allow for arbitrary
(i.e., uncoupled) parameters $\nu,\tau>0$.
\par
Our existence and uniqueness proof is
based on a fixed point argument and constructs
solutions to the nonlocal problem by iterating the solution operator of a linear PDE with a nonlinear integral operator.
The idea is as follows. For any $\si\in\fspaceC\at{\ccinterval{0}{T}}$, we denote by
$\calR\ato\si$
the solution to the linear PDE \eqref{Eqn:FP1}. In other words, for each $\si$ the function
$\calR\ato\si$ 
satisfies the initial value problem
\begin{align}
\label{Eqn:LinearFP}
\tau\partial_t \calR\ato\si\pair{t}{x}=\nu^2\partial_x^2\calR\ato\si\pair{t}{x}
+\partial_x\Bat{\bat{H^\prime\at{x}-\si\at{t}}\calR\ato\si\pair{t}{x}},\qquad
\calR\ato\si\pair{0}{x}=\varrho_\ini\at{x}
\end{align}
with $x\in\Rset$ and $t\in\ccinterval{0}{T}$. Using $\calR$ we now observe that 
the dynamical constraint 
\eqref{Eqn:FP2} is equivalent to the 
fixed point equation $\si=\calS\ato\si$,
where the operator $\calS$ is defined by 
\begin{align*}
\calS\ato{\si}\at{t}:=\int_\Rset H^\prime\at{x}\calR\ato{\si}\pair{t}{x}\dint{x} + \tau\dot\ell\at{t}.
\end{align*}
Notice that \eqref{Eqn:FP2} implies \eqref{Eqn:FP2P} if and only if the initial data are admissible in the sense of
$\int_\Rset x\varrho_\ini\at{x}\dint{x}=\ell\at{0}$.
\bigpar
Our first result in this section employs Banach's Fixed Point Theorem in order to show that $\calS$ admits a unique fixed point
in the space of continuous functions.
\begin{proposition}[existence and uniqueness of solutions]
\label{Prop:ExistenceAndUniqueness}
For any $\tau,\nu>0$, there exists a unique solution $\varrho$ to \eqref{Eqn:FP1}\emph{+}\eqref{Eqn:FP2} in the time-space-domain $\ccinterval{0}{T}\times\Rset$. In particular,
$\varrho$ is smooth in $\ocinterval{0}{T}\times\Rset$ as well as continuous in $t$ with respect to 
the weak$\star$ topology in $\fspace{P}^2\at\Rset$, and
$\si$ is continuously differentiable on $\ccinterval{0}{T}$.
\end{proposition}
\begin{proof}
\par
{\ul{\emph{Operators and moment balances:}}}
For given $\si\in\fspaceC\at{\ccinterval{0}{1}}$, the existence, uniqueness and regularity of $\calR\ato{\si}$ can be
established by adapting standard methods. For instance, \cite[Section 6, Corollary 4.2 and Theorem 4.5]{Fri75} guarantees the existence and uniqueness of smooth solutions under 
slightly stronger assumptions, namely the boundedness of $H^\prime$. For linearly increasing $H^\prime$,
we are only aware of results concerning the stochastic Langevin equation $\tau\dint{x}=\bat{\si\at{t}-H^\prime\at{x}}\dint{t}+\sqrt{2}\nu\dint{W}$; see for instance \cite[Section 5, Theorem 1.1]{Fri75}. The solution $\calR\ato\si$ to \eqref{Eqn:LinearFP} is then provided by the corresponding probability distribution function for finding a particle at $\pair{t}{x}$. We also refer 
the reader to \cite{JKO98,ASZ09}, which study the existence and uniqueness problem for similar equations in the framework of Wasserstein gradient flows, and to \cite{Ebe13}, which generalizes this approach to \eqref{Eqn:FP1}+
\eqref{Eqn:FP2P}. 
\par
Using the PDE \eqref{Eqn:LinearFP} as well as integration by parts we deduce that $\varrho=\calR\ato{\si}$ satisfies the moment balance
\begin{align}
\label{Prop:ExistenceAndUniqueness.Eqn0}
\tau\frac{\dint}{\dint{t}}\int_\Rset \psi\at{x}\varrho\pair{t}{x}\dint{x}=
\nu^2\int_\Rset\psi^{\prime\prime}\at{x}\varrho\pair{t}{x}\dint{x}+
\int_\Rset\psi^\prime\at{x}
\bat{\si\at{t}-H^\prime\at{x}} \varrho\pair{t}{x}\dint{x}
\end{align}
for any weight function $\psi$ with $\abs{\psi\at{x}}+\abs{\psi^{\prime\prime}\at{x}}\leq {C} \at{1+x^2}$
and $\abs{\psi^\prime\at{x}}\leq C\at{1+\abs{x}}$ for all $x\in\Rset$, and this implies 
the desired continuity of moments with respect to $t$. For $\psi\at{x}=1$  we obtain
$\int\varrho\pair{t}{x}\dint{x}=1$ and with $\psi\at{x}=1+x^2$ we verify that
\begin{align*}
\int_\Rset x^2\varrho\pair{t}{x}\dint{x}\leq \at{1+\int_\Rset x^2\varrho_\ini\at{x}\dint{x}}
\exp\at{C\frac{1+\nu^2+\norm{\si}_\infty}{\tau}t}
\end{align*} 
holds for all $t\in\ccinterval{0}{T}$, where we used that
$\abs{H^\prime\at{x}}$ grows at most linearly as $x\to\pm\infty$ according to Assumption \ref{Ass:Potential}.
Moreover, the choice $\psi\at{x}=H^\prime\at{x}$ reveals that the operator $\calS$ is well defined.
\par
{\ul{\emph{Lipschitz estimates:}}}
We next consider two functions $\si_1,\si_2\in\fspaceC\at{\ccinterval{0}{T}}$, abbreviate $\varrho_i:=\calR\ato{\si_i}$, and introduce
$R_1$ and $R_2$ by
\begin{align*}
R_i\pair{t}{x}:=\int_{-\infty}^x\varrho_i\pair{t}{y}\dint{y}.
\end{align*}
The function $R:=R_2-R_1$ then satisfies
\begin{align*}
\tau\partial_{t}R\pair{t}{x}=\nu^2\partial_x^2R\pair{t}{x}+ 
\bat{H^{\prime}\at{x}-\si_2\at{t}}\partial_x R\pair{t}{x}
-
\bat{\si_2\at{t}-\si_1\at{t}}\varrho_1\pair{t}{x}.
\end{align*}
In view of $\varrho_i\pair{t}{\cdot}\in\fspace{P}^2\at\Rset$ we verify 
\begin{align*}
\babs{R\pair{t}{x}}=\abs{\int_{x}^{\infty}\varrho_2\pair{t}{y}-\varrho_1\pair{t}{y}\dint{y}}\leq \frac{1}{x^2}
\int_{x}^{\infty}y^2\abs{\varrho_2\pair{t}{y}-\varrho_1\pair{t}{y}}\dint{y}=\Do{x^{-2}}
\end{align*}
for all $x>0$, and since a similar estimate holds for $x<0$ we verify
$R\pair{t}{\cdot}\in\fspaceL^1\at{\Rset}$ for all $t$ as well as 
\begin{align*}
\Babs{\calS\ato{\si_2}\at{t}-\calS\ato{\si_1}\at{t}}=\abs{\int_\Rset H^\prime\at{x}\partial_xR\pair{t}{x}\dint{x}}=
\abs{\int_\Rset H^{\prime\prime}\at{x}R\pair{t}{x}\dint{x}}
\leq{C}\int_\Rset \abs{R\pair{t}{x}}\dint{x}.
\end{align*}
In order to establish an $\fspaceL^1$-bounds for $R$, we now fix $\eps>0$ and approximate the modulus function
by
$h_\eps\at{r}:=\sqrt{\eps+r^2}$. Thanks to
$-1\leq h_\eps^\prime\at{r}\leq{1}$ and 
$
h_\eps^{\prime\prime}\at{r}\geq0$ for all $r\in\Rset$,
we obtain the moment estimate
\begin{align*}
\tau\frac{\dint}{\dint{t}}\int_\Rset h_\eps\bat{R\pair{t}{x}}\dint{x}&\leq
-\int_\Rset H^{\prime\prime}\at{x}h_\eps\bat{R\pair{t}{x}}\dint{x}
-
\bat{\si_2\at{t}-\si_1\at{t}}\int_\Rset h_\eps^\prime\bat{R\pair{t}{x}} \varrho_1\pair{t}{x}
\\&
\leq
C \int_\Rset  h_\eps\bat{R\pair{t}{x}}\dint{x}+\Babs{\si_2\at{t}-\si_1\at{t}},
\end{align*}
where $C:=\sup_{x\in\Rset}\abs{H^{\prime\prime}\at{x}}$. Using the comparison principle for ODEs and
passing to the limit $\eps\to0$ we therefore get
\begin{align*}
\int_\Rset\babs{R\pair{t}{x}}\dint{x}\leq \tau^{-1}\exp\at{C{\tau}^{-1}t} \int_0^t\Babs{\si_2\at{s}-\si_1\at{s}}\dint{s},
\end{align*}
where we used that $R\pair{0}{\cdot}=0$ holds by construction. 
\par
{\ul{\emph{Fixed point argument:}}}
The estimates derived so far
ensure that
\begin{align*}
\Babs{\calS\ato{\si_2}\at{t}-\calS\ato{\si_1}\at{t}}= C \int_0^t\Babs{\si_2\at{s}-\si_1\at{s}}\dint{s}
\end{align*}
holds for some constant $C$ depending on $\nu$, $\tau$, $T$, $H$, and the initial data $\varrho_\ini$. Consequently,
$\calS$ is contractive with respect to  
$\norm{\si} := \sup_{t\in\ccinterval{0}{T}} \exp\at{-2C t} \abs{\sigma\at{t}}$, which is 
equivalent to the standard norm in $\fspace{C}\at{\ccinterval{0}{T}}$. The existence of a unique fixed point is therefore granted by Banach's Contraction Principle.
Now suppose that $\calS\ato{\si}=\si$. From 
\eqref{Prop:ExistenceAndUniqueness.Eqn0} with $\psi\at{x}=H^\prime\at{x}$ and $\psi\at{x}=x$ we then conclude that
$\si$ is continuously differentiable and that \eqref{Eqn:FP2} is satisfied, respectively.
\end{proof}

We finally derive some bounds for the solutions 
to the nonlocal Fokker-Planck equation \eqref{Eqn:FP1}+\eqref{Eqn:FP2} which hold for all sufficiently small parameters $\tau$ and $\nu$.

\begin{proposition}[uniform bounds for solutions]
\label{Prop:UniformBounds}
Suppose that $0<\nu\leq \ol\nu$ and 
$0<\tau\leq\ol\tau$.
Then, each solution from 
Proposition \ref{Prop:ExistenceAndUniqueness} satisfies
\begin{align*}
\sup\limits_{t\in\ccinterval{0}{T}}\at{\babs{\si\at{t}}+\int_\Rset x^2\varrho\pair{t}{x}\dint{x}}\leq{C}
\end{align*}
and
\begin{align*}
\sup\limits_{t\in\ccinterval{\nu^2\tau }{T}}\norm{\varrho\pair{t}{\cdot}}_\infty
\leq\frac{C}{\nu^2},\qquad
\int_{\nu^2\tau}^T\calD\at{t}\dint{t}\leq C \tau ,
\end{align*}
where the constant $C$ is independent of $\tau$ and $\nu$ but depends on $H$, $\bar\tau$,
$\bar\nu$, $\ell$, $T$, and $\int_\Rset x^2\varrho_\ini\at{x}\dint{x}$.
\end{proposition}

\begin{proof}
\ul{\emph{Moment estimates:}}
Due to the dynamical constraint \eqref{Eqn:FP2}, the moment balance \eqref{Prop:ExistenceAndUniqueness.Eqn0} with $\psi\at{x}=x^2$ implies
\begin{align*}
\tau\frac{\dint}{\dint{t}}\int_\Rset x^2\varrho\pair{t}{x}\dint{x}\leq 2\nu^2+2\norm{\si}_\infty\norm{\ell}_\infty+2C-
2c \int_\Rset x^2\varrho\pair{t}{x}\dint{x},
\end{align*}
where $c$ and $C$ are chosen such that $xH^\prime\at{x}\geq cx^2-C$ holds for all $x\in\Rset$. Employing
the comparison principle for scalar ODEs we therefore find
\begin{align*}
\int_\Rset x^2\varrho\pair{t}{x}\dint{x}&\leq  \max\left\{\frac{\nu^2+\norm{\si}_\infty\norm{\ell}_\infty+C}{c},\,
\int_\Rset x^2\varrho_\ini\at{x}\dint{x}\right\}\leq {C}\bat{1+\norm{\si}_\infty}.
\end{align*}
Moreover, by applying H\"older's inequality to \eqref{Eqn:FP2} we get
\begin{align*}
\babs{\si\at{t}}&\leq \tau\babs{\dot\ell\at{t}}+\at{\int_\Rset \abs{H^\prime\at{x}}^2\varrho\pair{t}{x}\dint{x}}^{1/2}
\at{\int_\Rset\varrho\pair{t}{x}\dint{x}}^{1/2}\\
&\leq
C+C\at{\int_\Rset x^2\varrho\pair{t}{x}\dint{x}}^{1/2},
\end{align*}
where $C$ is some constant independent of $\tau$ and $\nu$. The combination of both estimates
gives
\begin{align*}
\norm{\si}_\infty\leq C\sqrt{1+\norm{\si}_\infty},
\end{align*}
and the desired moment bounds follow immediately. 
\par
\ul{\emph{$\fspaceL^\infty$-estimate after waiting time $\nu^2\tau$:}} Parabolic regularity theory implies that $\norm{\varrho\pair{t}{\cdot}}_\infty$
is well-defined for all $t>0$ but it remains to understand how this quantity depends on $t$ and the parameters $\tau$, $\nu$. To this end we  fix $t_0$ with $0< t_0<T$, consider the function
\begin{align*}
M_{t_0}\at{t}:=\sup\limits_{0\leq{s}\leq{t}}\norm{\sqrt{s}\varrho\pair{t_0+s}{\cdot}}_\infty,
\end{align*}
and 
denote by $C$ any generic constant that is independent of $\tau$, $\nu$ and $t_0$.
Using the rescaled heat kernel
\begin{align*}
K\pair{t}{x}:=\sqrt{\frac{\tau}{4\pi\nu^2t}}\exp\at{-\frac{\tau x^2}{4\nu^2t}},
\end{align*}
as well as Duhamel's Principle, any solution to \eqref{Eqn:FP1}+\eqref{Eqn:FP2} can be written as
\begin{align*}
\varrho\pair{t_0+t}{x}=I_{1,\,t_0}\pair{t}{x}+I_{2,\,t_0}\pair{t}{x},
\end{align*}
where 
\begin{align*}
I_{1,t_0}\pair{t}{x}:=\int_\Rset K\pair{t}{x-y}\varrho\pair{t_0}{y} \dint{y}
\end{align*}
and
\begin{align*}
I_{2,\,t_0}\pair{t}{x}:=\frac{1}{\tau}\int_0^t\int_\Rset K_x\pair{t-s}{x-y}f\pair{t_0+s}{y} \dint{y}\dint{s},\qquad
f\pair{t}{x}:=\bat{H^\prime\at{x}-\si\at{t}}\varrho\pair{t}{x}.
\end{align*}
The first term can be estimated by
\begin{align*}
\babs{I_{1,\,t_0}\pair{t}{x}}\leq \norm{K\pair{t}{\cdot}}_\infty 
\int_\Rset \varrho\pair{t_0}{y}\dint{y}\leq \frac{C}{\nu}\sqrt{\frac{\tau}{ t}},
\end{align*}
whereas for the second term we employ H\"older's inequality to find
\begin{align*}
\babs{I_{2,\,t_0}\pair{t}{x}}\leq \frac{1}{\tau}\int_0^t \at{\int_\Rset K_x\pair{t-s}{y}^2\dint{y}}^{1/2} \at{\int_\Rset f\pair{t_0+s}{y}^2\dint{y}}^{1/2} .
\end{align*}
By direct computations we verify 
\begin{align*}
\int_\Rset K_x\pair{t-s}{y}^2\dint{y}=\at{\frac{\tau}{\nu^2 \at{t-s}}}^{3/2}\at{\frac{1}{2\pi}\int_\Rset\abs{y}^2\exp\at{-2y^2}\dint{x}},
\end{align*}
and using  $\abs{H^\prime\at{x}}\leq {C}\at{1+\abs{x}}$, $\int_\Rset\varrho\pair{t}{x}\dint{x}=1$ as well
as the uniform moment bounds derived above we get
\begin{align*}
\int_\Rset f\pair{t_0+s}{y}^{2}\dint{y}&\leq C\norm{\varrho\pair{t_0+s}{\cdot}}_\infty\at{\babs{\si\at{t_0+s}}^2+1+\int_\Rset y^2\varrho\pair{t_0+s}{y}\dint{y}}\\
&\leq {C}s^{-1/2}M_{t_0}\at{s}.
\end{align*}
The latter three estimates imply
\begin{align*}
\abs{I_{2,\,t_0}\pair{t}{x}}\leq  
\frac{C}{\nu^{3/2}\tau^{1/4}}
\int_0^t \at{t-s}^{-3/4}s^{-1/4}\sqrt{M_{t_0}\at{s}}\dint{s}\leq\frac{C \sqrt{M_{t_0}\at{t}}}{\nu^{3/2}\tau^{1/4}},
\end{align*}
where we used the identity $\int_0^t \at{t-s}^{-3/4}s^{-1/4}\dint{s}=\int_0^1 \at{1-s}^{-3/4}s^{-1/4}\dint{s}<\infty$
and that
$M_{t_0}$ is an increasing function in $t$.
We therefore 
get
\begin{align*}
\sqrt{t}\norm{\varrho\pair{t_0+t}{\cdot}}_\infty\leq 
\frac{C\sqrt{\tau}}{\nu}+\frac{C\sqrt{t M_{t_0}\at{t}}}{\nu^{3/2}\tau^{1/4}},
\end{align*}
and since an analogous estimate holds for all $0\leq{s}\leq{t}$, we arrive 
at the estimate
\begin{align*}
M_{t_0}\at{t}\leq 
\frac{C\sqrt{\tau}}{\nu}+\frac{C\sqrt{t M_{t_0}\at{t}}}{\nu^{3/2}\tau^{1/4}}.
\end{align*}
This implies
\begin{align}
\label{Prop:UniformBounds.Eqn1}
\sqrt{t}\norm{\varrho\pair{t_0+t}{\cdot}}_\infty \leq M_{t_0}\at{t}\leq C\max\left\{
\frac{\sqrt{\tau}}{\nu},\,
\frac{t}{\nu^3\sqrt\tau}
\right\}, 
\end{align}
and for $t=\nu^2\tau$ we get
\begin{align*}
\norm{\varrho\pair{t_0+\nu^2\tau}{\cdot}}_\infty\leq \frac{C}{\nu^2}.
\end{align*} 
The claimed $\fspaceL^\infty$-estimate now follows since $t_0$ was arbitrary and $C$ independent of $t_0$.
\par
\ul{\emph{Bounds for energy and dissipation:}}
The energy balance \eqref{Eqn:EnergyBalance} implies
\begin{align*}
\int_{\nu^2\tau}^T\calD\at{t}\dint{t}&=
\tau\at{\calE\at{\nu^2\tau}-\calE\at{T}+\int_{\nu^2\tau}^T\si\at{t}\dot\ell\at{t}\dint{t}}\leq
\tau\at{\calE\at{\nu^2\tau}-\calE\at{T}+C},
\end{align*}
and from the definition of the energy \eqref{Eqn:DefEnergy}, the above $\fspaceL^\infty$-bounds,
and $H\at{x}\leq {C}\at{1+x^2}$ we infer that
\begin{align*}
\calE\at{\nu^2\tau}&\leq \nu^2 \int_\Rset \varrho\pair{\nu^2\tau}{x}\ln \varrho\pair{\nu^2\tau}{x}\dint{x}+C\int_\Rset \at{1+x^2} \varrho\pair{\nu^2\tau}{x}\dint{x}
\\&\leq \at{\nu^2\ln\frac{C}{\nu^2}}+C\leq {C}.
\end{align*}
In order to derive a lower bound for $\calE\at{T}$, we assume (without loss of generality)
that the global minimum of $H$ 
is normalized to $0$. The properties of $H$, see Assumption \ref{Ass:Potential}, then guarantee the existence
of constants $c>0$ as well as $\bar{x}_-<0$ and $\bar{x}_+>0$ such that
\begin{align*}
H\at{x}\geq c \left\{\begin{array}{lcl}
\at{x-\bar{x}_-}^2&\text{for}& x\leq{0},\\
\at{x-\bar{x}_+}^2&\text{for}& x\geq{0},
\end{array}\right.
\end{align*}
and hence we estimate
\begin{align*}
\int_\Rset \ga_0\at{x}\dint{x}\leq 
\int_{-\infty}^0\exp\at{-\frac{c\at{x-\bar{x}_-}^2}{\nu^2}}\dint{x}+
\int_0^{+\infty}\exp\at{-\frac{c\at{x-\bar{x}_+}^2}{\nu^2}}\dint{x}\leq{C}\nu,
\end{align*}
where $\ga_0\at{x}:=\exp\at{-H\at{x}/\nu^2}$. This implies, see 
also \eqref{Eqn:DefEnergyFunctional.A} and \eqref{Eqn:DefEnergyFunctional.B},
\begin{align*}
\calE\at{T}&=\calE_0\bat{\varrho\pair{T}{\cdot}}=\nu^2\int_\Rset\varrho\pair{t}{x}\ln\at{\frac{\varrho\pair{t}{x}}{\ga_0\at{x}}}\dint{x}
\\&\geq\nu^2\int_\Rset\varrho\pair{t}{x}\at{\ln\at{\frac{\varrho\pair{t}{x}}{\ga_0\at{x}}}+\frac{\ga_0\at{x}}{\varrho\pair{t}{x}}-1}\dint{x}-\nu^2\int_\Rset \ga_0\at{x}\dint{x}+\nu^2\int_\Rset \varrho\pair{t}{x}\dint{x}
\\&\geq 0-C\nu^3+\nu^2,
\end{align*}
where we used that $\ln{z}+1/z\geq1$ holds for all $z>0$. The desired $\fspaceL^1$-estimate for the dissipation follows immediately.
\end{proof}

\begin{lemma}[refined bounds for more regular initial data]
\label{Lem:ImpovedInitialData}
For initial data $\varrho_\ini\in\fspace L^\infty\at\Rset$ we have
\begin{align*}
\sup\limits_{t\in\ccinterval{0}{T}}\norm{\varrho\pair{t}{\cdot}}_\infty
\leq\frac{C}{\nu^2},\qquad
\int_{0}^T\calD\at{t}\dint{t}\leq C \tau 
\end{align*}
for some constant $C$ which depends only on $H$, $\bar\tau$,
$\bar\nu$, $\ell$, $\int_\Rset x^2\varrho_\ini\at{x}\dint{x}$, and $\nu^2\norm{\varrho_\ini}_\infty$.
\end{lemma}
\begin{proof}
In this case we can estimate 
\begin{align*}
I_{1,\,0}\pair{t}{x}\leq\norm{\varrho_\ini}_\infty  \int_\Rset K\pair{t}{x}\dint{x}=\norm{\varrho_\ini}_\infty.
\end{align*}
Moreover, for $0\leq{s}\leq {t}\leq \nu^2\tau$ we infer from \eqref{Prop:UniformBounds.Eqn1} that
\begin{align*}
\sqrt{s}\norm{\varrho\pair{s}{\cdot}}_\infty\leq M_0\at{s}\leq M_0\at{t}\leq \frac{C\sqrt{\tau}}{\nu}
\end{align*}
and this implies 
\begin{align*}
I_{2,\,0}\pair{t}{x}\leq
\frac{C}{\nu^{3/2}\tau^{1/4}}
\int_0^t \at{t-s}^{-3/4}\sqrt{\norm{\varrho\pair{s}{\cdot}}_\infty}\dint{s}
\leq
\frac{C}{\nu^{2}}.
\end{align*}
The claimed $\fspaceL^\infty$-estimate now follows from summing both inequalities (for $0\leq{t}\leq \nu^2\tau$) 
and using
Proposition \ref{Prop:UniformBounds} (for $\nu^2\tau\leq{t}\leq{T}$). Moreover, 
the $\fspaceL^1$-bound for the dissipation can be derived as in the proof of Proposition \ref{Prop:UniformBounds}.
\end{proof}

\section{Solutions to the limit model}\label{app:LimitModel}
%
%
\begin{figure}[ht!]%
\centering{%
\includegraphics[height=.3\textwidth, draft=\figdraft]{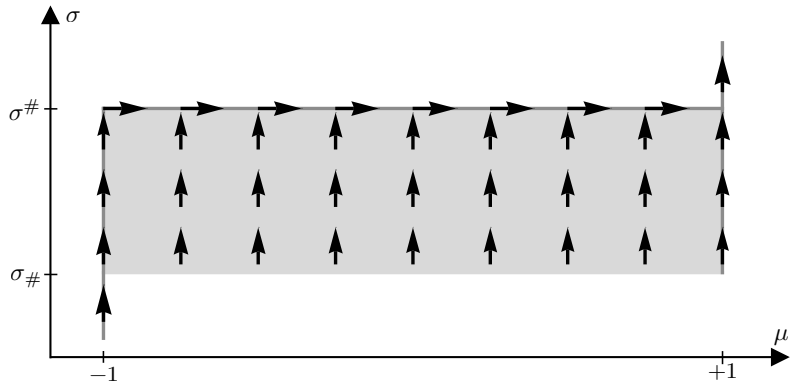}%
}%
\caption{Cartoon of the piecewise smooth vector field $\calV_+$ (arrows) on the set $\Xi$ (gray area) as used in the proof of
Proposition \ref{App:Prop:WellPosednessLimitModel}. For given initial data from $\Xi$, 
there exists a unique integral curve which is continuous and piecewise  continuous differentiable.
}%
\label{Fig:VectorField}
\end{figure}%
We prove that the initial value problem for the limit model has always a unique solution.
\begin{proposition}[well-posedness of the limit dynamics in the fast reaction regimes]
\label{App:Prop:WellPosednessLimitModel}
For any $\ell$ as in Assumption \ref{Ass:Constraint}, and any given initial data $\si\at{0}$ and $\mu\at{0}$ with
$\triple{\ell\at{0}}{\si\at{0}}{\mu\at{0}}\in\Om$, there exist two functions $\si$ and $\mu$ 
on $\ccinterval{0}{T}$ such that 
\begin{enumerate}
\item both $\si$ and $\mu$ are continuous, 
piecewise continuously differentiable, and attain the initial data,
\item the triple $\triple{\ell}{\si}{\mu}$ is a solution to the limit model in the sense of Definition \ref{Def:LimitModel}.
\end{enumerate}
Moreover, $\si$ and $\mu$ are uniquely determined by $\ell$,
$\si\at{0}$, and $\mu\at{0}$.
\end{proposition}
\begin{proof}
We observe that
\begin{align*}
\triple{\ell}{\si}{\mu}\in\Om\qquad\text{implies}\qquad \pair{\mu}{\si}\in\Xi,
\end{align*}
where $\Om$ has been introduced in \eqref{Eqn:LimDyn.Def1} and the closed set $\Xi$ is defined by
\begin{align*}
\Xi:=\{-1\}\times({-\infty},\,{\si^\#}] \; \cup\; \oointerval{-1}{+1}\times[{\si_\#},\,{\si^\#}]\;\cup\; \{+1\}\times\cointerval{\si_\#}{+\infty},
\end{align*}
see Figure \ref{Fig:VectorField} for an illustration. Moreover, for each point 
$\pair{\mu}{\si}\in\Xi$ there exists a unique value for $\ell$ such that $\ell=\calL\pair{\si}{\mu}$ with $\calL$ as in \eqref{Eqn:LimDyn.Def2}.
We proceed with discussing three special cases: If $\ell\at{t}=\ell\at{0}$ holds for all $t\in\ccinterval{0}{T}$, then the unique solution to the limit model is given by $\si\at{t}=\si\at{0}$ and
$\mu\at{t}=\mu\at{0}$. In the case of $\dot{\ell}\at{t}>0$ for all $t\in\oointerval{0}{T}$, we argue as follows.
By reparametrization of time, we can assume that $\dot{\ell}\at{t}=1$. The pointwise 
constraint $\ell\at{t}=\calL\bpair{\si\at{t}}{\mu\at{t}}$ then implies
that any solution to the limit model satisfies
\begin{align*}
\bpair{\dot\mu\at{t}}{\dot{\si}\at{t}}=\calV_+\bpair{\mu\at{t}}{{\si}\at{t}}
\end{align*}
for almost all $t\in\oointerval{0}{T}$, where the vector field
$\calV_+:\Xi\to\Rset^2$ is defined by
\begin{align*}
\calV_+\pair{\mu}{\si}=\left\{\begin{array}{lcl}
\pair{\D\Bat{\frac{X_+\at{\si^\#}-X_-\at{\si^\#}}{2}}^{-1}}{0}&&\text{for}-1\leq\mu<+1\text{ and }\si=\si^\#,\\
\pair{0}{\at{\D\frac{1-\mu}{2}X_-^\prime\at\si+\frac{1+\mu}{2}X_+^\prime\at\si}^{-1}}&&\text{for all other points in $\Xi$.}\\
\end{array}\right.
\end{align*}
Since $\calV_+$ is piecewise smooth on $\Xi$, there exists a unique continuous integral curve that 
emanates from the initial data and is moreover piecewise continuously differentiable. The arguments for the third case, that is $\dot{\ell}\at{t}<0$ for all $t\in\oointerval{0}{T}$, are entirely similar but involve a different vector field $\calV_-$. For arbitrary $\ell$, we introduce times
\begin{math}
0=T_0<T_1<\tdots<T_N=T
\end{math}
such that for any $i=1\tdots{N}$ and all $t\in\oointerval{T_{i-1}}{T_i}$ we have either
$\dot\ell\at{t}<0$, or $\dot\ell\at{t}=0$, or $\dot\ell\at{t}>0$.
The assertion now follows by iterating the arguments for the special cases. 
\end{proof}

\section*{Acknowledgement} 
The authors are grateful to the anonymous referees for their very helpful comments regarding the
exposition of the material and the readability of the paper, and to
Andr\'e Schlichting for pointing them to the intimate relation between Poincar\'e and Muckenhoupt constants.
They also wish to thank Henry Frohman and Alexander Mielke
for their valuable remarks on an earlier version of the paper.
Finally, the authors acknowledge the support by the 
Collaborative Research Center \emph{Singular Phenomena and Scaling in Mathematical Models} (DFG SFB 611, University of Bonn).
%
\newcommand{\etalchar}[1]{$^{#1}$}

\end{document}